\documentclass[twoside,11pt]{amsart}

\usepackage[cmtip,all]{xy}
\usepackage{xspace, enumerate, times, color}
\usepackage{amssymb, hyperref}

\usepackage{graphicx}
\input xypic
\input xy
\xyoption{all}
\def\sqr#1#2{{\vcenter{\hrule height.#2pt
        \hbox{\vrule width.#2pt height#1pt \kern#1pt
                \vrule width.#2pt}
        \hrule height.#2pt}}}

\numberwithin{equation}{section}

\newtheorem{theorem}{theorem}[section]
\newtheorem{lemma}[theorem]{Lemma}

\newtheorem{open Problem}[theorem]{Open Problem}

\newcommand{\be}{\begin{equation*}}
\newcommand{\ee}{\end{equation*}}
\newcommand{\bee}{\begin{equation}}
\newcommand{\eee}{\end{equation}}

\definecolor{lighterorange}{cmyk}{0,0.42,0.66,0.0}

\title[On odd deficient-perfect numbers with four distinct prime divisors]{On odd deficient-perfect numbers with four distinct prime divisors}
\author{Cui-Fang SUN and Zhao-Cheng HE}

\begin{document}

\date{2019-8-15\\E-mail:  cuifangsun@163.com,\;\; zhaochenghe6565@163.com}

\maketitle

\begin{abstract}
For a positive integer $n$, let $\sigma(n)$ denote the sum of the positive divisors of $n$. Let $d$ be a proper divisor of $n$. We call $n$ a deficient-perfect number if $\sigma(n)=2n-d$. In this paper, we show that the only odd deficient-perfect number with four distinct prime divisors is $3^{2}\cdot 7^{2}\cdot 11^{2}\cdot 13^{2}$.

\noindent{{\bf Keywords:}\hspace{2mm} deficient-perfect number, deficient divisor, multiplicative arithmetic function.}
\end{abstract}

\maketitle

\section{Introduction}
For a positive integer $n$, let arithmetic functions $\sigma(n)$ and $\omega(n)$ denote the sum of the positive divisors of $n$ and the number of distinct prime divisors of $n$, respectively. Let $d$ be a proper divisor of $n$. We call $n$ a deficient-perfect number with deficient divisor $d$ if $\sigma(n)=2n-d$. If $d=1$, then such a deficient-perfect number is called an almost perfect number. In 1978, Kishore \cite{K} proved that if $n$ is an odd almost
perfect number, then $\omega(n)\geqslant 6$. In 2013, Tang, Ren and Li \cite{TRL} determined all deficient-perfect numbers with at
most two distinct prime factors. In a similar vein, Tang and Feng \cite{TF} showed that no odd deficient-perfect number exists with three distinct prime factors. For related problems, see [1-3, 5, 6, 8].

In this paper, we obtain the following result:

\begin{theorem}\label{thm1} The only odd deficient-perfect number with four distinct prime divisors is $3^{2}\cdot 7^{2}\cdot 11^{2}\cdot 13^{2}$.
\end{theorem}

For convenience, let $(\frac{\cdot}{p})$ denote the Legendre symbol. Let $m$ be a positive integer and $a$ be any integer relatively prime to $m$. If $h$ is the least positive integer such that $a^{h}\equiv 1\pmod {m}$, then $h$ is called the order of $a$ modulo $m$, denoted by ${\rm{ord}}_{m}(a)$. We take $n=p_{1}^{\alpha_{1}}p_{2}^{\alpha_{2}}p_{3}^{\alpha_{3}}p_{4}^{\alpha_{4}}$ be an odd deficient-perfect number with deficient divisor $d=p_{1}^{\beta_{1}}p_{2}^{\beta_{2}}p_{3}^{\beta_{3}}p_{4}^{\beta_{4}}$, where $p_{1}<p_{2}<p_{3}<p_{4}$ are primes, $\alpha_{i}$'s are positive integers and $\beta_{i}$'s are nonnegative integers with $\beta_{i}\leqslant \alpha_{i}$ and $\sum\limits_{i=1}^{4}\beta_{i}<\sum\limits_{i=1}^{4}\alpha_{i}$.
By \cite{K} and \cite{TF}, we have $d>1$ and $\alpha_{i}$'s are all even. Let $D=p_{1}^{\alpha_{1}-\beta_{1}}p_{2}^{\alpha_{2}-\beta_{2}}p_{3}^{\alpha_{3}-\beta_{3}}p_{4}^{\alpha_{4}-\beta_{4}}$.
Then
\begin{equation}\label{Eq1}
\frac{p_{1}^{\alpha_{1}+1}-1}{p_{1}-1}\cdot \frac{p_{2}^{\alpha_{2}+1}-1}{p_{2}-1}\cdot\frac{p_{3}^{\alpha_{3}+1}-1}{p_{3}-1}\cdot \frac{p_{4}^{\alpha_{4}+1}-1}{p_{4}-1}=\sigma(n)=2n-d=(2D-1)d
\end{equation}
or equally
$$2=\frac{\sigma(n)}{n}+\frac{d}{n}=\frac{\sigma(n)}{n}+\frac{1}{D}.$$

\section{The case of $p_{2}=5$}

In this section, we consider the case of $n=p_{1}^{\alpha_{1}}p_{2}^{\alpha_{2}}p_{3}^{\alpha_{3}}p_{4}^{\alpha_{4}}$ with $p_{1}=3$ and $p_{2}=5$.

\begin{lemma}\label{lem2.1}
There is no odd deficient-perfect number of the form $n=3^{\alpha_{1}}5^{\alpha_{2}}11^{\alpha_{3}}p_{4}^{\alpha_{4}}$.
\end{lemma}
\begin{proof}
Assume that $n=3^{\alpha_{1}}5^{\alpha_{2}}11^{\alpha_{3}}p_{4}^{\alpha_{4}}$ is an odd deficient-perfect number with deficient divisor $d=3^{\beta_{1}}5^{\beta_{2}}11^{\beta_{3}}p_{4}^{\beta_{4}}$. If $\alpha_{1}\geqslant 4$, then
$$2=\frac{\sigma(n)}{n}+\frac{d}{n}>\frac{3^{5}-1}{2\cdot 3^{4}}\cdot \frac{5^{3}-1}{4\cdot 5^{2}}\cdot \frac{11^{3}-1}{10\cdot 11^{2}}>2,$$
which is clearly false. Thus $\alpha_{1}=2$. If $p_{4}\leqslant 61$, then
$$2=\frac{\sigma(n)}{n}+\frac{d}{n}>\frac{3^{3}-1}{2\cdot 3^{2}}\cdot \frac{5^{3}-1}{4\cdot 5^{2}}\cdot \frac{11^{3}-1}{10\cdot 11^{2}}
\cdot \frac{61^{3}-1}{60\cdot 61^{2}}>2,$$
which is a contradiction. Thus $p_{4}\geqslant 67$. By (\ref{Eq1}), we have
\begin{equation}\label{2.1}
13\cdot \frac{5^{\alpha_{2}+1}-1}{4}\cdot \frac{11^{\alpha_{3}+1}-1}{10}\cdot \frac{p_{4}^{\alpha_{4}+1}-1}{p_{4}-1}=2\cdot 3^{2}\cdot 5^{\alpha_{2}} 11^{\alpha_{3}}p_{4}^{\alpha_{4}}-3^{\beta_{1}}5^{\beta_{2}}11^{\beta_{3}} p_{4}^{\beta_{4}}
\end{equation}
and $13\mid (2D-1)$. Thus $D=26k+7$ for some positive integer $k$.

If $k=1$, then $D=33$. If $\alpha_{2}\geqslant 4$, then
$$2=\frac{\sigma(n)}{n}+\frac{d}{n}>\frac{3^{3}-1}{2\cdot 3^{2}}\cdot\frac{5^{5}-1}{4\cdot 5^{4}}\cdot \frac{11^{3}-1}{10\cdot 11^{2}}+\frac{1}{33}>2,$$
which is clearly false. Thus $\alpha_{2}=2$ and $p_{4}=31$, which contradicts with $p_{4}\geqslant 67$.

If $k\in\{2, 3, 4, 7, 8, 12, 13, 14, 16, 18, 20, 21\}$, then $p_{4}<61$, this is impossible.

If $k\in\{5, 6, 10, 11, 15, 17, 19, 22\}$, then $p_{4}\in \{137, 163, 89, 293, 397, 449, 167, 193\}$. Noting that $\rm{ord}_{25}(11)=5$ and ${\rm{ord}}_{5}(p_{4})$ are even, we have $5\mid (\alpha_{3}+1)$ and $(11^{5}-1)\mid (11^{\alpha_{3}+1}-1)$. However, $3221\mid (11^{5}-1)$, a contradiction.

If $k=9$, then $D=p_{4}=241$. Noting that $\rm{ord}_{11}(5)=5$ and $\rm{ord}_{11}(241)=2$, we have $5\mid (\alpha_{3}+1)$ and $(5^{5}-1)\mid (5^{\alpha_{3}+1}-1)$. However, $71\mid (5^{5}-1)$, a contradiction.

If $k\geqslant 23$, then $D\geqslant 605$. If $p_{4}\geqslant 167$, then
$$2=\frac{\sigma(n)}{n}+\frac{d}{n}<\frac{3^{3}-1}{2\cdot 3^{2}}\cdot\frac{5}{4}\cdot\frac{11}{10}\cdot\frac{167}{166}+\frac{1}{605}<2,$$
which is false. Thus $67\leqslant p_{4}\leqslant 163$. If $67\leqslant p_{4}\leqslant 113$, then $\alpha_{2}=2$ and $(13\cdot 31)\mid (2D-1)$. Otherwise, if $\alpha_{2}\geqslant 4$, then
$$2=\frac{\sigma(n)}{n}+\frac{d}{n}>\frac{3^{3}-1}{2\cdot 3^{2}}\cdot \frac{5^{5}-1}{4\cdot 5^{4}}\cdot \frac{11^{3}-1}{10\cdot 11^{2}}\cdot \frac{113^{3}-1}{112\cdot 113^{2}}>2,$$
which is impossible. Now we divide into the following four cases according to $p_{4}$.

{\bf Case 1.} $p_{4}=67$. If $\alpha_{3}\geqslant 4$, then
$$2=\frac{\sigma(n)}{n}+\frac{d}{n}>\frac{3^{3}-1}{2\cdot 3^{2}}\cdot \frac{5^{3}-1}{4\cdot 5^{2}}\cdot \frac{11^{5}-1}{10\cdot 11^{4}}\cdot \frac{67^{3}-1}{66\cdot 67^{2}}>2,$$
which is impossible. Thus $\alpha_{3}=2$ and $(13\cdot 31\cdot 7\cdot 19)\mid (2D-1)$. Thus $D>26800$ and
$$2=\frac{\sigma(n)}{n}+\frac{d}{n}<\frac{3^{3}-1}{2\cdot 3^{2}}\cdot \frac{5^{3}-1}{4\cdot 5^{2}}\cdot \frac{11^{3}-1}{10\cdot 11^{2}}\cdot \frac{67}{66}+\frac{1}{D}<2,$$
which is absurd.

{\bf Case 2.} $p_{4}=71$. Since $\rm{ord}_{3}(11)=\rm{ord}_{3}(71)=2$, we have $\beta_{1}=0$ and $9\mid D$. Thus $D>615$ and
$$2=\frac{\sigma(n)}{n}+\frac{d}{n}<\frac{3^{3}-1}{2\cdot 3^{2}}\cdot \frac{5^{3}-1}{4\cdot 5^{2}}\cdot \frac{11}{10}\cdot \frac{71}{70}+\frac{1}{D}<2,$$
which is clearly false.

{\bf Case 3.} $73 \leqslant p_{4}\leqslant 113$. By $\alpha_{2}=2$, we have
$$2=\frac{\sigma(n)}{n}+\frac{d}{n}<\frac{3^{3}-1}{2\cdot 3^{2}}\cdot \frac{5^{3}-1}{4\cdot 5^{2}}\cdot \frac{11}{10}\cdot \frac{73}{72}+\frac{1}{605}<2,$$
which is impossible.

{\bf Case 4.} $127 \leqslant p_{4}\leqslant 163$. If $\alpha_{2}=2$, then
$$2=\frac{\sigma(n)}{n}+\frac{d}{n}<\frac{3^{3}-1}{2\cdot 3^{2}}\cdot \frac{5^{3}-1}{4\cdot 5^{2}}\cdot \frac{11}{10}\cdot \frac{127}{126}+\frac{1}{605}<2,$$
which is absurd. Thus $\alpha_{2}\geqslant 4$.

{\bf Subcase 4.1 } $p_{4}=127$. If $\alpha_{2}\geqslant 6$, then
$$2=\frac{\sigma(n)}{n}+\frac{d}{n}>\frac{3^{3}-1}{2\cdot 3^{2}}\cdot \frac{5^{7}-1}{4\cdot 5^{6}}\cdot \frac{11^{3}-1}{10\cdot 11^{2}}\cdot \frac{127^{3}-1}{126\cdot 127^{2}}>2,$$
which is clearly false. Thus $\alpha_{2}=4$. If $\alpha_{3}\geqslant 4$, then
$$2=\frac{\sigma(n)}{n}+\frac{d}{n}>\frac{3^{3}-1}{2\cdot 3^{2}}\cdot \frac{5^{5}-1}{4\cdot 5^{4}}\cdot \frac{11^{5}-1}{10\cdot 11^{4}}\cdot \frac{127^{3}-1}{126\cdot 127^{2}}>2,$$
which is absurd.  Thus $\alpha_{3}=2$. Since $\rm{ord}_{5}(127)=4, \rm{ord}_{11}(127)=10$, we have $\beta_{2}=\beta_{4}=0, \beta_{3}=1, D\geqslant 5^{4}\cdot 11\cdot 127^{2}$ and
$$2=\frac{\sigma(n)}{n}+\frac{d}{n}<\frac{3^{3}-1}{2\cdot 3^{2}}\cdot \frac{5^{5}-1}{4\cdot 5^{4}}\cdot \frac{11^{3}-1}{10\cdot 11^{2}}\cdot \frac{127}{126}+\frac{1}{5^{4}\cdot 11\cdot 127^{2}}<2,$$
which is a contradiction.

{\bf Subcase 4.2 } $p_{4}\in\{131, 137\}$. If $\alpha_{3}\geqslant 4$, then
$$2=\frac{\sigma(n)}{n}+\frac{d}{n}>\frac{3^{3}-1}{2\cdot 3^{2}}\cdot \frac{5^{5}-1}{4\cdot 5^{4}}\cdot \frac{11^{5}-1}{10\cdot 11^{4}}\cdot \frac{137^{3}-1}{136\cdot 137^{2}}>2,$$
which is absurd. Thus $\alpha_{3}=2$ and $(7\cdot 19\cdot 13)\mid (2D-1)$. Since $\rm{ord}_{3}(5)=\rm{ord}_{3}(131)=\rm{ord}_{3}(137)=2$, we have $\beta_{1}=0$ and $9\mid D$. Thus $D>8715$ and
$$2=\frac{\sigma(n)}{n}+\frac{d}{n}<\frac{3^{3}-1}{2\cdot 3^{2}}\cdot \frac{5}{4}\cdot \frac{11^{3}-1}{10\cdot 11^{2}}\cdot \frac{131}{130}+\frac{1}{8715}<2,$$
which is a contradiction.

{\bf Subcase 4.3 } $p_{4}=139$. If $\alpha_{2}=4$, then $(71\cdot 13)\mid (2D-1)$. Thus $D>7303$ and
$$2=\frac{\sigma(n)}{n}+\frac{d}{n}<\frac{3^{3}-1}{2\cdot 3^{2}}\cdot \frac{5^{5}-1}{4\cdot 5^{4}}\cdot \frac{11}{10}\cdot \frac{139}{138}+\frac{1}{7303}<2,$$
which is absurd. Thus $\alpha_{2}\geqslant 6$. If $\alpha_{3}\geqslant 4$, then
$$2=\frac{\sigma(n)}{n}+\frac{d}{n}>\frac{3^{3}-1}{2\cdot 3^{2}}\cdot \frac{5^{7}-1}{4\cdot 5^{6}}\cdot \frac{11^{5}-1}{10\cdot 11^{4}}\cdot \frac{139^{3}-1}{138\cdot 139^{2}}>2,$$
which is a contradiction. Thus $\alpha_{3}=2$ and $(7\cdot 19\cdot 13)\mid (2D-1)$. Thus $D>1001$ and
$$2=\frac{\sigma(n)}{n}+\frac{d}{n}<\frac{3^{3}-1}{2\cdot 3^{2}}\cdot \frac{5}{4}\cdot \frac{11^{3}-1}{10\cdot 11^{2}}\cdot \frac{139}{138}+\frac{1}{1001}<2,$$
which is impossible.

{\bf Subcase 4.4 } $p_{4}=149$. Since $\rm{ord}_{3}(5)=\rm{ord}_{3}(11)=\rm{ord}_{3}(149)=2$, we have $\beta_{1}=0$ and $9\mid D$. Thus $D\geqslant 1125$.
If $D\geqslant 2133$, then
$$2=\frac{\sigma(n)}{n}+\frac{d}{n}<\frac{3^{3}-1}{2\cdot 3^{2}}\cdot \frac{5}{4}\cdot \frac{11}{10}\cdot \frac{149}{148}+\frac{1}{2133}<2,$$
which is clearly false. Thus $D=1125$. Since $\rm{ord}_{5}(149)=2, \rm{ord}_{25}(11)=5$, we have $5\mid(\alpha_{3}+1)$ and $(11^{5}-1)\mid(11^{\alpha_{3}+1}-1)$. However, $3221\mid(11^{5}-1)$, a contradiction.

{\bf Subcase 4.5 } $p_{4}=151$. If $D\geqslant 1543$, then
$$2=\frac{\sigma(n)}{n}+\frac{d}{n}<\frac{3^{3}-1}{2\cdot 3^{2}}\cdot \frac{5}{4}\cdot \frac{11}{10}\cdot \frac{151}{150}+\frac{1}{1543}<2,$$
which is clearly false. Thus $D\in\{605, 1125, 1359\}$ and $\beta_{4}\geqslant 1$. Noting that $\rm{ord}_{151}(5)=\rm{ord}_{151}(11)=75$, we have $75\mid(\alpha_{2}+1)$ and $(5^{75}-1)\mid(5^{\alpha_{2}+1}-1)$ or $75\mid(\alpha_{3}+1)$ and $(11^{75}-1)\mid(11^{\alpha_{3}+1}-1)$. However, $71\mid(5^{75}-1)$ and $7\mid(11^{75}-1)$,
 a contradiction.

{\bf Subcase 4.6 } $p_{4}\in\{157, 163\}$. If $D\geqslant 865$, then
$$2=\frac{\sigma(n)}{n}+\frac{d}{n}<\frac{3^{3}-1}{2\cdot 3^{2}}\cdot \frac{5}{4}\cdot \frac{11}{10}\cdot \frac{157}{156}+\frac{1}{865}<2,$$
which is false. Thus $D=605$ and $\beta_{2}\geqslant 1$. Since $\rm{ord}_{25}(11)=5$ and ${\rm{ord}}_{5}(p_{4})=4$, we have $5\mid(\alpha_{3}+1)$ and $(11^{5}-1)\mid(11^{\alpha_{3}+1}-1)$. However, $3221\mid(11^{5}-1)$,  a contradiction.

This completes the proof of Lemma \ref{lem2.1}.
\end{proof}

\begin{lemma}\label{lem2.2}
There is no odd deficient-perfect number of the form $n=3^{\alpha_{1}}5^{\alpha_{2}}13^{\alpha_{3}}p_{4}^{\alpha_{4}}$.
\end{lemma}
\begin{proof}
Assume that $n=3^{\alpha_{1}}5^{\alpha_{2}}13^{\alpha_{3}}p_{4}^{\alpha_{4}}$ is an odd deficient-perfect number with deficient divisor $d=3^{\beta_{1}}5^{\beta_{2}}13^{\beta_{3}}p_{4}^{\beta_{4}}$. If $\alpha_{1}\geqslant 4$, then
$$2=\frac{\sigma(n)}{n}+\frac{d}{n}>\frac{3^{5}-1}{2\cdot 3^{4}}\cdot \frac{5^{3}-1}{4\cdot 5^{2}}\cdot \frac{13^{3}-1}{12\cdot 13^{2}}>2,$$
which is false. Thus $\alpha_{1}=2$. If $p_{4}\leqslant 31$, then
$$2=\frac{\sigma(n)}{n}+\frac{d}{n}>\frac{3^{3}-1}{2\cdot 3^{2}}\cdot \frac{5^{3}-1}{4\cdot 5^{2}}\cdot \frac{13^{3}-1}{12\cdot 13^{2}} \cdot \frac{31^{3}-1}{30\cdot 31^{2}}>2,$$
which is impossible. Thus $p_{4}\geqslant 37$. If $D\leqslant 15$, then
$$2=\frac{\sigma(n)}{n}+\frac{d}{n}>\frac{3^{3}-1}{2\cdot 3^{2}}\cdot \frac{5^{3}-1}{4\cdot 5^{2}}\cdot \frac{13^{3}-1}{12\cdot 13^{2}}+\frac{1}{15}>2,$$
which is clearly false. Thus $D\geqslant 25$. By (\ref{Eq1}), we have
\begin{equation}\label{2.2}
13\cdot \frac{5^{\alpha_{2}+1}-1}{4}\cdot \frac{13^{\alpha_{3}+1}-1}{12}\cdot \frac{p_{4}^{\alpha_{4}+1}-1}{p_{4}-1}=2\cdot 3^{2}\cdot 5^{\alpha_{2}} 13^{\alpha_{3}}p_{4}^{\alpha_{4}}-3^{\beta_{1}}5^{\beta_{2}}13^{\beta_{3}} p_{4}^{\beta_{4}}.
\end{equation}
Now we divide into the following eight cases according to $D$.

{\bf Case 1.} $D=25$. If $p_{4}\geqslant 499$, then
$$2=\frac{\sigma(n)}{n}+\frac{d}{n}<\frac{3^{3}-1}{2\cdot 3^{2}}\cdot \frac{5}{4}\cdot \frac{13}{12}\cdot \frac{499}{498}+\frac{1}{25}<2,$$
which is false. If $p_{4}\leqslant 89$, then
$$2=\frac{\sigma(n)}{n}+\frac{d}{n}\geqslant\frac{3^{3}-1}{2\cdot 3^{2}}\cdot \frac{5^{3}-1}{4\cdot 5^{2}}\cdot \frac{13^{3}-1}{12\cdot 13^{2}}\cdot \frac{89^{3}-1}{88\cdot 89^{2}}+\frac{1}{25}>2,$$
which is also false. Thus $97\leqslant p_{4}\leqslant 491$. It implies that $\alpha_{2}\geqslant 8, \alpha_{3}\geqslant 4$ and
$$2=\frac{\sigma(n)}{n}+\frac{d}{n}\geqslant\frac{3^{3}-1}{2\cdot 3^{2}}\cdot \frac{5^{9}-1}{4\cdot 5^{8}}\cdot \frac{13^{5}-1}{12\cdot 13^{4}}\cdot \frac{491^{3}-1}{490\cdot 491^{2}}+\frac{1}{25}>2,$$
which is impossible.

{\bf Case 2.} $D\in\{37, 41, 43, 47\}$. Then $p_{4}=D$ and
$$2=\frac{\sigma(n)}{n}+\frac{d}{n}\geqslant\frac{3^{3}-1}{2\cdot 3^{2}}\cdot \frac{5^{3}-1}{4\cdot 5^{2}}\cdot \frac{13^{3}-1}{12\cdot 13^{2}}\cdot \frac{47^{3}-1}{46\cdot 47^{2}}+\frac{1}{47}>2,$$
which is false.

{\bf Case 3.} $D=39$. If $p_{4}\geqslant 109$, then
$$2=\frac{\sigma(n)}{n}+\frac{d}{n}<\frac{3^{3}-1}{2\cdot 3^{2}}\cdot \frac{5}{4}\cdot \frac{13}{12}\cdot \frac{109}{108}+\frac{1}{39}<2,$$
which is a contradiction. If $p_{4}\leqslant 53$, then
$$2=\frac{\sigma(n)}{n}+\frac{d}{n}\geqslant\frac{3^{3}-1}{2\cdot 3^{2}}\cdot \frac{5^{3}-1}{4\cdot 5^{2}}\cdot \frac{13^{3}-1}{12\cdot 13^{2}}\cdot \frac{53^{3}-1}{52\cdot 53^{2}}+\frac{1}{39}>2,$$
which is also a contradiction. Thus $59\leqslant p_{4}\leqslant 107$ and $\alpha_{2}\geqslant 4$. Since $\rm{ord}_{5}(13)=4$, $\rm{ord}_{7}(5)=6$ and $\rm{ord}_{7}(13)=2$, we have $p_{4}\leqslant 71$ and
$$2=\frac{\sigma(n)}{n}+\frac{d}{n}\geqslant\frac{3^{3}-1}{2\cdot 3^{2}}\cdot \frac{5^{5}-1}{4\cdot 5^{4}}\cdot \frac{13^{3}-1}{12\cdot 13^{2}}\cdot \frac{71^{3}-1}{70\cdot 71^{2}}+\frac{1}{39}>2,$$
which is impossible.

{\bf Case 4.} $D=45$. If $p_{4}\geqslant 97$, then
$$2=\frac{\sigma(n)}{n}+\frac{d}{n}<\frac{3^{3}-1}{2\cdot 3^{2}}\cdot \frac{5}{4}\cdot \frac{13}{12}\cdot \frac{97}{96}+\frac{1}{45}<2,$$
which is impossible. If $p_{4}\leqslant 47$, then
$$2=\frac{\sigma(n)}{n}+\frac{d}{n}\geqslant\frac{3^{3}-1}{2\cdot 3^{2}}\cdot \frac{5^{3}-1}{4\cdot 5^{2}}\cdot \frac{13^{3}-1}{12\cdot 13^{2}}\cdot \frac{47^{3}-1}{46\cdot 47^{2}}+\frac{1}{45}>2,$$
which is false. Thus $53\leqslant p_{4}\leqslant 89$ and $\alpha_{2}\geqslant 4$. Since $\rm{ord}_{5}(13)=4$, we have $p_{4}\leqslant 71$ and
$$2=\frac{\sigma(n)}{n}+\frac{d}{n}\geqslant\frac{3^{3}-1}{2\cdot 3^{2}}\cdot \frac{5^{5}-1}{4\cdot 5^{4}}\cdot \frac{13^{3}-1}{12\cdot 13^{2}}\cdot \frac{71^{3}-1}{70\cdot 71^{2}}+\frac{1}{45}>2,$$
which is impossible.

{\bf Case 5.} $D\in\{53, 59, 67, 73, 79, 83, 89, 97\}$.  Then $p_{4}=D$ and $\beta_{2}\geqslant 2$.
Noting that $\rm{ord}_{5}(13)=4$ and ${\rm{ord}}_{5}(p_{4})$ are all even, we deduce that the equality (\ref{2.2}) can not hold.

{\bf Case 6.} $D\in\{61, 71\}$. Then $p_{4}=D$ and $\beta_{2}\geqslant 2$.
Since $\rm{ord}_{5}(13)=4$ and ${\rm{ord}}_{25}(p_{4})=5$, we have $5\mid (\alpha_{4}+1)$ and $(p_{4}^{5}-1)\mid (p_{4}^{\alpha_{4}+1}-1)$. However, $131\mid (61^{5}-1)$ and $11\mid (71^{5}-1)$, a contradiction.

{\bf Case 7.} $D\in\{65, 75\}$. If $p_{4}\geqslant 71$, then
$$2=\frac{\sigma(n)}{n}+\frac{d}{n}<\frac{3^{3}-1}{2\cdot 3^{2}}\cdot \frac{5}{4}\cdot \frac{13}{12}\cdot \frac{71}{70}+\frac{1}{65}<2,$$
which is false. Thus $p_{4}\leqslant 67$ and $\alpha_{2}\geqslant 6$.  Since $\rm{ord}_{5}(13)=4$, we have $p_{4}\leqslant 61$ and
$$2=\frac{\sigma(n)}{n}+\frac{d}{n}\geqslant \frac{3^{3}-1}{2\cdot 3^{2}}\cdot \frac{5^{7}-1}{4\cdot 5^{6}}\cdot \frac{13^{3}-1}{12\cdot 13^{2}}\cdot \frac{61^{3}-1}{60\cdot 61^{2}}+\frac{1}{75}>2,$$
which is impossible.

{\bf Case 8.} $D\geqslant 101$. If $p_{4}\geqslant 59$, then
$$2=\frac{\sigma(n)}{n}+\frac{d}{n}<\frac{3^{3}-1}{2\cdot 3^{2}}\cdot \frac{5}{4}\cdot \frac{13}{12}\cdot \frac{59}{58}+\frac{1}{101}<2,$$
which is false. Thus $p_{4}\leqslant 53$.

{\bf Subcase 8.1 } $p_{4}\in\{37, 41, 43\}$. If $\alpha_{2}\geqslant 4$, then
$$2=\frac{\sigma(n)}{n}+\frac{d}{n}>\frac{3^{3}-1}{2\cdot 3^{2}}\cdot \frac{5^{5}-1}{4\cdot 5^{4}}\cdot \frac{13^{3}-1}{12\cdot 13^{2}}\cdot \frac{43^{3}-1}{42\cdot 43^{2}}>2,$$
which is absurd. Thus $\alpha_{2}=2$ and $31\mid (2D-1)$. Thus $D>175$ and
$$2=\frac{\sigma(n)}{n}+\frac{d}{n}<\frac{3^{3}-1}{2\cdot 3^{2}}\cdot \frac{5^{3}-1}{4\cdot 5^{2}}\cdot \frac{13}{12}\cdot \frac{37}{36}+\frac{1}{175}<2,$$
which is impossible.

{\bf Subcase 8.2 } $p_{4}=47$. Since $\rm{ord}_{5}(13)=\rm{ord}_{5}(47)=2$ and $\rm{ord}_{47}(5)=\rm{ord}_{47}(13)=46$, we have $\beta_{2}=\beta_{4}=0, D\geqslant 5^{2}\cdot 47^{2}$ and
$$2=\frac{\sigma(n)}{n}+\frac{d}{n}<\frac{3^{3}-1}{2\cdot 3^{2}}\cdot \frac{5}{4}\cdot \frac{13}{12}\cdot \frac{47}{46}+\frac{1}{5^{2}\cdot 47^{2}}<2,$$
which is impossible.

{\bf Subcase 8.3 } $p_{4}=53$. If $D\geqslant 159$, then
$$2=\frac{\sigma(n)}{n}+\frac{d}{n}<\frac{3^{3}-1}{2\cdot 3^{2}}\cdot \frac{5}{4}\cdot \frac{13}{12}\cdot \frac{53}{52}+\frac{1}{159}<2,$$
which is false. Thus $D\in\{117, 125\}$. Since $\rm{ord}_{5}(13)=\rm{ord}_{5}(53)=4$,  we deduce that the equality (\ref{2.2}) can not hold.

This completes the proof of Lemma \ref{lem2.2}.
\end{proof}

\begin{lemma}\label{lem2.3}
There is no odd deficient-perfect number of the form $n=3^{\alpha_{1}}5^{\alpha_{2}}17^{\alpha_{3}}p_{4}^{\alpha_{4}}$.
\end{lemma}
\begin{proof}
Assume that $n=3^{\alpha_{1}}5^{\alpha_{2}}17^{\alpha_{3}}p_{4}^{\alpha_{4}}$ is an odd deficient-perfect number with deficient divisor $d=3^{\beta_{1}}5^{\beta_{2}}17^{\beta_{3}}p_{4}^{\beta_{4}}$. If $p_{4}=19$, then
$$2=\frac{\sigma(n)}{n}+\frac{d}{n}>\frac{3^{3}-1}{2\cdot 3^{2}}\cdot \frac{5^{3}-1}{4\cdot 5^{2}}\cdot \frac{17^{3}-1}{16\cdot 17^{2}}\cdot \frac{19^{3}-1}{18\cdot 19^{2}}>2,$$
which is false. Thus $p_{4}\geqslant 23$. If $\alpha_{1}=2$, then $13\mid(2D-1)$. Thus $D\geqslant 59$. If $p_{4}\geqslant 31$, then
$$2=\frac{\sigma(n)}{n}+\frac{d}{n}<\frac{3^{3}-1}{2\cdot 3^{2}}\cdot \frac{5}{4}\cdot \frac{17}{16}\cdot \frac{31}{30}+\frac{1}{59}<2,$$
which is false. Thus $p_{4}\in\{23, 29\}$ and $d=1$, a contradiction. Thus $\alpha_{1}\geqslant 4$. If $p_{4}\leqslant 61$, then
$$2=\frac{\sigma(n)}{n}+\frac{d}{n}>\frac{3^{5}-1}{2\cdot 3^{4}}\cdot \frac{5^{3}-1}{4\cdot 5^{2}}\cdot \frac{17^{3}-1}{16\cdot 17^{2}}\cdot \frac{61^{3}-1}{60\cdot 61^{2}}>2,$$
which is clearly false. Thus $p_{4}\geqslant 67$. If $D\leqslant 37$, then $\alpha_{1}\geqslant 6$ and
$$2=\frac{\sigma(n)}{n}+\frac{d}{n}>\frac{3^{7}-1}{2\cdot 3^{6}}\cdot \frac{5^{3}-1}{4\cdot 5^{2}}\cdot \frac{17^{3}-1}{16\cdot 17^{2}}+\frac{1}{37}>2,$$
which is absurd. Thus $D\geqslant 45$. By (\ref{Eq1}), we have
\begin{equation}\label{2.3}
\frac{3^{\alpha_{1}+1}-1}{2}\cdot \frac{5^{\alpha_{2}+1}-1}{4}\cdot \frac{17^{\alpha_{3}+1}-1}{16}\cdot \frac{p_{4}^{\alpha_{4}+1}-1}{p_{4}-1}=2\cdot 3^{\alpha_{1}}5^{\alpha_{2}} 17^{\alpha_{3}}p_{4}^{\alpha_{4}}-3^{\beta_{1}}5^{\beta_{2}}17^{\beta_{3}} p_{4}^{\beta_{4}}.
\end{equation}
Now we divide into the following nine cases according to $D$.

{\bf Case 1.} $D\in\{45, 51, 67, 71, 73, 75, 79, 81, 85\}$. By (\ref{2.3}), we have $\alpha_{1}\geqslant 6, \alpha_{2}\geqslant 4$ and
$$2=\frac{\sigma(n)}{n}+\frac{d}{n}>\frac{3^{7}-1}{2\cdot 3^{6}}\cdot \frac{5^{5}-1}{4\cdot 5^{4}}\cdot \frac{17^{3}-1}{16\cdot 17^{2}}+\frac{1}{85}>2,$$
which is impossible.

{\bf Case 2.} $D\in\{83, 89, 97, 101, 103, 107, 109, 113\}$. Then $p_{4}=D, \alpha_{1}\geqslant 6$ and
$$2=\frac{\sigma(n)}{n}+\frac{d}{n}\geqslant\frac{3^{7}-1}{2\cdot 3^{6}}\cdot \frac{5^{3}-1}{4\cdot 5^{2}}\cdot \frac{17^{3}-1}{16\cdot 17^{2}}\cdot \frac{113^{3}-1}{112\cdot 113^{2}}+\frac{1}{113}>2,$$
which is impossible.

{\bf Case 3.} $D\in\{125, 135, 153, 243, 255\}$. Noting that $\rm{ord}_{3}(5)=\rm{ord}_{3}(17)=2, \rm{ord}_{5}(3)=\rm{ord}_{5}(17)=4$ and $\rm{ord}_{17}(3)=\rm{ord}_{17}(5)=16$, we have $p_{4}\equiv 1\pmod {255}$. Thus $\alpha_{1}\geqslant 12$, $\alpha_{2}\geqslant 8, \alpha_{3}\geqslant 4$ and  $p_{4}\geqslant 1021$.

If $D=125$, then
$$2=\frac{\sigma(n)}{n}+\frac{d}{n}>\frac{3^{13}-1}{2\cdot 3^{12}}\cdot \frac{5^{9}-1}{4\cdot 5^{8}}\cdot \frac{17^{5}-1}{16\cdot 17^{4}}+\frac{1}{125}>2,$$
which is a contradiction.

If $D=135$, then $p_{4}\leqslant 4591$. Otherwise, if $p_{4}\geqslant 4919$, then
$$2=\frac{\sigma(n)}{n}+\frac{d}{n}<\frac{3}{2}\cdot \frac{5}{4}\cdot \frac{17}{16}\cdot \frac{4919}{4918}+\frac{1}{135}<2,$$
which is a contradiction. However,
$$2=\frac{\sigma(n)}{n}+\frac{d}{n}\geqslant\frac{3^{13}-1}{2\cdot 3^{12}}\cdot \frac{5^{9}-1}{4\cdot 5^{8}}\cdot \frac{17^{5}-1}{16\cdot 17^{4}}\cdot \frac{4591^{3}-1}{4590\cdot 4591^{2}}+\frac{1}{135}>2,$$
which is also a contradiction.

If $D=153$, then $p_{4}\leqslant 1531$. Otherwise, if $p_{4}\geqslant 1567$, then
$$2=\frac{\sigma(n)}{n}+\frac{d}{n}<\frac{3}{2}\cdot \frac{5}{4}\cdot \frac{17}{16}\cdot \frac{1567}{1566}+\frac{1}{153}<2,$$
which is a contradiction. However,
$$2=\frac{\sigma(n)}{n}+\frac{d}{n}\geqslant\frac{3^{13}-1}{2\cdot 3^{12}}\cdot \frac{5^{9}-1}{4\cdot 5^{8}}\cdot \frac{17^{5}-1}{16\cdot 17^{4}}\cdot \frac{1531^{3}-1}{1530\cdot 1531^{2}}+\frac{1}{153}>2,$$
which is also a contradiction.

If $D\in\{243, 255\}$, then
$$2=\frac{\sigma(n)}{n}+\frac{d}{n}<\frac{3}{2}\cdot \frac{5}{4}\cdot \frac{17}{16}\cdot \frac{1021}{1020}+\frac{1}{243}<2,$$
which is a contradiction.

{\bf Case 4.} $127\leqslant D=p_{4}\leqslant 359$. Then $\alpha_{1}\geqslant 8, \alpha_{2}\geqslant 6$  and
$$2=\frac{\sigma(n)}{n}+\frac{d}{n}\geqslant\frac{3^{9}-1}{2\cdot 3^{8}}\cdot \frac{5^{7}-1}{4\cdot 5^{6}}\cdot \frac{17^{3}-1}{16\cdot 17^{2}}\cdot \frac{359^{3}-1}{358\cdot 359^{2}}+\frac{1}{359}>2,$$
which is impossible.

{\bf Case 5.} $D=225$. Since $\rm{ord}_{3}(17)=\rm{ord}_{3}(5)=2$ and $\rm{ord}_{17}(3)=\rm{ord}_{17}(5)=16$, we have $p_{4}\equiv 1\pmod {51}$.
If $p_{4}\geqslant 593$, then
$$2=\frac{\sigma(n)}{n}+\frac{d}{n}<\frac{3}{2}\cdot \frac{5}{4}\cdot \frac{17}{16}\cdot \frac{593}{592}+\frac{1}{225}<2,$$
which is false. Thus $p_{4}\in\{103, 307, 409\}$. Since $\rm{ord}_{5}(3)=\rm{ord}_{5}(17)=4$ and ${\rm{ord}}_{5}(p_{4})$ are all even, we have $\alpha_{2}=2$.  Sine $31\mid (5^{3}-1)$, we deduce that the equality (\ref{2.3}) can not hold.

{\bf Case 6.} $D\in\{201, 213, 219, 237, 249, 267, 291, 303, 309, 321, 327, 339\}$. Then $p_{4}\in\{67, 71, 73, 79, 83, 89, 97, 101, 103, 107, 109, 113\}, \alpha_{2}\geqslant 4$ and
$$2=\frac{\sigma(n)}{n}+\frac{d}{n}\geqslant\frac{3^{5}-1}{2\cdot 3^{4}}\cdot \frac{5^{5}-1}{4\cdot 5^{4}}\cdot \frac{17^{3}-1}{16\cdot 17^{2}}\cdot \frac{113^{3}-1}{112\cdot 113^{2}}+\frac{1}{339}>2,$$
which is impossible.

{\bf Case 7.} $D=289$. If $p_{4}\geqslant 461$, then
$$2=\frac{\sigma(n)}{n}+\frac{d}{n}<\frac{3}{2}\cdot \frac{5}{4}\cdot \frac{17}{16}\cdot \frac{461}{460}+\frac{1}{289}<2,$$
which is clearly false. Thus $p_{4}\leqslant 457$. Since $\rm{ord}_{3}(17)=\rm{ord}_{3}(5)=2, \rm{ord}_{5}(3)=\rm{ord}_{5}(17)=4$, we have $p_{4}\equiv 1\pmod {15}$. Thus $p_{4}\leqslant 421, \alpha_{1}\geqslant 8, \alpha_{2}\geqslant 6, \alpha_{3}\geqslant 4$ and
$$2=\frac{\sigma(n)}{n}+\frac{d}{n}\geqslant\frac{3^{9}-1}{2\cdot 3^{8}}\cdot \frac{5^{7}-1}{4\cdot 5^{6}}\cdot \frac{17^{5}-1}{16\cdot 17^{4}}\cdot \frac{421^{3}-1}{420\cdot 421^{2}}+\frac{1}{289}>2,$$
which is impossible.

{\bf Case 8.} $D\in\{335, 355, 365\}$. Then $p_{4}\in\{67, 71, 73\}, \alpha_{1}\geqslant 6$  and
$$2=\frac{\sigma(n)}{n}+\frac{d}{n}\geqslant\frac{3^{7}-1}{2\cdot 3^{6}}\cdot \frac{5^{3}-1}{4\cdot 5^{2}}\cdot \frac{17^{3}-1}{16\cdot 17^{2}}\cdot \frac{73^{3}-1}{72\cdot 73^{2}}+\frac{1}{365}>2,$$
which is impossible.

{\bf Case 9.} $D\geqslant 367$. If $p_{4}\geqslant 397$, then
$$2=\frac{\sigma(n)}{n}+\frac{d}{n}<\frac{3}{2}\cdot \frac{5}{4}\cdot \frac{17}{16}\cdot \frac{397}{396}+\frac{1}{367}<2,$$
which is clearly false. Thus $p_{4}\leqslant 389$.

If $67\leqslant p_{4}\leqslant 113$, then $\alpha_{2}=2$ and $31\mid (2D-1)$. Otherwise, if $\alpha_{2}\geqslant 4$, then
$$2=\frac{\sigma(n)}{n}+\frac{d}{n}>\frac{3^{5}-1}{2\cdot 3^{4}}\cdot \frac{5^{5}-1}{4\cdot 5^{4}}\cdot \frac{17^{3}-1}{16\cdot 17^{2}}\cdot \frac{113^{3}-1}{112\cdot 113^{2}}>2,$$
which is clearly false. If $97\leqslant p_{4}\leqslant 113$, then
$$2=\frac{\sigma(n)}{n}+\frac{d}{n}<\frac{3}{2}\cdot \frac{5^{3}-1}{4\cdot 5^{2}}\cdot \frac{17}{16}\cdot \frac{97}{96}+\frac{1}{367}<2,$$
which is false. Thus $67\leqslant p_{4}\leqslant 89$.

If $127\leqslant p_{4}\leqslant 389$, then $\alpha_{2}\geqslant 4$. Otherwise, if $\alpha_{2}=2$, then
$$2=\frac{\sigma(n)}{n}+\frac{d}{n}<\frac{3}{2}\cdot \frac{5^{3}-1}{4\cdot 5^{2}}\cdot \frac{17}{16}\cdot \frac{127}{126}+\frac{1}{367}<2,$$
which is absurd. If $127\leqslant p_{4}\leqslant 199$, then $\alpha_{1}=4$. Otherwise, if $\alpha_{1}\geqslant 6$, then
$$2=\frac{\sigma(n)}{n}+\frac{d}{n}>\frac{3^{7}-1}{2\cdot 3^{6}}\cdot \frac{5^{5}-1}{4\cdot 5^{4}}\cdot \frac{17^{3}-1}{16\cdot 17^{2}}\cdot \frac{199^{3}-1}{198\cdot 199^{2}}>2,$$
which is impossible. If $151\leqslant p_{4}\leqslant 389$, then $\alpha_{1}\geqslant 6$. Otherwise, if $\alpha_{1}=4$, then
$$2=\frac{\sigma(n)}{n}+\frac{d}{n}<\frac{3^{5}-1}{2\cdot 3^{4}}\cdot \frac{5}{4}\cdot \frac{17}{16}\cdot \frac{151}{150}+\frac{1}{367}<2,$$
which is impossible. Thus $127\leqslant p_{4}\leqslant 149$ or $211\leqslant p_{4}\leqslant 389$.

Now we divide into the following fifteen subcases according to $p_{4}$.

{\bf Subcase 9.1 } $67\leqslant p_{4}\leqslant 89$. If $67\leqslant p_{4}\leqslant 79$, then $\alpha_{1}=4$ and $11\mid(2D-1)$. Otherwise, if $\alpha_{1}\geqslant 6$, then
 $$2=\frac{\sigma(n)}{n}+\frac{d}{n}>\frac{3^{7}-1}{2\cdot 3^{6}}\cdot \frac{5^{3}-1}{4\cdot 5^{2}}\cdot \frac{17^{3}-1}{16\cdot 17^{2}}\cdot \frac{79^{3}-1}{78\cdot 79^{2}}>2,$$
which is clearly false. It follows that $D>485$ and
$$2=\frac{\sigma(n)}{n}+\frac{d}{n}<\frac{3^{5}-1}{2\cdot 3^{4}}\cdot \frac{5^{3}-1}{4\cdot 5^{2}}\cdot \frac{17}{16}\cdot \frac{67}{66}+\frac{1}{485}<2,$$
which is impossible.

If $p_{4}=83$, then $\beta_{1}=\beta_{2}=\beta_{3}=0$ by $\rm{ord}_{3}(17)=\rm{ord}_{3}(83)=2, \rm{ord}_{5}(3)=\rm{ord}_{5}(17)=\rm{ord}_{5}(83)=4$, $\rm{ord}_{17}(3)=16$ and $\rm{ord}_{17}(83)=8$. Thus $D\geqslant 3^{4}\cdot 5^{2}\cdot 17^{2}$. If $\alpha_{1}\leqslant 6$, then
$$2=\frac{\sigma(n)}{n}+\frac{d}{n}<\frac{3^{7}-1}{2\cdot 3^{6}}\cdot \frac{5^{3}-1}{4\cdot 5^{2}}\cdot \frac{17}{16}\cdot \frac{83}{82}+\frac{1}{3^{4}\cdot 5^{2}\cdot 17^{2}}<2,$$
which is clearly false. Thus $\alpha_{1}\geqslant 8$.  If $\alpha_{3}=2$, then
$$2=\frac{\sigma(n)}{n}+\frac{d}{n}<\frac{3}{2}\cdot \frac{5^{3}-1}{4\cdot 5^{2}}\cdot \frac{17^{3}-1}{16\cdot 17^{2}}\cdot \frac{83}{82}+\frac{1}{3^{4}\cdot 5^{2}\cdot 17^{2}}<2,$$
which is false. Thus $\alpha_{3}\geqslant 4$ and
$$2=\frac{\sigma(n)}{n}+\frac{d}{n}>\frac{3^{9}-1}{2\cdot 3^{8}}\cdot \frac{5^{3}-1}{4\cdot 5^{2}}\cdot \frac{17^{5}-1}{16\cdot 17^{4}}\cdot \frac{83^{3}-1}{82\cdot 83^{2}}>2,$$
which is impossible.

If $p_{4}=89$, then $\beta_{1}=\beta_{2}=\beta_{3}=0$ by $\rm{ord}_{3}(17)=\rm{ord}_{3}(89)=2, \rm{ord}_{5}(3)=\rm{ord}_{5}(17)=4, \rm{ord}_{5}(89)=2$, $\rm{ord}_{17}(3)=16$ and $\rm{ord}_{17}(89)=4$. Thus $D\geqslant 3^{4}\cdot 5^{2}\cdot 17^{2}$ and
$$2=\frac{\sigma(n)}{n}+\frac{d}{n}<\frac{3}{2}\cdot \frac{5^{3}-1}{4\cdot 5^{2}}\cdot \frac{17}{16}\cdot \frac{89}{88}+\frac{1}{3^{4}\cdot 5^{2}\cdot 17^{2}}<2,$$
which is clearly false.

{\bf Subcase 9.2 } $p_{4}\in\{127, 139, 149\}$. Since $\rm{ord}_{5}(3)=\rm{ord}_{5}(17)=4$, $\rm{ord}_{17}(3)=\rm{ord}_{17}(5)=16$,  ${\rm{ord}}_{5}(p_{4})$ and ${\rm{ord}}_{17}(p_{4})$ are all even, we have $\beta_{2}=\beta_{3}=0, D\geqslant 5^{4}\cdot 17^{2}$ and
$$2=\frac{\sigma(n)}{n}+\frac{d}{n}<\frac{3^{5}-1}{2\cdot 3^{4}}\cdot \frac{5}{4}\cdot \frac{17}{16}\cdot \frac{127}{126}+\frac{1}{5^{4}\cdot 17^{2}}<2,$$
which is clearly false.

{\bf Subcase 9.3 } $p_{4}=131$. Since $\rm{ord}_{3}(5)=\rm{ord}_{3}(17)=\rm{ord}_{3}(131)=2$ and $\rm{ord}_{17}(3)=\rm{ord}_{17}(5)=\rm{ord}_{17}(131)=16$, we have $\beta_{1}=\beta_{3}=0, D\geqslant 3^{4}\cdot 17^{2}$ and
$$2=\frac{\sigma(n)}{n}+\frac{d}{n}<\frac{3^{5}-1}{2\cdot 3^{4}}\cdot \frac{5}{4}\cdot \frac{17}{16}\cdot \frac{131}{130}+\frac{1}{3^{4}\cdot 17^{2}}<2,$$
which is clearly false.

{\bf Subcase 9.4 } $p_{4}=137$. Since $\rm{ord}_{3}(5)=\rm{ord}_{3}(17)=\rm{ord}_{3}(137)=2$ and $\rm{ord}_{5}(3)=\rm{ord}_{5}(17)=\rm{ord}_{5}(137)=4$, we have $\beta_{1}=\beta_{2}=0, D\geqslant 3^{4}\cdot 5^{4}$ and
$$2=\frac{\sigma(n)}{n}+\frac{d}{n}<\frac{3^{5}-1}{2\cdot 3^{4}}\cdot \frac{5}{4}\cdot \frac{17}{16}\cdot \frac{137}{136}+\frac{1}{3^{4}\cdot 5^{4}}<2,$$
which is clearly false.

{\bf Subcase 9.5 } $p_{4}=211$. If $\alpha_{1}\geqslant 8$, then
$$2=\frac{\sigma(n)}{n}+\frac{d}{n}>\frac{3^{9}-1}{2\cdot 3^{8}}\cdot \frac{5^{5}-1}{4\cdot 5^{4}}\cdot \frac{17^{3}-1}{16\cdot 17^{2}}\cdot \frac{211^{3}-1}{210\cdot 211^{2}}>2,$$
which is clearly false. Thus $\alpha_{1}=6$. If $\alpha_{2}\geqslant 6$, then
$$2=\frac{\sigma(n)}{n}+\frac{d}{n}>\frac{3^{7}-1}{2\cdot 3^{6}}\cdot \frac{5^{7}-1}{4\cdot 5^{6}}\cdot \frac{17^{3}-1}{16\cdot 17^{2}}\cdot \frac{211^{3}-1}{210\cdot 211^{2}}>2,$$
which is false. Thus $\alpha_{2}=4$. If $\alpha_{3}\geqslant 4$, then
$$2=\frac{\sigma(n)}{n}+\frac{d}{n}>\frac{3^{7}-1}{2\cdot 3^{6}}\cdot \frac{5^{5}-1}{4\cdot 5^{4}}\cdot \frac{17^{5}-1}{16\cdot 17^{4}}\cdot \frac{211^{3}-1}{210\cdot 211^{2}}>2,$$
which is absurd. Thus $\alpha_{3}=2$.  By (\ref{2.3}), we have $\beta_{3}=\beta_{4}=0$ and
$$-1=\left(\frac{2}{11}\right)=\left(\frac{2\cdot 3^{6}\cdot 5^{4}\cdot 17^{2}\cdot 223^{\alpha_{4}}}{11}\right)=\left(\frac{3^{\beta_{1}}5^{\beta_{2}}}{11}\right)=1,$$
which is a contradiction.

{\bf Subcase 9.6 } $p_{4}=223$. Since $\rm{ord}_{5}(3)=\rm{ord}_{5}(17)=\rm{ord}_{5}(223)=4, \rm{ord}_{17}(3)=\rm{ord}_{17}(5)=16$
and $\rm{ord}_{17}(223)=8$, we have $\beta_{2}=\beta_{3}=0$. If $\alpha_{1}=6$, then
$$-1=\left(\frac{2}{1093}\right)=\left(\frac{2\cdot 3^{6}\cdot 5^{\alpha_{2}}17^{\alpha_{3}}223^{\alpha_{4}}}{1093}\right)=\left(\frac{3^{\beta_{1}} 223^{\beta_{4}}}{1093}\right)=1,$$
which is a contradiction. If $\alpha_{1}\geqslant 8$, then
$$2=\frac{\sigma(n)}{n}+\frac{d}{n}>\frac{3^{9}-1}{2\cdot 3^{8}}\cdot \frac{5^{5}-1}{4\cdot 5^{4}}\cdot \frac{17^{3}-1}{16\cdot 17^{2}}\cdot \frac{223^{3}-1}{222\cdot 223^{2}}>2,$$
which is also a contradiction.

{\bf Subcase 9.7 } $p_{4}=227$. Since $\rm{ord}_{17}(3)=\rm{ord}_{17}(5)=\rm{ord}_{17}(227)=16$ and $\rm{ord}_{5}(3)=\rm{ord}_{5}(17)=\rm{ord}_{5}(227)=4$, we have $\beta_{2}=\beta_{3}=0$. If $\alpha_{2}=4$, then
$$-1=\left(\frac{2}{11}\right)=\left(\frac{2\cdot 3^{\alpha_{1}}\cdot5^{4}\cdot17^{\alpha_{3}}227^{\alpha_{4}}}{11}\right)=\left(\frac{3^{\beta_{1}} 227^{\beta_{4}}}{11}\right)=(-1)^{\beta_{4}}$$
and
$$1=\left(\frac{2}{71}\right)=\left(\frac{2\cdot 3^{\alpha_{1}}\cdot5^{4}\cdot17^{\alpha_{3}}227^{\alpha_{4}}}{71}\right)=\left(\frac{3^{\beta_{1}} 227^{\beta_{4}}}{71}\right)=(-1)^{\beta_{4}}$$
which is false. Thus $\alpha_{2}\geqslant 6$. If $\alpha_{1}=6$, then
$$-1=\left(\frac{2}{1093}\right)=\left(\frac{2\cdot 3^{6}\cdot5^{\alpha_{2}}17^{\alpha_{3}}227^{\alpha_{4}}}{1093}\right)=\left(\frac{3^{\beta_{1}} 227^{\beta_{4}}}{1093}\right)=1,$$
which is a contradiction. If $\alpha_{1}\geqslant 8$, then
$$2=\frac{\sigma(n)}{n}+\frac{d}{n}>\frac{3^{9}-1}{2\cdot 3^{8}}\cdot \frac{5^{7}-1}{4\cdot 5^{6}}\cdot \frac{17^{3}-1}{16\cdot 17^{2}}\cdot \frac{227^{3}-1}{226\cdot 227^{2}}>2,$$
which is also a contradiction.

{\bf Subcase 9.8} $p_{4}=229$. Since $\rm{ord}_{17}(3)=\rm{ord}_{17}(5)=16, \rm{ord}_{17}(229)=8$ and $\rm{ord}_{5}(3)=\rm{ord}_{5}(17)=4, \rm{ord}_{5}(229)=2$, we have $\beta_{2}=\beta_{3}=0$. If $\alpha_{2}=4$, then
$$-1=\left(\frac{2}{11}\right)=\left(\frac{2\cdot 3^{\alpha_{1}}\cdot 5^{4}\cdot 17^{\alpha_{3}}229^{\alpha_{4}}}{11}\right)=\left(\frac{3^{\beta_{1}} 229^{\beta_{4}}}{11}\right)=1,$$
which is false. Thus $\alpha_{2}\geqslant 6$. If $\alpha_{1}=6$, then
$$-1=\left(\frac{2}{1093}\right)=\left(\frac{2\cdot 3^{6}\cdot 5^{\alpha_{2}}17^{\alpha_{3}}229^{\alpha_{4}}}{1093}\right)=\left(\frac{3^{\beta_{1}} 229^{\beta_{4}}}{1093}\right)=(-1)^{\beta_{4}}.$$
Thus $2\nmid\beta_{4}$. Noting that $\rm{ord}_{229}(17)=19, \rm{ord}_{229}(5)=114$, we have
$$1=\left(\frac{2}{1103}\right)
=\left(\frac{2\cdot 3^{6}\cdot 5^{\alpha_{2}}17^{\alpha_{3}}229^{\alpha_{4}}}{1103}\right)
=\left(\frac{3^{\beta_{1}} 229^{\beta_{4}}}{1103}\right)=-1.$$
which is a contradiction. If $\alpha_{1}\geqslant 8$, then
$$2=\frac{\sigma(n)}{n}+\frac{d}{n}>\frac{3^{9}-1}{2\cdot 3^{8}}\cdot \frac{5^{7}-1}{4\cdot 5^{6}}\cdot \frac{17^{3}-1}{16\cdot 17^{2}}\cdot \frac{229^{3}-1}{228\cdot 229^{2}}>2,$$
which is also a contradiction.

{\bf Subcase 9.9} $p_{4}\in\{233, 257\}$. Since ${\rm{ord}}_{p_{i}}(p_{j})$ are all even for $1\leqslant i\neq j\leqslant 4$, we have $\beta_{1}=\beta_{2}=\beta_{3}=\beta_{4}=0$ and $d=1$, a contradiction.

{\bf Subcase 9.10 } $p_{4}=239$. Since $\rm{ord}_{3}(5)=\rm{ord}_{3}(17)=\rm{ord}_{3}(239)=2$ and $\rm{ord}_{5}(3)=\rm{ord}_{5}(17)=4, \rm{ord}_{5}(239)=2$, we have $\beta_{1}=\beta_{2}=0$ and $D\geqslant 3^{6}\cdot 5^{4}$. If $\alpha_{2}=4$, then
$$2=\frac{\sigma(n)}{n}+\frac{d}{n}<\frac{3}{2}\cdot \frac{5^{5}-1}{4\cdot 5^{4}}\cdot \frac{17}{16}\cdot \frac{239}{238}+\frac{1}{3^{6}\cdot 5^{4}}<2,$$
which is false. Thus $\alpha_{2}\geqslant 6$. If $\alpha_{1}=6$, then
$$2=\frac{\sigma(n)}{n}+\frac{d}{n}<\frac{3^{7}-2}{2\cdot 3^{6}}\cdot \frac{5}{4}\cdot \frac{17}{16}\cdot \frac{239}{238}+\frac{1}{3^{6}\cdot 5^{4}}<2,$$
which is a contradiction. If $\alpha_{1}\geqslant 8$, then
$$2=\frac{\sigma(n)}{n}+\frac{d}{n}>\frac{3^{9}-1}{2\cdot 3^{8}}\cdot \frac{5^{7}-1}{4\cdot 5^{6}}\cdot \frac{17^{3}-1}{16\cdot 17^{2}}\cdot \frac{239^{3}-1}{238\cdot 239^{2}}>2,$$
which is also a contradiction.

{\bf Subcase 9.11 } $p_{4}=241$. Since $\rm{ord}_{17}(3)=\rm{ord}_{17}(5)=\rm{ord}_{17}(241)=16, \rm{ord}_{241}(3)=120, \rm{ord}_{241}(5)=40, \rm{ord}_{241}(17)=80$, we have $\beta_{3}=\beta_{4}=0$ and $D\geqslant 17^{2}\cdot 241^{2}$. If $\alpha_{2}=4$, then
$$2=\frac{\sigma(n)}{n}+\frac{d}{n}<\frac{3}{2}\cdot \frac{5^{5}-1}{4\cdot 5^{4}}\cdot \frac{17}{16}\cdot \frac{241}{240}+\frac{1}{17^{2}\cdot 241^{2}}<2,$$
which is false. Thus $\alpha_{2}\geqslant 6$. If $\alpha_{1}=6$, then
$$2=\frac{\sigma(n)}{n}+\frac{d}{n}<\frac{3^{7}-1}{2\cdot 3^{6}}\cdot \frac{5}{4}\cdot \frac{17}{16}\cdot \frac{241}{240}+\frac{1}{17^{2}\cdot 241^{2}}<2,$$
which is impossible. If $\alpha_{1}\geqslant 10$, then
$$2=\frac{\sigma(n)}{n}+\frac{d}{n}>\frac{3^{11}-1}{2\cdot 3^{10}}\cdot \frac{5^{7}-1}{4\cdot 5^{6}}\cdot \frac{17^{3}-1}{16\cdot 17^{3}}\cdot \frac{241^{3}-1}{240\cdot 241^{2}}>2,$$
which is absurd. Thus $\alpha_{1}=8$. If $\alpha_{3}\geqslant 4$, then
$$2=\frac{\sigma(n)}{n}+\frac{d}{n}>\frac{3^{9}-1}{2\cdot 3^{8}}\cdot \frac{5^{7}-1}{4\cdot 5^{6}}\cdot \frac{17^{5}-1}{16\cdot 17^{4}}\cdot \frac{241^{3}-1}{240\cdot 241^{2}}>2,$$
which is a contradiction. Thus $\alpha_{3}=2$. However,
$$2=\frac{\sigma(n)}{n}+\frac{d}{n}<\frac{3^{9}-1}{2\cdot 3^{8}}\cdot \frac{5}{4}\cdot \frac{17^{3}-1}{16\cdot 17^{3}}\cdot \frac{241}{240}+\frac{1}{17^{2}\cdot 241^{2}}<2,$$
which is also a contradiction.

{\bf Subcase 9.12 } $p_{4}=251$. Since $\rm{ord}_{3}(5)=\rm{ord}_{3}(17)=\rm{ord}_{3}(251)=2$ and $\rm{ord}_{17}(3)=\rm{ord}_{17}(5)=16, \rm{ord}_{17}(251)=4$, we have $\beta_{1}=\beta_{3}=0$ and $D\geqslant 3^{6}\cdot 17^{2}$. If $\alpha_{2}=4$, then
$$2=\frac{\sigma(n)}{n}+\frac{d}{n}<\frac{3}{2}\cdot \frac{5^{5}-1}{4\cdot 5^{4}}\cdot \frac{17}{16}\cdot \frac{251}{250}+\frac{1}{3^{6}\cdot 17^{2}}<2,$$
which is false. Thus $\alpha_{2}\geqslant 6$. If $\alpha_{1}=6$, then
$$2=\frac{\sigma(n)}{n}+\frac{d}{n}<\frac{3^{7}-1}{2\cdot 3^{6}}\cdot \frac{5}{4}\cdot \frac{17}{16}\cdot \frac{251}{250}+\frac{1}{3^{6}\cdot 17^{2}}<2,$$
which is impossible. Thus $\alpha_{1}\geqslant 8$. If $\alpha_{3}=2$, then
$$2=\frac{\sigma(n)}{n}+\frac{d}{n}<\frac{3}{2}\cdot \frac{5}{4}\cdot \frac{17^{3}-1}{16\cdot 17^{2}}\cdot \frac{251}{250}+\frac{1}{3^{6}\cdot 17^{2}}<2,$$
which is absurd. Thus $\alpha_{3}\geqslant 4$ and
$$2=\frac{\sigma(n)}{n}+\frac{d}{n}>\frac{3^{9}-1}{2\cdot 3^{8}}\cdot \frac{5^{7}-1}{4\cdot 5^{6}}\cdot \frac{17^{5}-1}{16\cdot 17^{4}}\cdot \frac{251^{3}-1}{250\cdot 251^{2}}>2,$$
which is a contradiction.

{\bf Subcase 9.13 } $p_{4}\in\{263, 269, 277, 283, 293, 313, 317, 337, 347, 349, 353, 359, 367, 373,\\
379, 383, 389\}$. Since $\rm{ord}_{5}(3)=\rm{ord}_{5}(17)=4, \rm{ord}_{17}(3)=\rm{ord}_{17}(5)=16$, ${\rm{ord}}_{5}(p_{4})$ and ${\rm{ord}}_{17}(p_{4})$ are all even, we have $\beta_{2}=\beta_{3}=0, D\geqslant 5^{4}\cdot 17^{2}$ and
$$2=\frac{\sigma(n)}{n}+\frac{d}{n}<\frac{3}{2}\cdot \frac{5}{4}\cdot \frac{17}{16}\cdot \frac{263}{262}+\frac{1}{5^{4}\cdot 17^{2}}<2,$$
which is impossible.

{\bf Subcase 9.14 } $p_{4}\in\{271, 281, 311, 331\}$. If $D\geqslant 2305$, then
$$2=\frac{\sigma(n)}{n}+\frac{d}{n}<\frac{3}{2}\cdot \frac{5}{4}\cdot \frac{17}{16}\cdot \frac{271}{270}+\frac{1}{2305}<2,$$
which is impossible. Since $\rm{ord}_{17}(3)=\rm{ord}_{17}(5)=16$ and ${\rm{ord}}_{17}(p_{4})$ are all even, we have $\beta_{3}=0$. Thus $D\in\{867, 1445\}$ and $\alpha_{3}=2$. However, $307\mid (17^{3}-1)$, a contradiction.

{\bf Subcase 9.15 } $p_{4}=307$. If $D\geqslant 769$, then
$$2=\frac{\sigma(n)}{n}+\frac{d}{n}<\frac{3}{2}\cdot \frac{5}{4}\cdot \frac{17}{16}\cdot \frac{307}{306}+\frac{1}{769}<2,$$
which is impossible. Since $\rm{ord}_{5}(3)=\rm{ord}_{5}(17)=\rm{ord}_{5}(307)=4$, we have $\beta_{2}=0, D=625$ and $\alpha_{2}=4$. However, $11\mid (5^{5}-1)$, a contradiction.

This completes the proof of Lemma \ref{lem2.3}.
\end{proof}

\begin{lemma}\label{lem2.4}
There is no odd deficient-perfect number of the form $n=3^{\alpha_{1}}5^{\alpha_{2}}19^{\alpha_{3}}p_{4}^{\alpha_{4}}$ with $D\geqslant 19$.
\end{lemma}
\begin{proof}
Assume that $n=3^{\alpha_{1}}5^{\alpha_{2}}19^{\alpha_{3}}p_{4}^{\alpha_{4}}$ is an odd deficient-perfect number with deficient divisor $d=3^{\beta_{1}}5^{\beta_{2}}19^{\beta_{3}}p_{4}^{\beta_{4}}$. By (\ref{Eq1}), we have
\begin{equation}\label{2.4}
\frac{3^{\alpha_{1}+1}-1}{2}\cdot \frac{5^{\alpha_{2}+1}-1}{4}\cdot \frac{19^{\alpha_{3}+1}-1}{18}\cdot \frac{p_{4}^{\alpha_{4}+1}-1}{p_{4}-1}=2\cdot 3^{\alpha_{1}}5^{\alpha_{2}} 19^{\alpha_{3}}p_{4}^{\alpha_{4}}-3^{\beta_{1}}5^{\beta_{2}}19^{\beta_{3}} p_{4}^{\beta_{4}}.
\end{equation}
If $\alpha_{1}=2$, then $13\mid (2D-1)$. Thus $D\geqslant 59$. If $p_{4}\geqslant 29$, then
$$2=\frac{\sigma(n)}{n}+\frac{d}{n}<\frac{3^{3}-1}{2\cdot 3^{2}}\cdot \frac{5}{4}\cdot \frac{19}{18}\cdot \frac{29}{28}+\frac{1}{59}<2,$$
which is false. Thus $p_{4}=23$. Since $\rm{ord}_{5}(3)=\rm{ord}_{5}(23)=4, \rm{ord}_{5}(19)=2, \rm{ord}_{23}(5)=\rm{ord}_{23}(19)=22$, we have $\beta_{2}=\beta_{4}=0$, $D\geqslant 5^{2}\cdot 23^{2}$ and
$$2=\frac{\sigma(n)}{n}+\frac{d}{n}<\frac{3^{3}-1}{2\cdot 3^{2}}\cdot \frac{5}{4}\cdot \frac{19}{18}\cdot \frac{23}{22}+\frac{1}{5^{2}\cdot 23^{2}}<2,$$
which is impossible. Thus $\alpha_{1}\geqslant 4$. If $p_{4}\leqslant 43$, then
$$2=\frac{\sigma(n)}{n}+\frac{d}{n}>\frac{3^{5}-1}{2\cdot 3^{4}}\cdot \frac{5^{3}-1}{4\cdot 5^{2}}\cdot \frac{19^{3}-1}{18\cdot 19^{2}}\cdot \frac{43^{3}-1}{42\cdot 43^{2}}>2,$$
which is absurd. Thus $p_{4}\geqslant 47$. If $D\leqslant 21$, then
$$2=\frac{\sigma(n)}{n}+\frac{d}{n}>\frac{3^{5}-1}{2\cdot 3^{4}}\cdot \frac{5^{3}-1}{4\cdot 5^{2}}\cdot \frac{19^{3}-1}{18\cdot 19^{2}}+\frac{1}{21}>2,$$
which is false. Thus $D\geqslant 25$. Now we divide into the following nine cases according to $D$.

{\bf Case 1.} $D\in\{25, 27, 45\}$. By (\ref{2.4}), we have $\alpha_{1}\geqslant 6, \alpha_{2}\geqslant 6$ and
$$2=\frac{\sigma(n)}{n}+\frac{d}{n}>\frac{3^{7}-1}{2\cdot 3^{6}}\cdot \frac{5^{7}-1}{4\cdot 5^{6}}\cdot \frac{19^{3}-1}{18\cdot 19^{2}}+\frac{1}{45}>2,$$
which is impossible.

{\bf Case 2.} $D\in\{47, 53, 59, 67, 73, 79, 83, 89, 97, 103, 107, 109, 113, 127,137, 139\}$. Then $p_{4}=D$.
Noting that $\rm{ord}_{5}(3)=4, \rm{ord}_{5}(19)=2$ and ${\rm{ord}}_{5}(p_{4})$ are all even, we deduce that the equality (\ref{2.4}) can not hold.

{\bf Case 3.} $D=57$. If $p_{4}\geqslant 607$
$$2=\frac{\sigma(n)}{n}+\frac{d}{n}<\frac{3}{2}\cdot \frac{5}{4}\cdot \frac{19}{18}\cdot \frac{607}{606}+\frac{1}{57}<2,$$
which is impossible. If $19\mid (5^{\alpha_{2}+1}-1)$, then $9\mid (\alpha_{2}+1)$ and $(5^{9}-1)\mid (5^{\alpha_{2}+1}-1)$. However, $31\mid (5^{9}-1)$, a contradiction. Since $\rm{ord}_{19}(3)=18$ and $\rm{ord}_{5}(3)=4, \rm{ord}_{5}(19)=2$, we have $p_{4}\leqslant 571$ and $5\mid (\alpha_{4}+1)$. Thus $\alpha_{1}\geqslant 10, \alpha_{2}\geqslant 6, \alpha_{3}\geqslant 4$ and
$$2=\frac{\sigma(n)}{n}+\frac{d}{n}\geqslant \frac{3^{11}-1}{2\cdot 3^{10}}\cdot \frac{5^{7}-1}{4\cdot 5^{6}}\cdot \frac{19^{5}-1}{18\cdot 19^{4}}\cdot \frac{571^{5}-1}{570\cdot 571^{4}}+\frac{1}{57}>2,$$
which is impossible.

{\bf Case 4.} $D\in\{61, 71, 101, 131\}$. Then $p_{4}=D$. Noting that $\rm{ord}_{5}(3)=4, \rm{ord}_{5}(19)=2$, we have $5\mid (\alpha_{4}+1)$ and $(p_{4}^{5}-1)\mid (p_{4}^{\alpha_{4}+1}-1)$. However, $131\mid (61^{5}-1), 11\mid (71^{5}-1), 31\mid (101^{5}-1)$ and $61\mid (131^{5}-1)$, a contradiction.

{\bf Case 5.} $D=75$. If $p_{4}\geqslant 269$
$$2=\frac{\sigma(n)}{n}+\frac{d}{n}<\frac{3}{2}\cdot \frac{5}{4}\cdot \frac{19}{18}\cdot \frac{269}{268}+\frac{1}{75}<2,$$
which is clearly false. By (\ref{2.4}), we have $\alpha_{1}\geqslant 10, \alpha_{2}\geqslant 6, p_{4}\leqslant 251$ and
$$2=\frac{\sigma(n)}{n}+\frac{d}{n}\geqslant \frac{3^{11}-1}{2\cdot 3^{10}}\cdot \frac{5^{7}-1}{4\cdot 5^{6}}\cdot \frac{19^{3}-1}{18\cdot 19^{2}}\cdot \frac{251^{3}-1}{250\cdot 251^{3}}+\frac{1}{75}>2,$$
which is impossible.

{\bf Case 6.} $D\in\{81, 95\}$. If $p_{4}\geqslant 239$
$$2=\frac{\sigma(n)}{n}+\frac{d}{n}<\frac{3}{2}\cdot \frac{5}{4}\cdot \frac{19}{18}\cdot \frac{239}{238}+\frac{1}{81}<2,$$
which is false. Thus $p_{4}\leqslant 233$. If $19\mid (5^{\alpha_{2}+1}-1)$, then $9\mid (\alpha_{2}+1)$ and $(5^{9}-1)\mid (5^{\alpha_{2}+1}-1)$. However, $31\mid (5^{9}-1)$, a contradiction. Since $\rm{ord}_{19}(3)=18$ and $\rm{ord}_{5}(3)=4, \rm{ord}_{5}(19)=2$, we have $p_{4}\leqslant 191$ and $5\mid (\alpha_{4}+1)$. Thus $\alpha_{1}\geqslant 10, \alpha_{2}\geqslant 6, \alpha_{3}\geqslant 4$ and
$$2=\frac{\sigma(n)}{n}+\frac{d}{n}\geqslant\frac{3^{11}-1}{2\cdot 3^{10}}\cdot \frac{5^{7}-1}{4\cdot 5^{6}}\cdot \frac{19^{5}-1}{18\cdot 19^{4}}\cdot \frac{191^{5}-1}{190\cdot 191^{4}}+\frac{1}{95}>2,$$
which is impossible.

{\bf Case 7.} $D=125$. If $p_{4}\geqslant 157$, then
$$2=\frac{\sigma(n)}{n}+\frac{d}{n}<\frac{3}{2}\cdot \frac{5}{4}\cdot \frac{19}{18}\cdot \frac{157}{156}+\frac{1}{125}<2,$$
which is false. Thus $p_{4}\leqslant 151$. By (\ref{2.4}), we have $\alpha_{1}\geqslant 12, \alpha_{2}\geqslant 6$ and
$$2=\frac{\sigma(n)}{n}+\frac{d}{n}\geqslant\frac{3^{13}-1}{2\cdot 3^{12}}\cdot \frac{5^{7}-1}{4\cdot 5^{6}}\cdot \frac{19^{3}-1}{18\cdot 19^{2}}\cdot \frac{151^{3}-1}{150\cdot 151^{2}}+\frac{1}{125}>2,$$
which is impossible.

{\bf Case 8.} $D=135$. If $p_{4}\geqslant 149$, then
$$2=\frac{\sigma(n)}{n}+\frac{d}{n}<\frac{3}{2}\cdot \frac{5}{4}\cdot \frac{19}{18}\cdot \frac{149}{148}+\frac{1}{135}<2,$$
which is false. Thus $p_{4}\leqslant 139$. By (\ref{2.4}), we have $\alpha_{1}\geqslant 8, \alpha_{2}\geqslant 6$ and
$$2=\frac{\sigma(n)}{n}+\frac{d}{n}\geqslant\frac{3^{9}-1}{2\cdot 3^{8}}\cdot \frac{5^{7}-1}{4\cdot 5^{6}}\cdot \frac{19^{3}-1}{18\cdot 19^{2}}\cdot \frac{139^{3}-1}{138\cdot 139^{2}}+\frac{1}{135}>2,$$
which is impossible.

{\bf Case 9.} $D\geqslant 141$. If $p_{4}\geqslant 149$, then
$$2=\frac{\sigma(n)}{n}+\frac{d}{n}<\frac{3}{2}\cdot \frac{5}{4}\cdot \frac{19}{18}\cdot \frac{149}{148}+\frac{1}{141}<2,$$
which is false. Thus $p_{4}\leqslant 139$.

If $47\leqslant p_{4}\leqslant 61$, then $\alpha_{2}=2, 31\mid (2D-1)$ and $D\geqslant 171$. Otherwise, if $\alpha_{2}\geqslant 4$, then
$$2=\frac{\sigma(n)}{n}+\frac{d}{n}>\frac{3^{5}-1}{2\cdot 3^{4}}\cdot \frac{5^{5}-1}{4\cdot 5^{4}}\cdot \frac{19^{3}-1}{18\cdot 19^{2}}\cdot \frac{61^{3}-1}{60\cdot 61^{2}}>2,$$
which is impossible.

If $67\leqslant p_{4}\leqslant 139$, then $\alpha_{2}\geqslant 4$. Otherwise, if $\alpha_{2}=2$, then $D\geqslant 171$ and
$$2=\frac{\sigma(n)}{n}+\frac{d}{n}<\frac{3}{2}\cdot \frac{5^{3}-1}{4\cdot 5^{2}}\cdot \frac{19}{18}\cdot \frac{67}{66}+\frac{1}{171}<2,$$
which is absurd.

{\bf Subcase 9.1 } $47\leqslant p_{4}\leqslant 61$. If $p_{4}=47$, then $\alpha_{1}=4$. Otherwise, if $\alpha_{1}\geqslant 6$, then
$$2=\frac{\sigma(n)}{n}+\frac{d}{n}>\frac{3^{7}-1}{2\cdot 3^{6}}\cdot \frac{5^{3}-1}{4\cdot 5^{2}}\cdot \frac{19^{3}-1}{18\cdot 19^{2}}\cdot \frac{47^{3}-1}{46\cdot 47^{2}}>2,$$
which is false. Since $\rm{ord}_{5}(3)=\rm{ord}_{5}(47)=4$ and $\rm{ord}_{5}(19)=2$, we have $\beta_{2}=0, D>1225$ and
$$2=\frac{\sigma(n)}{n}+\frac{d}{n}<\frac{3^{5}-1}{2\cdot 3^{4}}\cdot \frac{5^{3}-1}{4\cdot 5^{2}}\cdot \frac{19}{18}\cdot \frac{47}{46}+\frac{1}{1225}<2,$$
which is impossible.

If $p_{4}=53$, then $\alpha_{1}\leqslant 6$. Otherwise, if $\alpha_{1}\geqslant 8$, then
$$2=\frac{\sigma(n)}{n}+\frac{d}{n}>\frac{3^{9}-1}{2\cdot 3^{8}}\cdot \frac{5^{3}-1}{4\cdot 5^{2}}\cdot \frac{19^{3}-1}{18\cdot 19^{2}}\cdot \frac{53^{3}-1}{52\cdot 53^{2}}>2,$$
which is impossible. If $\alpha_{1}=4$, then
$$2=\frac{\sigma(n)}{n}+\frac{d}{n}<\frac{3^{5}-1}{2\cdot 3^{4}}\cdot \frac{5^{3}-1}{4\cdot 5^{2}}\cdot \frac{19}{18}\cdot \frac{53}{52}+\frac{1}{171}<2,$$
which is impossible. Thus $\alpha_{1}=6$. Since $\rm{ord}_{5}(3)=\rm{ord}_{5}(53)=4, \rm{ord}_{5}(19)=2, \rm{ord}_{53}(19)=52$ and $ \rm{ord}_{19}(53)=18$, we have $\beta_{2}=\beta_{3}=\beta_{4}=0$ and
$$-1=\left(\frac{2\cdot 3^{6}\cdot 5^{2}\cdot 19^{\alpha_{3}}53^{\alpha_{4}}}{1093}\right)=\left(\frac{3^{\beta_{1}}}{1093}\right)=1,$$
which is a contradiction.

If $p_{4}=59$, then $\beta_{2}=0$ by $\rm{ord}_{5}(3)=4$ and $\rm{ord}_{5}(19)=\rm{ord}_{5}(59)=2$. Thus $D>1225$ and
$$2=\frac{\sigma(n)}{n}+\frac{d}{n}<\frac{3}{2}\cdot \frac{5^{3}-1}{4\cdot 5^{2}}\cdot \frac{19}{18}\cdot \frac{59}{58}+\frac{1}{1225}<2,$$
which is impossible.

If $p_{4}=61$, then $D\leqslant 253$. Otherwise, if $D\geqslant 255$, then
$$2=\frac{\sigma(n)}{n}+\frac{d}{n}<\frac{3}{2}\cdot \frac{5^{3}-1}{4\cdot 5^{2}}\cdot \frac{19}{18}\cdot \frac{61}{60}+\frac{1}{255}<2,$$
which is impossible. Thus $D=171$. Since $\rm{ord}_{5}(3)=4, \rm{ord}_{5}(19)=2$ and $\rm{ord}_{25}(61)=5$, we have $5\mid (\alpha_{4}+1)$ and $(61^{5}-1)\mid (61^{\alpha_{4}+1}-1)$. However, $131\mid (61^{5}-1)$, a contradiction.

{\bf Subcase 9.2 } $p_{4}=67$. By $\rm{ord}_{19}(3)=\rm{ord}_{19}(67)=18$, we have $\beta_{3}=0$. If $\alpha_{2}\geqslant 6$, then
$$2=\frac{\sigma(n)}{n}+\frac{d}{n}>\frac{3^{5}-1}{2\cdot 3^{4}}\cdot \frac{5^{7}-1}{4\cdot 5^{6}}\cdot \frac{19^{3}-1}{18\cdot 19^{2}}\cdot \frac{67^{3}-1}{66\cdot 67^{2}}>2,$$
which is a contradiction. If $\alpha_{2}=4$, then
 $$-1=\left(\frac{2}{11}\right)=\left(\frac{2\cdot 3^{\alpha_{1}}5^{4}19^{\alpha_{3}}67^{\alpha_{4}}}{11}\right)=\left(\frac{3^{\beta_{1}}5^{\beta_{2}}67^{\beta_{4}}}{11}\right)=1,$$
which is also a contradiction.

{\bf Subcase 9.3 } $p_{4}=71$. If $\alpha_{1}\geqslant 6$, then
$$2=\frac{\sigma(n)}{n}+\frac{d}{n}>\frac{3^{7}-1}{2\cdot 3^{6}}\cdot \frac{5^{5}-1}{4\cdot 5^{4}}\cdot \frac{19^{3}-1}{18\cdot 19^{2}}\cdot \frac{71^{3}-1}{70\cdot 71^{2}}>2,$$
which is false. Thus $\alpha_{1}=4$ and $121|(2D-1)$. Thus $D>1271$ and
$$2=\frac{\sigma(n)}{n}+\frac{d}{n}<\frac{3^{5}-1}{2\cdot 3^{4}}\cdot \frac{5}{4}\cdot \frac{19}{18}\cdot \frac{71}{70}+\frac{1}{1271}<2,$$
which is impossible.

{\bf Subcase 9.4 } $p_{4}\in\{73, 79, 83\}$. Since $\rm{ord}_{5}(3)=4, \rm{ord}_{5}(19)=2$ and ${\rm{ord}}_{5}(p_{4})$ are all even, we have $\beta_{2}=0$ and $D\geqslant 625$. If $\alpha_{1}=4$, then
$$2=\frac{\sigma(n)}{n}+\frac{d}{n}<\frac{3^{5}-1}{2\cdot 3^{4}}\cdot \frac{5}{4}\cdot \frac{19}{18}\cdot \frac{73}{72}+\frac{1}{625}<2,$$
which is a contradiction. If $\alpha_{1}\geqslant 6$, then
$$2=\frac{\sigma(n)}{n}+\frac{d}{n}>\frac{3^{7}-1}{2\cdot 3^{6}}\cdot \frac{5^{5}-1}{4\cdot 5^{4}}\cdot \frac{19^{3}-1}{18\cdot 19^{2}}\cdot \frac{83^{3}-1}{82\cdot 83^{2}}>2,$$
which is also a contradiction.

{\bf Subcase 9.5 } $p_{4}=89$. Since $\rm{ord}_{5}(3)=4, \rm{ord}_{5}(19)=\rm{ord}_{5}(89)=2$, we have $\beta_{2}=0$ and $D\geqslant 625$. If $\alpha_{1}\geqslant 8$, then
$$2=\frac{\sigma(n)}{n}+\frac{d}{n}>\frac{3^{9}-1}{2\cdot 3^{8}}\cdot \frac{5^{5}-1}{4\cdot 5^{4}}\cdot \frac{19^{3}-1}{18\cdot 19^{2}}\cdot \frac{89^{3}-1}{88\cdot 89^{2}}>2,$$
which is a contradiction. If $\alpha_{1}=4$, then
$$2=\frac{\sigma(n)}{n}+\frac{d}{n}<\frac{3^{5}-1}{2\cdot 3^{4}}\cdot \frac{5}{4}\cdot \frac{19}{18}\cdot \frac{89}{88}+\frac{1}{625}<2,$$
which is also a contradiction. Thus $\alpha_{1}=6$. If $\alpha_{2}\geqslant 6$, then
$$2=\frac{\sigma(n)}{n}+\frac{d}{n}>\frac{3^{7}-1}{2\cdot 3^{6}}\cdot \frac{5^{7}-1}{4\cdot 5^{6}}\cdot \frac{19^{3}-1}{18\cdot 19^{2}}\cdot \frac{89^{3}-1}{88\cdot 89^{2}}>2,$$
which is impossible. Thus $\alpha_{2}=4$. If $\alpha_{3}\geqslant 4$, then
$$2=\frac{\sigma(n)}{n}+\frac{d}{n}>\frac{3^{7}-1}{2\cdot 3^{6}}\cdot \frac{5^{5}-1}{4\cdot 5^{4}}\cdot \frac{19^{5}-1}{18\cdot 19^{4}}\cdot \frac{89^{3}-1}{88\cdot 89^{2}}>2,$$
which is false. Thus $\alpha_{3}=2$. By (\ref{2.4}) and $\rm{ord}_{19}(89)=18$, we have
$\beta_{1}=1, \beta_{3}=\beta_{4}=0$ and
$$-1=\left(\frac{2}{11}\right)=\left(\frac{2\cdot 3^{6}\cdot 5^{4}\cdot 19^{2}\cdot 89^{\alpha_{4}}}{11}\right)=\left(\frac{3}{11}\right)=1,$$
which is absurd.

{\bf Subcase 9.6 } $p_{4}=97$. Since $\rm{ord}_{5}(3)=\rm{ord}_{5}(97)=4, \rm{ord}_{5}(19)=2, \rm{ord}_{97}(3)=48, \rm{ord}_{97}(5)=\rm{ord}_{97}(19)=96$,
 we have $\beta_{2}=\beta_{3}=0, D\geqslant 5^{2}\cdot 97^{2}$ and
 $$2=\frac{\sigma(n)}{n}+\frac{d}{n}<\frac{3}{2}\cdot \frac{5}{4}\cdot \frac{19}{18}\cdot \frac{97}{96}+\frac{1}{5^{2}\cdot 97^{2}}<2,$$
which is impossible.

{\bf Subcase 9.7 } $p_{4}=101$. If $\beta_{2}\geqslant 1$, then $5\mid (\alpha_{4}+1)$ and $(101^{5}-1)\mid(101^{\alpha_{4}+1}-1)$. Since $491\mid (101^{5}-1)$, we have $491\mid (2D-1)$. Thus $D>1719$ and
$$2=\frac{\sigma(n)}{n}+\frac{d}{n}<\frac{3}{2}\cdot \frac{5}{4}\cdot \frac{19}{18}\cdot \frac{101}{100}+\frac{1}{D}<2,$$
which is impossible. Thus $\beta_{2}=0$ and $5^{4}\mid D$. If $D\geqslant 961$, then
$$2=\frac{\sigma(n)}{n}+\frac{d}{n}<\frac{3}{2}\cdot \frac{5}{4}\cdot \frac{19}{18}\cdot \frac{101}{100}+\frac{1}{961}<2,$$
which is false. Thus $D=5^{4}$ and $\alpha_{2}=4$. However, we deduce that the equality (\ref{2.4}) cannot hold.

{\bf Subcase 9.8 } $p_{4}\in\{103, 107, 109, 113, 127, 137, 139\}$. Since $\rm{ord}_{5}(3)=4, \rm{ord}_{5}(19)=2$ and ${\rm{ord}}_{5}(p_{4})$ are all even, we have $\beta_{2}=0$. If $D\geqslant 701$, then
$$2=\frac{\sigma(n)}{n}+\frac{d}{n}<\frac{3}{2}\cdot \frac{5}{4}\cdot \frac{19}{18}\cdot \frac{103}{102}+\frac{1}{701}<2,$$
which is false. Thus $D=5^{4}$ and $\alpha_{2}=4$. However, we deduce that the equality (\ref{2.4}) cannot hold.

{\bf Subcase 9.9 } $p_{4}=131$. If $D\geqslant 179$, then
$$2=\frac{\sigma(n)}{n}+\frac{d}{n}<\frac{3}{2}\cdot \frac{5}{4}\cdot \frac{19}{18}\cdot \frac{131}{130}+\frac{1}{179}<2,$$
which is false. Thus $D=171$ and $\beta_{2}\geqslant 2$. Since $\rm{ord}_{5}(3)=4, \rm{ord}_{5}(19)=2, \rm{ord}_{25}(131)=5$, we have $(131^{5}-1)\mid(131^{\alpha_{4}+1}-1)$. However, $61\mid (131^{5}-1)$, a contradiction.

This completes the proof of Lemma \ref{lem2.4}.
\end{proof}

\begin{lemma}\label{lem2.5}
There is no odd deficient-perfect number of the form $n=3^{\alpha_{1}}5^{\alpha_{2}}23^{\alpha_{3}}p_{4}^{\alpha_{4}}$ with $D\geqslant 23$.
\end{lemma}
\begin{proof}
Assume that $n=3^{\alpha_{1}}5^{\alpha_{2}}23^{\alpha_{3}}p_{4}^{\alpha_{4}}$ is an odd deficient-perfect number with deficient divisor $d=3^{\beta_{1}}5^{\beta_{2}}23^{\beta_{3}}p_{4}^{\beta_{4}}$. By (\ref{Eq1}), we have
\begin{equation}\label{2.5}
\frac{3^{\alpha_{1}+1}-1}{2}\cdot \frac{5^{\alpha_{2}+1}-1}{4}\cdot \frac{23^{\alpha_{3}+1}-1}{22}\cdot \frac{p_{4}^{\alpha_{4}+1}-1}{p_{4}-1}=2\cdot 3^{\alpha_{1}}5^{\alpha_{2}}23^{\alpha_{3}}p_{4}^{\alpha_{4}}-3^{\beta_{1}}5^{\beta_{2}}23^{\beta_{3}}p_{4}^{\beta_{4}}.
\end{equation}
If $\alpha_{1}=2$, then
$$2=\frac{\sigma(n)}{n}+\frac{d}{n}<\frac{3^{3}-1}{2\cdot 3^{2}}\cdot \frac{5}{4}\cdot \frac{23}{22}\cdot \frac{29}{28}+\frac{1}{23}<2,$$
which is false. Thus $\alpha_{1}\geqslant 4$. If $p_{4}\leqslant 31$, then
$$2=\frac{\sigma(n)}{n}+\frac{d}{n}>\frac{3^{5}-1}{2\cdot 3^{4}}\cdot \frac{5^{3}-1}{4\cdot 5^{2}}\cdot \frac{23^{3}-1}{22\cdot 23^{2}}\cdot \frac{31^{3}-1}{30\cdot 31^{2}}>2,$$
which is false. Thus $p_{4}\geqslant 37$. Now we divide into the following five cases according to $D$.

{\bf Case 1.} $D\in\{23, 25\}$. By (\ref{2.5}), we have $p_{4}\equiv 1\pmod 5, \alpha_{1}\geqslant 10, \alpha_{2}\geqslant 6$ and
$$2=\frac{\sigma(n)}{n}+\frac{d}{n}>\frac{3^{11}-1}{2\cdot 3^{10}}\cdot \frac{5^{7}-1}{4\cdot 5^{6}}\cdot \frac{23^{3}-1}{22\cdot 23^{2}}+\frac{1}{25}>2,$$
which is impossible.

{\bf Case 2.} $D=27$. If $p_{4}\geqslant 719$, then
$$2=\frac{\sigma(n)}{n}+\frac{d}{n}<\frac{3}{2}\cdot \frac{5}{4}\cdot \frac{23}{22}\cdot \frac{719}{718}+\frac{1}{27}<2,$$
which is impossible. Since $\rm{ord}_{5}(3)=\rm{ord}_{5}(23)=4, \rm{ord}_{3}(5)=\rm{ord}_{3}(23)=2$, we have $p_{4}\equiv 1\pmod {15}$. Thus $p_{4}\leqslant 661$. By (\ref{2.5}), we have $\alpha_{1}\geqslant 10, \alpha_{2}\geqslant 6$ and
$$2=\frac{\sigma(n)}{n}+\frac{d}{n}\geqslant \frac{3^{11}-1}{2\cdot 3^{10}}\cdot \frac{5^{7}-1}{4\cdot 5^{6}}\cdot \frac{23^{3}-1}{22\cdot 23^{2}}\cdot \frac{661^{3}-1}{660\cdot 661^{2}}+\frac{1}{27}>2,$$
which is false.

{\bf Case 3.} $D\in\{37, 41, 43, 47, 53, 59, 61\}$. Then $p_{4}=D$. By (\ref{2.5}), we have $\alpha_{2}\geqslant 4$ and
$$2=\frac{\sigma(n)}{n}+\frac{d}{n}\geqslant\frac{3^{5}-1}{2\cdot 3^{4}}\cdot \frac{5^{5}-1}{4\cdot 5^{4}}\cdot \frac{23^{3}-1}{22\cdot 23^{2}}\cdot \frac{61^{3}-1}{60\cdot 61^{2}}+\frac{1}{61}>2,$$
which is impossible.

{\bf Case 4.} $D=45$. If $p_{4}\geqslant 113$, then
$$2=\frac{\sigma(n)}{n}+\frac{d}{n}<\frac{3}{2}\cdot \frac{5}{4}\cdot \frac{23}{22}\cdot \frac{113}{112}+\frac{1}{45}<2,$$
which is false. Since $\rm{ord}_{5}(3)=\rm{ord}_{5}(23)=4, \rm{ord}_{3}(5)=\rm{ord}_{3}(23)=2$, we have $p_{4}\equiv 1\pmod {15}$. Thus $p_{4}\leqslant 61$. By (\ref{2.5}), we have $\alpha_{2}\geqslant 4$ and
$$2=\frac{\sigma(n)}{n}+\frac{d}{n}\geqslant\frac{3^{5}-1}{2\cdot 3^{4}}\cdot \frac{5^{5}-1}{4\cdot 5^{4}}\cdot \frac{23^{3}-1}{22\cdot 23^{2}}\cdot \frac{61^{3}-1}{60\cdot 61^{2}}+\frac{1}{45}>2,$$
which is impossible.

{\bf Case 5.} $D\geqslant 67$. If $p_{4}\geqslant 83$, then
$$2=\frac{\sigma(n)}{n}+\frac{d}{n}<\frac{3}{2}\cdot \frac{5}{4}\cdot \frac{23}{22}\cdot \frac{83}{82}+\frac{1}{67}<2,$$
which is impossible. Thus $p_{4}\leqslant 79$.

{\bf Subcase 5.1 } $p_{4}=37$. If $\alpha_{2}\geqslant 4$, then
$$2=\frac{\sigma(n)}{n}+\frac{d}{n}>\frac{3^{5}-1}{2\cdot 3^{4}}\cdot \frac{5^{5}-1}{4\cdot 5^{4}}\cdot \frac{23^{3}-1}{22\cdot 23^{2}}\cdot \frac{37^{3}-1}{36\cdot 37^{2}}>2,$$
which is false. Thus $\alpha_{2}=2$ and $31\mid (2D-1)$. Since $\rm{ord}_{5}(3)=\rm{ord}_{5}(23)=\rm{ord}_{5}(37)=2$, we have $\beta_{2}=0$. Thus $D\geqslant 1225$ and
$$2=\frac{\sigma(n)}{n}+\frac{d}{n}<\frac{3}{2}\cdot \frac{5^{3}-1}{4\cdot 5^{2}}\cdot \frac{23}{22}\cdot \frac{37}{36}+\frac{1}{1225}<2,$$
which is impossible.

{\bf Subcase 5.2 } $p_{4}=41$. Since $\rm{ord}_{3}(5)=\rm{ord}_{3}(23)=\rm{ord}_{3}(41)=2$, we have
$\beta_{1}=0$ and $81\mid D$. If $\alpha_{2}\geqslant 4$, then
$$2=\frac{\sigma(n)}{n}+\frac{d}{n}>\frac{3^{5}-1}{2\cdot 3^{4}}\cdot \frac{5^{5}-1}{4\cdot 5^{4}}\cdot \frac{23^{3}-1}{22\cdot 23^{2}}\cdot \frac{41^{3}-1}{40\cdot 41^{2}}>2,$$
which is false. Thus $\alpha_{2}=2$ and $31\mid (2D-1)$. Thus $D\geqslant 729$ and
$$2=\frac{\sigma(n)}{n}+\frac{d}{n}<\frac{3}{2}\cdot \frac{5^{3}-1}{4\cdot 5^{2}}\cdot \frac{23}{22}\cdot \frac{41}{40}+\frac{1}{729}<2,$$
which is impossible.

{\bf Subcase 5.3 } $p_{4}=47$. Since $\rm{ord}_{3}(5)=\rm{ord}_{3}(23)=\rm{ord}_{3}(47)=2, \rm{ord}_{5}(3)=\rm{ord}_{5}(23)=\rm{ord}_{5}(47)=4$, we have $\beta_{1}=\beta_{2}=0$ and $D\geqslant 3^{4}\cdot 5^{2}$. If $\alpha_{2}=2$, then
$$2=\frac{\sigma(n)}{n}+\frac{d}{n}<\frac{3}{2}\cdot \frac{5^{3}-1}{4\cdot 5^{2}}\cdot \frac{23}{22}\cdot \frac{47}{46}+\frac{1}{3^{4}\cdot 5^{2}}<2,$$
which is impossible. Thus $\alpha_{2}\geqslant 4$. If $\alpha_{1}\geqslant 6$, then
$$2=\frac{\sigma(n)}{n}+\frac{d}{n}>\frac{3^{7}-1}{2\cdot 3^{6}}\cdot \frac{5^{5}-1}{4\cdot 5^{4}}\cdot \frac{23^{3}-1}{22\cdot 23^{2}}\cdot \frac{47^{3}-1}{46\cdot 47^{2}}>2,$$
which is a contradiction. If $\alpha_{1}=4$, then
$$2=\frac{\sigma(n)}{n}+\frac{d}{n}<\frac{3^{5}-1}{2\cdot 3^{4}}\cdot \frac{5}{4}\cdot \frac{23}{22}\cdot \frac{47}{46}+\frac{1}{3^{4}\cdot 5^{2}}<2,$$
which is also a contradiction.

{\bf Subcase 5.4 } $p_{4}\in\{53, 59\}$. Since $\rm{ord}_{3}(5)=\rm{ord}_{3}(23)=$$ {\rm{ord}_{3}}(p_{4})=2, \rm{ord}_{5}(3)=\rm{ord}_{5}(23)=4$ and ${\rm{ord}}_{5}(p_{4})$ are all even, we have $\beta_{1}=\beta_{2}=0, D\geqslant 3^{4}\cdot 5^{2}$ and
$$2=\frac{\sigma(n)}{n}+\frac{d}{n}<\frac{3}{2}\cdot \frac{5}{4}\cdot \frac{23}{22}\cdot \frac{53}{52}+\frac{1}{3^{4}\cdot 5^{2}}<2,$$
which is impossible.

{\bf Subcase 5.5 } $p_{4}=61$. If $\alpha_{2}=2$, then
$$2=\frac{\sigma(n)}{n}+\frac{d}{n}<\frac{3}{2}\cdot \frac{5^{3}-1}{4\cdot 5^{2}}\cdot \frac{23}{22}\cdot \frac{61}{60}+\frac{1}{67}<2,$$
which is impossible. Thus $\alpha_{2}\geqslant 4$. If $D\geqslant 141$, then
$$2=\frac{\sigma(n)}{n}+\frac{d}{n}<\frac{3}{2}\cdot \frac{5}{4}\cdot \frac{23}{22}\cdot \frac{61}{60}+\frac{1}{141}<2,$$
which is impossible. Thus $D\in\{69, 75, 81, 115, 125, 135\}$. Since $\rm{ord}_{5}(3)=\rm{ord}_{5}(23)=4$ and $\rm{ord}_{25}(61)=5$, we have $5|(\alpha_{4}+1)$ and $(61^{5}-1)|(61^{\alpha_{4}+1}-1)$. Noting that $131\mid (61^{5}-1)$, we have $131\mid (2D-1)$, a contradiction.

{\bf Subcase 5.6 } $p_{4}=71$. Since $\rm{ord}_{3}(5)=\rm{ord}_{3}(23)=\rm{ord}_{3}(71)=2$, we have $\beta_{1}=0$. If $\alpha_{1}=4$, then
$$2=\frac{\sigma(n)}{n}+\frac{d}{n}<\frac{3^{5}-1}{2\cdot 3^{4}}\cdot \frac{5}{4}\cdot \frac{23}{22}\cdot \frac{71}{70}+\frac{1}{67}<2,$$
which is impossible. If $\alpha_{1}\geqslant 6$, then $D\geqslant 3^{6}$ and
$$2=\frac{\sigma(n)}{n}+\frac{d}{n}<\frac{3}{2}\cdot \frac{5}{4}\cdot \frac{23}{22}\cdot \frac{71}{70}+\frac{1}{3^{6}}<2,$$
which is false.

{\bf Subcase 5.7 } $p_{4}\in\{67, 73, 79\}$. Since $\rm{ord}_{5}(3)=\rm{ord}_{5}(23)=4$ and ${\rm{ord}}_{5}(p_{4})$ are all even, we have $\beta_{2}=0$. If $\alpha_{2}=2$, then
$$2=\frac{\sigma(n)}{n}+\frac{d}{n}<\frac{3}{2}\cdot \frac{5^{3}-1}{4\cdot 5^{2}}\cdot \frac{23}{22}\cdot \frac{67}{66}+\frac{1}{67}<2,$$
which is impossible. If $\alpha_{2}\geqslant 4$, then $D\geqslant 5^{4}$ and
$$2=\frac{\sigma(n)}{n}+\frac{d}{n}<\frac{3}{2}\cdot \frac{5}{4}\cdot \frac{23}{22}\cdot \frac{67}{66}+\frac{1}{5^{4}}<2,$$
which is false.

This completes the proof of Lemma \ref{lem2.5}.
\end{proof}

\begin{lemma}\label{lem2.6}
There is no odd deficient-perfect number of the form $n=3^{\alpha_{1}}5^{\alpha_{2}}29^{\alpha_{3}}p_{4}^{\alpha_{4}}$ with $D\geqslant 25$.
\end{lemma}
\begin{proof}
Assume that $n=3^{\alpha_{1}}5^{\alpha_{2}}29^{\alpha_{3}}p_{4}^{\alpha_{4}}$ is an odd deficient-perfect number with deficient divisor $d=3^{\beta_{1}}5^{\beta_{2}}29^{\beta_{3}}p_{4}^{\beta_{4}}$. By (\ref{Eq1}), we have
\begin{equation}\label{2.6}
\frac{3^{\alpha_{1}+1}-1}{2}\cdot \frac{5^{\alpha_{2}+1}-1}{4}\cdot \frac{29^{\alpha_{3}+1}-1}{28}\cdot \frac{p_{4}^{\alpha_{4}+1}-1}{p_{4}-1}=2\cdot 3^{\alpha_{1}}5^{\alpha_{2}}29^{\alpha_{3}}p_{4}^{\alpha_{4}}-3^{\beta_{1}}5^{\beta_{2}}29^{\beta_{3}}p_{4}^{\beta_{4}}.
\end{equation}
If $\alpha_{1}=2$, then
$$2=\frac{\sigma(n)}{n}+\frac{d}{n}<\frac{3^{3}-1}{2\cdot 3^{2}}\cdot \frac{5}{4}\cdot \frac{29}{28}\cdot \frac{31}{30}+\frac{1}{25}<2,$$
which is false. Thus $\alpha_{1}\geqslant 4$. Now we divide into the following four cases according to $D$.

{\bf Case 1.} $D\in\{25, 27, 29\}$. By (\ref{2.6}), we have $\alpha_{1}\geqslant 6$ and $\alpha_{2}\neq 4$. Since $\rm{ord}_{3}(5)=\rm{ord}_{3}(29)=2, \rm{ord}_{5}(3)=4$ and $\rm{ord}_{5}(29)=2$, we have $p_{4}\equiv 1\pmod {15}$. If $p_{4}\geqslant 109$, then
$$2=\frac{\sigma(n)}{n}+\frac{d}{n}<\frac{3}{2}\cdot \frac{5}{4}\cdot \frac{29}{28}\cdot \frac{109}{108}+\frac{1}{25}<2,$$
which is impossible. Thus $p_{4}\leqslant 61$. If $\alpha_{2}=2$, then $p_{4}=31$ and
$$2=\frac{\sigma(n)}{n}+\frac{d}{n}\geqslant\frac{3^{7}-1}{2\cdot 3^{6}}\cdot \frac{5^{3}-1}{4\cdot 5^{2}}\cdot \frac{29^{3}-1}{28\cdot 29^{2}}\cdot \frac{31^{3}-1}{30\cdot 31^{2}}+\frac{1}{29}>2,$$
which is false. If $\alpha_{2}\geqslant 6$, then
$$2=\frac{\sigma(n)}{n}+\frac{d}{n}\geqslant\frac{3^{7}-1}{2\cdot 3^{6}}\cdot \frac{5^{7}-1}{4\cdot 5^{6}}\cdot \frac{29^{3}-1}{28\cdot 29^{2}}\cdot \frac{61^{3}-1}{60\cdot 61^{2}}+\frac{1}{29}>2,$$
which is impossible.

{\bf Case 2.} $D\in\{31, 37\}$. Then $p_{4}=D, \alpha_{1}\geqslant 6$ and
$$2=\frac{\sigma(n)}{n}+\frac{d}{n}\geqslant \frac{3^{7}-1}{2\cdot 3^{6}}\cdot \frac{5^{3}-1}{4\cdot 5^{2}}\cdot\frac{29^{3}-1}{28\cdot 29^{2}}\cdot \frac{37^{3}-1}{36\cdot 37^{2}}+\frac{1}{37}>2,$$
which is impossible.

{\bf Case 3.} $D=43$. Then $p_{4}=43, \alpha_{1}\geqslant 6, \alpha_{2}\geqslant 4$ and
$$2=\frac{\sigma(n)}{n}+\frac{d}{n}\geqslant \frac{3^{7}-1}{2\cdot 3^{6}}\cdot \frac{5^{5}-1}{4\cdot 5^{4}}\cdot \frac{29^{3}-1}{28\cdot 29^{2}}\cdot \frac{43^{3}-1}{42\cdot 43^{2}}+\frac{1}{43}>2,$$
which is impossible.

{\bf Case 4.} $D\geqslant 45$. If $p_{4}\geqslant 59$
$$2=\frac{\sigma(n)}{n}+\frac{d}{n}<\frac{3}{2}\cdot \frac{5}{4}\cdot \frac{29}{28}\cdot \frac{59}{58}+\frac{1}{45}<2,$$
which is impossible. Thus $p_{4}\leqslant 53$.

{\bf Subcase 4.1 } $p_{4}=31$. Since $\rm{ord}_{29}(3)=\rm{ord}_{29}(31)=28, \rm{ord}_{29}(5)=14$, we have $\beta_{3}=0$ and $D\geqslant 29^{2}$. If $\alpha_{1}=4$, then
$$-1=\left(\frac{2}{11}\right)=\left(\frac{2\cdot 3^{\alpha_{1}}5^{\alpha_{2}}29^{\alpha_{3}}31^{\alpha_{4}}}{11}\right)=\left(\frac{3^{\beta_{1}}5^{\beta_{2}}31^{\beta_{4}}}{11}\right)=1,$$
which is false. Thus $\alpha_{1}\geqslant 6$. If $\alpha_{2}=2$, then
$$2=\frac{\sigma(n)}{n}+\frac{d}{n}<\frac{3}{2}\cdot \frac{5^{3}-1}{4\cdot 5^{2}}\cdot \frac{29}{28}\cdot \frac{31}{30}+\frac{1}{29^{2}}<2,$$
which is a contradiction. If $\alpha_{2}\geqslant 4$, then
$$2=\frac{\sigma(n)}{n}+\frac{d}{n}>\frac{3^{7}-1}{2\cdot 3^{6}}\cdot \frac{5^{5}-1}{4\cdot 5^{4}}\cdot \frac{29^{3}-1}{28\cdot 29^{2}}\cdot \frac{31^{3}-1}{30\cdot 31^{2}}>2,$$
which is also a contradiction.

{\bf Subcase 4.2 } $p_{4}\in\{37, 41, 43, 47\}$. Since $\rm{ord}_{29}(3)=$${\rm{ord}}_{29}(p_{4})=28, \rm{ord}_{29}(5)=14$, we have $\beta_{3}=0, D\geqslant 29^{2}$ and
$$2=\frac{\sigma(n)}{n}+\frac{d}{n}<\frac{3}{2}\cdot \frac{5}{4}\cdot \frac{29}{28}\cdot \frac{37}{36}+\frac{1}{29^{2}}<2,$$
which is impossible.

{\bf Subcase 4.3 } $p_{4}=53$. By $\rm{ord}_{3}(5)=\rm{ord}_{3}(29)=\rm{ord}_{3}(53)=2$, we have $\beta_{1}=0, D\geqslant 3^{4}$ and
$$2=\frac{\sigma(n)}{n}+\frac{d}{n}<\frac{3}{2}\cdot \frac{5}{4}\cdot \frac{29}{28}\cdot \frac{53}{52}+\frac{1}{3^{4}}<2,$$
which is impossible.

This completes the proof of Lemma \ref{lem2.6}.
\end{proof}

\begin{lemma}\label{lem2.7}
There is no odd deficient-perfect number of the form $n=3^{\alpha_{1}}5^{\alpha_{2}}31^{\alpha_{3}}p_{4}^{\alpha_{4}}$ with $D\geqslant 15$.
\end{lemma}
\begin{proof}
Assume that $n=3^{\alpha_{1}}5^{\alpha_{2}}31^{\alpha_{3}}p_{4}^{\alpha_{4}}$ is an odd deficient-perfect number with deficient divisor $d=3^{\beta_{1}}5^{\beta_{2}}31^{\beta_{3}}p_{4}^{\beta_{4}}$. By (\ref{Eq1}), we have
\begin{equation}\label{2.7}
\frac{3^{\alpha_{1}+1}-1}{2}\cdot \frac{5^{\alpha_{2}+1}-1}{4}\cdot \frac{31^{\alpha_{3}+1}-1}{30}\cdot \frac{p_{4}^{\alpha_{4}+1}-1}{p_{4}-1}=2\cdot 3^{\alpha_{1}}5^{\alpha_{2}}31^{\alpha_{3}}p_{4}^{\alpha_{4}}-3^{\beta_{1}}5^{\beta_{2}}31^{\beta_{3}}p_{4}^{\beta_{4}}.
\end{equation}
If $\alpha_{1}=2$, then
$$2=\frac{\sigma(n)}{n}+\frac{d}{n}<\frac{3^{3}-1}{2\cdot 3^{2}}\cdot \frac{5}{4}\cdot \frac{31}{30}\cdot \frac{37}{36}+\frac{1}{15}<2,$$
which is clearly false. Thus $\alpha_{1}\geqslant 4$. If $D\geqslant 116$, then
$$2=\frac{\sigma(n)}{n}+\frac{d}{n}<\frac{3}{2}\cdot \frac{5}{4}\cdot \frac{31}{30}\cdot \frac{37}{36}+\frac{1}{116}<2,$$
which is false. If $53\leqslant D=p_{4}\leqslant 113$, then
$$2=\frac{\sigma(n)}{n}+\frac{d}{n}<\frac{3}{2}\cdot \frac{5}{4}\cdot \frac{31}{30}\cdot \frac{53}{52}+\frac{1}{53}<2,$$
which is impossible. Thus $D\in\{15, 25, 27, 31, 37, 41, 43, 45, 47, 75, 81, 93, 111\}$.

{\bf Case 1.} $D=15$. By (\ref{2.7}), we have $\alpha_{1}\geqslant 6$. If $\alpha_{2}\geqslant 4$, then
$$2=\frac{\sigma(n)}{n}+\frac{d}{n}>\frac{3^{7}-1}{2\cdot 3^{6}}\cdot \frac{5^{5}-1}{4\cdot 5^{4}}\cdot \frac{31^{3}-1}{30\cdot 31^{2}} +\frac{1}{15}>2,$$
which is impossible. Thus $\alpha_{2}=2$. If $p_{4}\geqslant 173$, then
$$2=\frac{\sigma(n)}{n}+\frac{d}{n}<\frac{3}{2}\cdot \frac{5^{3}-1}{4\cdot 5^{2}}\cdot \frac{31}{30}\cdot \frac{173}{172}+\frac{1}{15}<2,$$
which is clearly false. Thus $p_{4}\leqslant 167$. It follows that $\alpha_{1}\geqslant 8, \alpha_{3}\geqslant 4$ and
$$2=\frac{\sigma(n)}{n}+\frac{d}{n}\geqslant\frac{3^{9}-1}{2\cdot 3^{8}}\cdot \frac{5^{3}-1}{4\cdot 5^{2}}\cdot \frac{31^{5}-1}{30\cdot 31^{2}}\cdot \frac{167^{3}-1}{166\cdot 167^{2}}+\frac{1}{15}>2,$$
which is impossible.

{\bf Case 2.} $D=25$. If $p_{4}\geqslant 89$, then
$$2=\frac{\sigma(n)}{n}+\frac{d}{n}<\frac{3}{2}\cdot \frac{5}{4}\cdot \frac{31}{30}\cdot \frac{89}{88}+\frac{1}{25}<2,$$
which is clearly false. Thus $p_{4}\leqslant 83$. It follows that $\alpha_{1}\geqslant 12, \alpha_{3}\geqslant 4$. If $\alpha_{2}\geqslant 4$, then
$$2=\frac{\sigma(n)}{n}+\frac{d}{n}\geqslant\frac{3^{13}-1}{2\cdot 3^{12}}\cdot \frac{5^{5}-1}{4\cdot 5^{4}}\cdot \frac{31^{5}-1}{30\cdot 31^{4}}\cdot \frac{83^{3}-1}{82\cdot 83^{2}}+\frac{1}{25}>2,$$
which is impossible. Thus $\alpha_{2}=2$. If $p_{4}\geqslant 53$, then
$$2=\frac{\sigma(n)}{n}+\frac{d}{n}<\frac{3}{2}\cdot \frac{5^{3}-1}{4\cdot 5^{2}}\cdot \frac{31}{30}\cdot \frac{53}{52}+\frac{1}{25}<2,$$
which is a contradiction. If $p_{4}\leqslant 47$, then
$$2=\frac{\sigma(n)}{n}+\frac{d}{n}\geqslant\frac{3^{13}-1}{2\cdot 3^{12}}\cdot \frac{5^{3}-1}{4\cdot 5^{2}}\cdot \frac{31^{5}-1}{30\cdot 31^{2}}\cdot \frac{47^{3}-1}{46\cdot 47^{2}}+\frac{1}{25}>2,$$
which is also a contradiction.

{\bf Case 3.} $D\in\{27, 31, 45\}$. If $p_{4}\geqslant 79$, then
$$2=\frac{\sigma(n)}{n}+\frac{d}{n}<\frac{3}{2}\cdot \frac{5}{4}\cdot \frac{31}{30}\cdot \frac{79}{78}+\frac{1}{27}<2,$$
which is false. Thus $p_{4}\in\{37, 41, 43, 47, 53, 59, 61, 67, 71, 73\}$ and $\beta_{2}\geqslant 1$. Since $\rm{ord}_{5}(3)=\rm{ord}_{5}(37)=\rm{ord}_{5}(43)=\rm{ord}_{5}(47)=\rm{ord}_{5}(53)=\rm{ord}_{5}(67)=\rm{ord}_{5}(73)=4, \rm{ord}_{5}(59)=2$ and ${\rm{ord}}_{25}(31)=\rm{ord}_{25}(41)={\rm{ord}}_{25}(61)={\rm{ord}}_{25}(71)=5$, we have $(31^{5}-1)\mid(31^{\alpha_{3}+1}-1)$,  $(41^{5}-1)\mid(41^{\alpha_{4}+1}-1), (61^{5}-1)\mid(61^{\alpha_{4}+1}-1)$ or $(71^{5}-1)\mid(71^{\alpha_{4}+1}-1)$. However, $11\mid {31^{5}-1}, 579281\mid {41^{5}-1}, 131\mid {61^{5}-1}$ and $7\mid {71^{5}-1}$, a contradiction.

{\bf Case 4.} $D\in\{37, 43\}$. Then $p_{4}=D$. Since $\rm{ord}_{5}(3)=$${\rm{ord}}_{5}(p_{4})=4, \rm{ord}_{25}(31)=5$, we have $5\mid (\alpha_{3}+1)$ and $(31^{5}-1)\mid(31^{\alpha_{3}+1}-1)$. However, $11\mid (31^{5}-1)$, a contradiction.

{\bf Case 5.} $D=\{41, 47\}$. Then $p_{4}=D$. Since $\rm{ord}_{3}(5)=$${\rm{ord}}_{3}(p_{4})=2, \rm{ord}_{9}(31)=3$, we have $3\mid (\alpha_{3}+1)$ and $(31^{3}-1)\mid(31^{\alpha_{3}+1}-1)$. However, $331\mid (31^{3}-1)$, a contradiction.

{\bf Case 6.} $D\in\{75, 81, 93, 111\}$. If $p_{4}\geqslant 41$, then
$$2=\frac{\sigma(n)}{n}+\frac{d}{n}<\frac{3}{2}\cdot \frac{5}{4}\cdot \frac{31}{30}\cdot \frac{41}{40}+\frac{1}{75}<2,$$
which is clearly false. Thus $p_{4}=37$. If $\alpha_{2}=2$, then
$$2=\frac{\sigma(n)}{n}+\frac{d}{n}<\frac{3}{2}\cdot \frac{5^{3}-1}{4\cdot 5^{2}}\cdot \frac{31}{30}\cdot \frac{37}{36}+\frac{1}{75}<2,$$
which is false. Thus $\alpha_{2}\geqslant 4$ and $\beta_{2}\geqslant 1$. Since $\rm{ord}_{5}(3)=\rm{ord}_{5}(37)=4$ and ${\rm{ord}}_{25}(31)=5$, we have $(31^{5}-1)\mid(31^{\alpha_{3}+1}-1)$. However, $11\mid (31^{5}-1)$, a contradiction.

This completes the proof of Lemma \ref{lem2.7}.
\end{proof}

\begin{lemma}\label{lem2.8}
There is no odd deficient-perfect number of the form $n=3^{\alpha_{1}}5^{\alpha_{2}}37^{\alpha_{3}}p_{4}^{\alpha_{4}}$ with $D\geqslant 15$.
\end{lemma}

\begin{proof}
Assume that $n=3^{\alpha_{1}}5^{\alpha_{2}}37^{\alpha_{3}}p_{4}^{\alpha_{4}}$ is an odd deficient-perfect number with deficient divisor $d=3^{\beta_{1}}5^{\beta_{2}}37^{\beta_{3}}p_{4}^{\beta_{4}}$. If $D\geqslant 41$, then
$$2=\frac{\sigma(n)}{n}+\frac{d}{n}<\frac{3}{2}\cdot \frac{5}{4}\cdot \frac{37}{36}\cdot \frac{41}{40}+\frac{1}{41}<2,$$
which is false. Thus $D\in\{15, 25, 27, 37\}$. Now we divide into the following four cases.

{\bf Case 1.} $D=15$. If $p_{4}\geqslant 311$, then
$$2=\frac{\sigma(n)}{n}+\frac{d}{n}<\frac{3}{2}\cdot \frac{5}{4}\cdot \frac{37}{36}\cdot \frac{311}{310}+\frac{1}{15}<2,$$
which is false. Since $\rm{ord}_{5}(3)=\rm{ord}_{5}(37)=4$, we have $p_{4}\equiv 1\pmod {5}$. Thus $p_{4}\leqslant 281, \alpha_{1}\geqslant 12, \alpha_{2}\geqslant 8$ and
$$2=\frac{\sigma(n)}{n}+\frac{d}{n}\geqslant\frac{3^{13}-1}{2\cdot 3^{12}}\cdot \frac{5^{9}-1}{4\cdot 5^{8}}\cdot \frac{37^{3}-1}{36\cdot 37^{2}}\cdot \frac{281^{3}-1}{280\cdot 281^{2}}+\frac{1}{15}>2,$$
which is impossible.

{\bf Case 2.} $D=25$. If $p_{4}\geqslant 61$, then
$$2=\frac{\sigma(n)}{n}+\frac{d}{n}<\frac{3}{2}\cdot \frac{5}{4}\cdot \frac{37}{36}\cdot \frac{61}{60}+\frac{1}{25}<2,$$
which is clearly false. Thus $p_{4}\leqslant 59$. It follows that $\alpha_{1}\geqslant 12, \alpha_{2}\geqslant 8$ and
$$2=\frac{\sigma(n)}{n}+\frac{d}{n}\geqslant\frac{3^{13}-1}{2\cdot 3^{12}}\cdot \frac{5^{9}-1}{4\cdot 5^{8}}\cdot \frac{37^{3}-1}{36\cdot 37^{2}}\cdot \frac{59^{3}-1}{58\cdot 59^{2}}+\frac{1}{25}>2,$$
which is impossible.

{\bf Case 3.} $D=27$. If $p_{4}\geqslant 59$, then
$$2=\frac{\sigma(n)}{n}+\frac{d}{n}<\frac{3}{2}\cdot \frac{5}{4}\cdot \frac{37}{36}\cdot \frac{59}{58}+\frac{1}{25}<2,$$
which is clearly false. Thus $p_{4}\leqslant 53$. It follows that $\alpha_{1}\geqslant 12, \alpha_{2}\geqslant 8$ and
$$2=\frac{\sigma(n)}{n}+\frac{d}{n}\geqslant\frac{3^{13}-1}{2\cdot 3^{12}}\cdot \frac{5^{9}-1}{4\cdot 5^{8}}\cdot \frac{37^{3}-1}{36\cdot 37^{2}}\cdot \frac{53^{3}-1}{52\cdot 53^{2}}+\frac{1}{27}>2,$$
which is impossible.

{\bf Case 4.} $D=37$. If $p_{4}\geqslant 43$, then
$$2=\frac{\sigma(n)}{n}+\frac{d}{n}<\frac{3}{2}\cdot \frac{5}{4}\cdot \frac{37}{36}\cdot \frac{43}{42}+\frac{1}{27}<2,$$
which is clearly false. Thus $p_{4}=41$. It follows that $\alpha_{1}\geqslant 12, \alpha_{2}\geqslant 8$ and
$$2=\frac{\sigma(n)}{n}+\frac{d}{n}\geqslant\frac{3^{13}-1}{2\cdot 3^{12}}\cdot \frac{5^{9}-1}{4\cdot 5^{8}}\cdot \frac{37^{3}-1}{36\cdot 37^{2}}\cdot \frac{41^{3}-1}{40\cdot 41^{2}}+\frac{1}{37}>2,$$
which is impossible.

This completes the proof of Lemma \ref{lem2.8}.
\end{proof}

\begin{lemma}\label{lem2.9}
There is no odd deficient-perfect number of the form $n=3^{\alpha_{1}}5^{\alpha_{2}}41^{\alpha_{3}}p_{4}^{\alpha_{4}}$ with $D\geqslant 15$.
\end{lemma}
\begin{proof}
Assume that $n=3^{\alpha_{1}}5^{\alpha_{2}}41^{\alpha_{3}}p_{4}^{\alpha_{4}}$ is an odd deficient-perfect number with deficient divisor $d=3^{\beta_{1}}5^{\beta_{2}}41^{\beta_{3}}p_{4}^{\beta_{4}}$. If $D\geqslant 31$, then
$$2=\frac{\sigma(n)}{n}+\frac{d}{n}<\frac{3}{2}\cdot \frac{5}{4}\cdot \frac{41}{40}\cdot \frac{43}{42}+\frac{1}{31}<2,$$
a contradiction. Thus $D\in\{15, 25, 27\}$.

If $D=15$, then $p_{4}\leqslant 167$. Otherwise, if $p_{4}\geqslant 173$, then
$$2=\frac{\sigma(n)}{n}+\frac{d}{n}<\frac{3}{2}\cdot \frac{5}{4}\cdot \frac{41}{40}\cdot \frac{173}{172}+\frac{1}{15}<2,$$
which is clearly false. It follows that $\alpha_{1}\geqslant 12, \alpha_{2}\geqslant 8$ and
$$2=\frac{\sigma(n)}{n}+\frac{d}{n}\geqslant\frac{3^{13}-1}{2\cdot 3^{12}}\cdot \frac{5^{9}-1}{4\cdot 5^{8}}\cdot \frac{41^{3}-1}{40\cdot 41^{2}}\cdot \frac{167^{3}-1}{166\cdot 167^{2}}+\frac{1}{15}>2,$$
which is impossible.

If $D\in\{25, 27\}$, then $p_{4}\leqslant 47$. Otherwise, if $p_{4}\geqslant 53$, then
$$2=\frac{\sigma(n)}{n}+\frac{d}{n}<\frac{3}{2}\cdot \frac{5}{4}\cdot \frac{41}{40}\cdot \frac{53}{52}+\frac{1}{27}<2,$$
which is clearly false. It follows that $\alpha_{1}\geqslant 8, \alpha_{2}\geqslant 6$ and
$$2=\frac{\sigma(n)}{n}+\frac{d}{n}\geqslant\frac{3^{9}-1}{2\cdot 3^{8}}\cdot \frac{5^{7}-1}{4\cdot 5^{6}}\cdot \frac{41^{3}-1}{40\cdot 41^{2}}\cdot \frac{47^{3}-1}{46\cdot 47^{2}}+\frac{1}{27}>2,$$
which is impossible.

This completes the proof of Lemma \ref{lem2.9}.
\end{proof}

\begin{lemma}\label{lem2.10}
There is no odd deficient-perfect number $n$ of the form $n=3^{\alpha_{1}}5^{\alpha_{2}}43^{\alpha_{3}}p_{4}^{\alpha_{4}}$ with $D\geqslant 15$.
\end{lemma}
\begin{proof}
Assume that $n=3^{\alpha_{1}}5^{\alpha_{2}}43^{\alpha_{3}}p_{4}^{\alpha_{4}}$ is an odd deficient-perfect number with deficient divisor $d=3^{\beta_{1}}5^{\beta_{2}}43^{\beta_{3}}p_{4}^{\beta_{4}}$. If $D\geqslant 27$, then
$$2=\frac{\sigma(n)}{n}+\frac{d}{n}<\frac{3}{2}\cdot \frac{5}{4}\cdot \frac{43}{42}\cdot \frac{47}{46}+\frac{1}{27}<2,$$
which is clearly false. Thus $D\in\{15, 25\}$.

If $D=15$, then $p_{4}\leqslant 139$. Otherwise, if $p_{4}\geqslant 149$, then
$$2=\frac{\sigma(n)}{n}+\frac{d}{n}<\frac{3}{2}\cdot \frac{5}{4}\cdot \frac{43}{42}\cdot \frac{149}{148}+\frac{1}{15}<2,$$
which is false. Since $\rm{ord}_{5}(3)=\rm{ord}_{5}(43)=4$, we have $p_{4}\equiv 1\pmod {5}$. Thus $p_{4}\leqslant 131, \alpha_{1}\geqslant 6, \alpha_{2}\geqslant 6$ and
$$2=\frac{\sigma(n)}{n}+\frac{d}{n}\geqslant\frac{3^{7}-1}{2\cdot 3^{6}}\cdot \frac{5^{7}-1}{4\cdot 5^{6}}\cdot \frac{43^{3}-1}{42\cdot 43^{2}}\cdot \frac{131^{3}-1}{130\cdot 131^{2}}+\frac{1}{15}>2,$$
which is impossible.

If $D=25$, then $p_{4}=47$. Otherwise, if $p_{4}\geqslant 53$, then
$$2=\frac{\sigma(n)}{n}+\frac{d}{n}<\frac{3}{2}\cdot \frac{5}{4}\cdot \frac{43}{42}\cdot \frac{53}{52}+\frac{1}{25}<2,$$
which is clearly false. Thus $\alpha_{1}\geqslant 6, \alpha_{2}\geqslant 6$ and
$$2=\frac{\sigma(n)}{n}+\frac{d}{n}\geqslant\frac{3^{7}-1}{2\cdot 3^{6}}\cdot \frac{5^{7}-1}{4\cdot 5^{6}}\cdot \frac{43^{3}-1}{42\cdot 43^{2}}\cdot \frac{47^{3}-1}{46\cdot 47^{2}}+\frac{1}{25}>2,$$
which is impossible.

This completes the proof of Lemma \ref{lem2.10}.
\end{proof}

\begin{lemma}\label{lem2.11}
There is no odd deficient-perfect number of the form $n=3^{\alpha_{1}}5^{\alpha_{2}}p_{3}^{\alpha_{3}}p_{4}^{\alpha_{4}}$ with $67\leqslant p_{3}\leqslant 269$.
\end{lemma}
\begin{proof}
Assume that $n=3^{\alpha_{1}}5^{\alpha_{2}}p_{3}^{\alpha_{3}}p_{4}^{\alpha_{4}}$ is an odd deficient-perfect number with deficient divisor $d=3^{\beta_{1}}5^{\beta_{2}}p_{3}^{\beta_{3}}p_{4}^{\beta_{4}}$ and $67\leqslant p_{3}\leqslant 269$. If $D\geqslant 15$, then
$$2=\frac{\sigma(n)}{n}+\frac{d}{n}<\frac{3}{2}\cdot \frac{5}{4}\cdot \frac{67}{66}\cdot \frac{71}{70}+\frac{1}{15}<2,$$
which is clearly false. Thus $D=9$. By (\ref{Eq1}), we have
\begin{equation}\label{2.9}
\frac{3^{\alpha_{1}+1}-1}{2}\cdot \frac{5^{\alpha_{2}+1}-1}{4}\cdot \frac{p_{3}^{\alpha_{3}+1}-1}{p_{3}-1}\cdot \frac{p_{4}^{\alpha_{4}+1}-1}{p_{4}-1}=17\cdot 3^{\alpha_{1}-2}5^{\alpha_{2}}p_{3}^{\alpha_{3}}p_{4}^{\alpha_{4}}.
\end{equation}
Thus $\alpha_{1}\geqslant 6$ and $\alpha_{2}\geqslant 6$. If $\alpha_{1}=6$, then $p_{4}=1093$. Since $\rm{ord}_{17}(3)=\rm{ord}_{17}(5)=\rm{ord}_{17}(1093)=16$, we have $p_{3}\equiv 1\pmod {17}$. Thus $p_{3}\in\{103, 137, 239\}$. If $p_{3}=239$, then
$$2=\frac{\sigma(n)}{n}+\frac{d}{n}<\frac{3}{2}\cdot \frac{5}{4}\cdot \frac{239}{238}\cdot \frac{1093}{1092}+\frac{1}{9}<2,$$
which is clearly false. If $p_{3}\in\{103, 137\}$, then
$$2=\frac{\sigma(n)}{n}+\frac{d}{n}\geqslant\frac{3^{7}-1}{2\cdot 3^{6}}\cdot \frac{5^{7}-1}{4\cdot 5^{6}}\cdot \frac{137^{3}-1}{136\cdot 137^{2}}\cdot \frac{1093^{3}-1}{1092\cdot 1093^{2}}+\frac{1}{9}>2,$$
which is impossible. By (\ref{2.9}), we have $\alpha_{1}\geqslant 12$.
If $p_{3}\leqslant 131$, then
$$2=\frac{\sigma(n)}{n}+\frac{d}{n}\geqslant\frac{3^{13}-1}{2\cdot 3^{12}}\cdot \frac{5^{7}-1}{4\cdot 5^{6}}\cdot \frac{131^{3}-1}{130\cdot 131^{2}}+\frac{1}{9}>2,$$
which is impossible. Thus $p_{3}\geqslant 137$. If $p_{4}\geqslant 18503$, then
$$2=\frac{\sigma(n)}{n}+\frac{d}{n}<\frac{3}{2}\cdot \frac{5}{4}\cdot \frac{137}{136}\cdot \frac{18503}{18502}+\frac{1}{9}<2,$$
which is clearly false. Thus $p_{4}\leqslant 18493$. By (\ref{2.9}), we have $\alpha_{2}\geqslant 10$. Now we divide into the following twelve cases according to $p_{3}$.

{\bf Case 1.} $p_{3}=137$. Since $\rm{ord}_{3}(5)=\rm{ord}_{3}(137)=2$ and $\rm{ord}_{5}(3)=\rm{ord}_{5}(137)=4$, we have $p_{4}\equiv 1\pmod {15}$. Noting that $\rm{ord}_{289}(137)=17$, we have if $17\mid(\alpha_{3}+1)$ and $(137^{17}-1)\mid (137^{\alpha_{3}+1}-1)$, then $103\mid (137^{17}-1)$, a contradiction.
Since $\rm{ord}_{17}(3)=\rm{ord}_{17}(5)=16$, we have $p_{4}\equiv 1\pmod {17}$. Thus $p_{4}\leqslant 17851, \alpha_{3}\geqslant 4$ and
$$2=\frac{\sigma(n)}{n}+\frac{d}{n}\geqslant\frac{3^{13}-1}{2\cdot 3^{12}}\cdot \frac{5^{11}-1}{4\cdot 5^{10}}\cdot \frac{137^{5}-1}{136\cdot 137^{4}}\cdot \frac{17851^{3}-1}{17850\cdot 17851^{2}}+\frac{1}{9}>2,$$
which is impossible.

{\bf Case 2.} $p_{3}=139$. If $p_{4}\geqslant 6257$, then
$$2=\frac{\sigma(n)}{n}+\frac{d}{n}<\frac{3}{2}\cdot \frac{5}{4}\cdot \frac{139}{138}\cdot \frac{6257}{6256}+\frac{1}{9}<2,$$
which is false. Since $\rm{ord}_{17}(3)=\rm{ord}_{17}(5)=\rm{ord}_{17}(139)=16$, we have $p_{4}\equiv 1\pmod {17}$. Thus $p_{4}\leqslant 6121, \alpha_{3}\geqslant 4$ and
$$2=\frac{\sigma(n)}{n}+\frac{d}{n}\geqslant\frac{3^{13}-1}{2\cdot 3^{12}}\cdot \frac{5^{11}-1}{4\cdot 5^{10}}\cdot \frac{139^{5}-1}{138\cdot 139^{4}}\cdot \frac{6121^{3}-1}{6120\cdot 6121^{2}}+\frac{1}{9}>2,$$
which is impossible.

{\bf Case 3.} $p_{3}=149$. If $p_{4}\geqslant 1549$, then
$$2=\frac{\sigma(n)}{n}+\frac{d}{n}<\frac{3}{2}\cdot \frac{5}{4}\cdot \frac{149}{148}\cdot \frac{1549}{1548}+\frac{1}{9}<2,$$
which is clearly false. Noting that $\rm{ord}_{17}(3)=\rm{ord}_{17}(5)=16$ and $\rm{ord}_{17}(149)=4$, we have $p_{4}\equiv 1\pmod {17}$. Thus $p_{4}\leqslant 1531, \alpha_{3}\geqslant 4$ and
$$2=\frac{\sigma(n)}{n}+\frac{d}{n}\geqslant\frac{3^{13}-1}{2\cdot 3^{12}}\cdot \frac{5^{11}-1}{4\cdot 5^{10}}\cdot \frac{149^{5}-1}{148\cdot 149^{4}}\cdot \frac{1531^{3}-1}{1530\cdot 1531^{2}}+\frac{1}{9}>2,$$
which is impossible.

{\bf Case 4.} $p_{3}=151$. If $p_{4}\geqslant 1361$, then
$$2=\frac{\sigma(n)}{n}+\frac{d}{n}<\frac{3}{2}\cdot \frac{5}{4}\cdot \frac{151}{150}\cdot \frac{1361}{1360}+\frac{1}{9}<2,$$
which is clearly false. Thus $p_{4}\leqslant 1327, \alpha_{3}\geqslant 4$ and
$$2=\frac{\sigma(n)}{n}+\frac{d}{n}\geqslant\frac{3^{13}-1}{2\cdot 3^{12}}\cdot \frac{5^{11}-1}{4\cdot 5^{10}}\cdot \frac{151^{5}-1}{150\cdot 151^{4}}\cdot \frac{1327^{3}-1}{1326\cdot 1327^{2}}+\frac{1}{9}>2,$$
which is impossible.

{\bf Case 5.} $p_{3}=157$. If $p_{4}\geqslant 1013$, then
$$2=\frac{\sigma(n)}{n}+\frac{d}{n}<\frac{3}{2}\cdot \frac{5}{4}\cdot \frac{157}{156}\cdot \frac{1013}{1012}+\frac{1}{9}<2,$$
which is clearly false. Thus $p_{4}\leqslant 1009, \alpha_{3}\geqslant 4$ and
$$2=\frac{\sigma(n)}{n}+\frac{d}{n}\geqslant\frac{3^{13}-1}{2\cdot 3^{12}}\cdot \frac{5^{11}-1}{4\cdot 5^{10}}\cdot \frac{157^{5}-1}{156\cdot 157^{4}}\cdot \frac{1009^{3}-1}{1008\cdot 1009^{2}}+\frac{1}{9}>2,$$
which is impossible.

{\bf Case 6.} $p_{3}\in\{163, 167\}$. If $p_{4}\geqslant 821$, then
$$2=\frac{\sigma(n)}{n}+\frac{d}{n}<\frac{3}{2}\cdot \frac{5}{4}\cdot \frac{163}{162}\cdot \frac{821}{820}+\frac{1}{9}<2,$$
which is clearly false. Since $\rm{ord}_{17}(3)=\rm{ord}_{17}(5)=\rm{ord}_{17}(163)=\rm{ord}_{17}(167)=16$, we have $p_{4}\equiv 1\pmod {17}$. Thus $p_{4}\leqslant 647, \alpha_{3}\geqslant 4$ and
$$2=\frac{\sigma(n)}{n}+\frac{d}{n}\geqslant\frac{3^{13}-1}{2\cdot 3^{12}}\cdot \frac{5^{11}-1}{4\cdot 5^{10}}\cdot \frac{167^{5}-1}{166\cdot 167^{4}}\cdot \frac{647^{3}-1}{646\cdot 647^{2}}+\frac{1}{9}>2,$$
which is impossible.

{\bf Case 7.} $p_{3}=173$. If $p_{4}\geqslant 641$, then
$$2=\frac{\sigma(n)}{n}+\frac{d}{n}<\frac{3}{2}\cdot \frac{5}{4}\cdot \frac{173}{172}\cdot \frac{641}{640}+\frac{1}{9}<2,$$
which is false. Since $\rm{ord}_{17}(3)=\rm{ord}_{17}(5)=\rm{ord}_{17}(173)=16$, we have $p_{4}\equiv 1\pmod {17}$. Thus $p_{4}\leqslant 613, \alpha_{3}\geqslant 4$ and
$$2=\frac{\sigma(n)}{n}+\frac{d}{n}\geqslant\frac{3^{13}-1}{2\cdot 3^{12}}\cdot \frac{5^{11}-1}{4\cdot 5^{10}}\cdot \frac{173^{5}-1}{172\cdot 173^{4}}\cdot \frac{613^{3}-1}{612\cdot 613^{2}}+\frac{1}{9}>2,$$
which is impossible.

{\bf Case 8.} $p_{3}\in\{179, 181, 191, 193\}$. If $p_{4}\geqslant 563$, then
$$2=\frac{\sigma(n)}{n}+\frac{d}{n}<\frac{3}{2}\cdot \frac{5}{4}\cdot \frac{179}{178}\cdot \frac{563}{562}+\frac{1}{9}<2,$$
which is clearly false. Since $\rm{ord}_{17}(3)=\rm{ord}_{17}(5)=16$ and ${\rm{ord}}_{17}(p_{4})$ are even, we have $p_{4}\equiv 1\pmod {17}$. Thus $p_{4}\leqslant 443, \alpha_{3}\geqslant 4$ and
$$2=\frac{\sigma(n)}{n}+\frac{d}{n}\geqslant\frac{3^{13}-1}{2\cdot 3^{12}}\cdot \frac{5^{11}-1}{4\cdot 5^{10}}\cdot \frac{193^{5}-1}{192\cdot 193^{4}}\cdot \frac{443^{3}-1}{442\cdot 443^{2}}+\frac{1}{9}>2,$$
which is impossible.

{\bf Case 9.} $p_{3}\in\{197, 199\}$. If $p_{4}\geqslant 439$, then
$$2=\frac{\sigma(n)}{n}+\frac{d}{n}<\frac{3}{2}\cdot \frac{5}{4}\cdot \frac{197}{196}\cdot \frac{439}{438}+\frac{1}{9}<2,$$
which is clearly false. Since $\rm{ord}_{17}(3)=\rm{ord}_{17}(5)=16$ and ${\rm{ord}}_{17}(p_{4})$ are even, we have $p_{4}\equiv 1\pmod {17}$. Thus $p_{4}\leqslant 409, \alpha_{3}\geqslant 4$ and
$$2=\frac{\sigma(n)}{n}+\frac{d}{n}\geqslant\frac{3^{13}-1}{2\cdot 3^{12}}\cdot \frac{5^{11}-1}{4\cdot 5^{10}}\cdot \frac{199^{5}-1}{198\cdot 199^{4}}\cdot \frac{409^{3}-1}{408\cdot 409^{2}}+\frac{1}{9}>2,$$
which is impossible.

{\bf Case 10.} $p_{3}\in\{211, 223, 227, 229, 233, 241\}$. If $p_{4}\geqslant 383$, then
$$2=\frac{\sigma(n)}{n}+\frac{d}{n}<\frac{3}{2}\cdot \frac{5}{4}\cdot \frac{211}{210}\cdot \frac{383}{382}+\frac{1}{9}<2,$$
which is clearly false. Since $\rm{ord}_{17}(3)=\rm{ord}_{17}(5)=16$ and ${\rm{ord}}_{17}(p_{4})$ are even, we have $p_{4}\equiv 1\pmod {17}$. Thus $p_{4}\leqslant 307, \alpha_{3}\geqslant 4$ and
$$2=\frac{\sigma(n)}{n}+\frac{d}{n}\geqslant\frac{3^{13}-1}{2\cdot 3^{12}}\cdot \frac{5^{11}-1}{4\cdot 5^{10}}\cdot \frac{241^{5}-1}{240\cdot 241^{4}}\cdot \frac{307^{3}-1}{306\cdot 307^{2}}+\frac{1}{9}>2,$$
which is impossible.

{\bf Case 11.} $p_{3}=239$. If $p_{4}\geqslant 317$, then
$$2=\frac{\sigma(n)}{n}+\frac{d}{n}<\frac{3}{2}\cdot \frac{5}{4}\cdot \frac{239}{238}\cdot \frac{317}{316}+\frac{1}{9}<2,$$
which is clearly false. Thus $p_{4}\leqslant 313, \alpha_{3}\geqslant 4$ and
$$2=\frac{\sigma(n)}{n}+\frac{d}{n}\geqslant\frac{3^{13}-1}{2\cdot 3^{12}}\cdot \frac{5^{11}-1}{4\cdot 5^{10}}\cdot \frac{239^{5}-1}{238\cdot 239^{4}}\cdot \frac{313^{3}-1}{312\cdot 313^{2}}+\frac{1}{9}>2,$$
which is impossible.

{\bf Case 12.} $p_{3}\in\{251, 257, 263, 269\}$. If $p_{4}\geqslant 307$, then
$$2=\frac{\sigma(n)}{n}+\frac{d}{n}<\frac{3}{2}\cdot \frac{5}{4}\cdot \frac{251}{250}\cdot \frac{307}{306}+\frac{1}{9}<2,$$
which is clearly false. Thus $p_{4}\leqslant 293$. Since $\rm{ord}_{17}(3)=\rm{ord}_{17}(5)=16$ and ${\rm{ord}}_{17}(p_{3}), {\rm{ord}}_{17}(p_{4})$ are even, we deduce that the equality (\ref{2.9}) can not hold.

This completes the proof of Lemma \ref{lem2.11}.
\end{proof}

\section{The case of $p_{2}=7$}

 In this section, we consider the case of $n=p_{1}^{\alpha_{1}}p_{2}^{\alpha_{2}}p_{3}^{\alpha_{3}}p_{4}^{\alpha_{4}}$ with $p_{1}=3$ and $p_{2}=7$.

\begin{lemma}\label{lem3.1}
If $n=3^{\alpha_{1}}7^{\alpha_{2}}11^{\alpha_{3}}p_{4}^{\alpha_{4}}$ is an odd deficient-perfect number with $D\geqslant 7$, then $n=3^{2}\cdot 7^{2}\cdot 11^{2}\cdot 13^{2}$ with deficient divisor $d=3^{2}\cdot 7\cdot 13$.
\end{lemma}
\begin{proof}
 By (\ref{Eq1}), we have
\begin{equation}\label{3.1}
\frac{3^{\alpha_{1}+1}-1}{2}\cdot\frac{7^{\alpha_{2}+1}-1}{6}\cdot \frac{11^{\alpha_{3}+1}-1}{10}\cdot \frac{p_{4}^{\alpha_{4}+1}-1}{p_{4}-1}=2\cdot 3^{\alpha_{1}}7^{\alpha_{2}}11^{\alpha_{3}}p_{4}^{\alpha_{4}}-3^{\beta_{1}}7^{\beta_{2}}11^{\beta_{3}}p_{4}^{\beta_{4}}.
\end{equation}
We will divide into the following seven cases according to $D$.

{\bf Case 1.} $D=7$. If $\alpha_{1}\geqslant 4$, then
$$2=\frac{\sigma(n)}{n}+\frac{d}{n}\geqslant\frac{3^{5}-1}{2\cdot 3^{4}}\cdot\frac{7^{3}-1}{6\cdot 7^{2}}\cdot\frac{11^{3}-1}{10\cdot 11^{2}}+\frac{1}{7}>2,$$
which is clearly false. Thus $\alpha_{1}=2$. If $p_{4}\leqslant 181$, then
$$2=\frac{\sigma(n)}{n}+\frac{d}{n}\geqslant\frac{3^{3}-1}{2\cdot 3^{2}}\cdot\frac{7^{3}-1}{6\cdot 7^{2}}\cdot\frac{11^{3}-1}{10\cdot 11^{2}}\cdot\frac{181^{3}-1}{180\cdot 181^{2}}+\frac{1}{7}>2,$$
which is impossible. If $p_{4}\geqslant 541$, then
$$2=\frac{\sigma(n)}{n}+\frac{d}{n}<\frac{13}{9}\cdot \frac{7}{6}\cdot \frac{11}{10}\cdot \frac{541}{540}+\frac{1}{7}<2,$$
which is also impossible. Thus $191\leqslant p_{4}\leqslant 523$. By (\ref{3.1}), we have $\alpha_{2}\geqslant 6, \alpha_{3}\geqslant 6$ and
$$2=\frac{\sigma(n)}{n}+\frac{d}{n}\geqslant \frac{3^{3}-1}{2\cdot 3^{2}}\cdot\frac{7^{7}-1}{6\cdot 7^{6}}\cdot\frac{11^{7}-1}{10\cdot 11^{6}}\cdot \frac{523^{3}-1}{522\cdot 523^{2}} +\frac{1}{7}>2,$$
which is false.

{\bf Case 2.} $D\in\{9, 11\}$. If $\alpha_{1}\geqslant 4$, then
$$2=\frac{\sigma(n)}{n}+\frac{d}{n}>\frac{3^{5}-1}{2\cdot 3^{4}}\cdot\frac{7^{3}-1}{6\cdot 7^{2}}\cdot\frac{11^{3}-1}{10\cdot 11^{2}}+\frac{1}{11}>2,$$
which is clearly false. Thus $\alpha_{1}=2, p_{4}=13$ and
$$2=\frac{\sigma(n)}{n}+\frac{d}{n}\geqslant \frac{3^{3}-1}{2\cdot 3^{2}}\cdot\frac{7^{3}-1}{6\cdot 7^{2}}\cdot\frac{11^{3}-1}{10\cdot 11^{2}}\cdot \frac{13^{3}-1}{10\cdot 13^{2}} +\frac{1}{11}>2,$$
which is impossible.

{\bf Case 3.} $D\in\{13, 17, 19\}$. Then $p_{4}=D$ and
$$2=\frac{\sigma(n)}{n}+\frac{d}{n}\geqslant\frac{3^{3}-1}{2\cdot 3^{2}}\cdot\frac{7^{3}-1}{6\cdot 7^{2}}\cdot\frac{11^{3}-1}{10\cdot 11^{2}}\cdot\frac{p_{4}^{3}-1}{(p_{4}-1)\cdot p_{4}^{2}}+\frac{1}{p_{4}}>2,$$
which is impossible.

{\bf Case 4.} $D=21$. If $\alpha_{1}=2$, then $p_{4}=13$ and
$$2=\frac{\sigma(n)}{n}+\frac{d}{n}\geqslant\frac{3^{3}-1}{2\cdot 3^{2}}\cdot\frac{7^{3}-1}{6\cdot 7^{2}}\cdot\frac{11^{3}-1}{10\cdot 11^{2}}\cdot\frac{13^{3}-1}{12\cdot 13^{2}}+\frac{1}{21}>2,$$
which is impossible. Thus $\alpha_{1}\geqslant 4$. If $p_{4}\geqslant 73$, then
$$2=\frac{\sigma(n)}{n}+\frac{d}{n}<\frac{3}{2}\cdot\frac{7}{6}\cdot\frac{11}{10}\cdot \frac{73}{72} +\frac{1}{21}<2,$$
which is clearly false. If $p_{4}\leqslant 43$, then
$$2=\frac{\sigma(n)}{n}+\frac{d}{n}\geqslant\frac{3^{5}-1}{2\cdot 3^{4}}\cdot\frac{7^{3}-1}{6\cdot 7^{2}}\cdot\frac{11^{3}-1}{10\cdot 11^{2}}\cdot\frac{43^{3}-1}{42\cdot 43^{2}}+\frac{1}{21}>2,$$
which is false. Thus $47 \leqslant p_{4} \leqslant 71$. By (\ref{3.1}), we have $\alpha_{2}\geqslant 4$ and $\alpha_{3}\geqslant 4$. If $p_{4}\leqslant 53$, then
$$2=\frac{\sigma(n)}{n}+\frac{d}{n}\geqslant\frac{3^{5}-1}{2\cdot 3^{4}}\cdot\frac{7^{5}-1}{6\cdot 7^{4}}\cdot\frac{11^{5}-1}{10\cdot 11^{4}}\cdot\frac{53^{3}-1}{52\cdot 53^{2}}+\frac{1}{21}>2,$$
which is absurd. Thus $p_{4}\in\{59, 61, 67, 71\}$. Since ${\rm{ord}}_{41}(3)=8, {\rm{ord}}_{41}(7)={\rm{ord}}_{41}(11)={\rm{ord}}_{41}(67)={\rm{ord}}_{41}(71)=40, {\rm{ord}}_{41}(61)=20$ and ${\rm{ord}}_{41}(59)=5$, we have $p_{4}=59, 5\mid (\alpha_{4}+1)$ and $(59^{5}-1)\mid (59^{\alpha_{4}+1}-1)$. However, $151\mid (59^{5}-1)$, a contradiction.

{\bf Case 5.} $D\in\{23, 29, 31\}$. Then $p_{4}=D, \alpha_{1}\geqslant 4$ and
$$2=\frac{\sigma(n)}{n}+\frac{d}{n}\geqslant\frac{3^{5}-1}{2\cdot 3^{4}}\cdot\frac{7^{3}-1}{6\cdot 7^{2}}\cdot\frac{11^{3}-1}{10\cdot 11^{2}}\cdot\frac{p_{4}^{3}-1}{(p_{4}-1)\cdot p_{4}^{2}}+\frac{1}{p_{4}}>2,$$
which is impossible.

{\bf Case 6.} $D=27$. Then $\alpha_{1}\geqslant 4$. If $p_{4}\geqslant 53$, then
$$2=\frac{\sigma(n)}{n}+\frac{d}{n}<\frac{3}{2}\cdot\frac{7}{6}\cdot\frac{11}{10}\cdot \frac{53}{52} +\frac{1}{27}<2,$$
which is clearly false. If $p_{4}\leqslant 37$, then
$$2=\frac{\sigma(n)}{n}+\frac{d}{n}\geqslant\frac{3^{5}-1}{2\cdot 3^{4}}\cdot\frac{7^{3}-1}{6\cdot 7^{2}}\cdot\frac{11^{3}-1}{10\cdot 11^{2}}\cdot\frac{37^{3}-1}{36\cdot 37^{2}}+\frac{1}{27}>2,$$
which is false. Thus $p_{4}\in\{41, 43, 47\}$. Since $\rm{ord}_{7}(3)=\rm{ord}_{7}(47)=6, \rm{ord}_{7}(41)=2, \rm{ord}_{7}(11)=3$ and $ \rm{ord}_{7^{2}}(43)=7$, we have $3\mid (\alpha_{3}+1)$ and $(11^{3}-1)\mid (11^{\alpha_{3}+1}-1)$ or $7\mid (\alpha_{4}+1)$ and $(43^{7}-1)\mid (43^{\alpha_{4}+1}-1)$. However, $19\mid (11^{3}-1)$ and $5839\mid (43^{7}-1)$, a contradiction.

{\bf Case 7.} $D\geqslant 33$. If $p_{4}\geqslant 47$, then
$$2=\frac{\sigma(n)}{n}+\frac{d}{n}< \frac{3}{2}\cdot\frac{7}{6}\cdot\frac{11}{10}\cdot\frac{47}{46}+\frac{1}{33}<2,$$
which is clearly false. Thus $p_{4}\leqslant 43$.

{\bf Subcase 7.1 } $p_{4}=13$. If $\alpha_{1}\geqslant 4$, then
$$2=\frac{\sigma(n)}{n}+\frac{d}{n}>\frac{3^{5}-1}{2\cdot 3^{4}}\cdot \frac{7^{3}-1}{6\cdot 7^{2}}\cdot\frac{11^{3}-1}{10\cdot 11^{2}}\cdot\frac{13^{3}-1}{12\cdot 13^{2}}>2,$$
which is false. Thus $\alpha_{1}=2$. If $\alpha_{2}\geqslant 4$, then
$$2=\frac{\sigma(n)}{n}+\frac{d}{n}>\frac{3^{3}-1}{2\cdot 3^{2}}\cdot \frac{7^{5}-1}{6\cdot 7^{4}}\cdot\frac{11^{3}-1}{10\cdot 11^{2}}\cdot\frac{13^{3}-1}{12\cdot 13^{2}}>2,$$
which is absurd. Thus $\alpha_{2}=2$. If $\alpha_{3}\geqslant 4$, then
$$2=\frac{\sigma(n)}{n}+\frac{d}{n}>\frac{3^{3}-1}{2\cdot 3^{2}}\cdot \frac{7^{3}-1}{6\cdot 7^{2}}\cdot\frac{11^{5}-1}{10\cdot 11^{4}}\cdot\frac{13^{3}-1}{12\cdot 13^{2}}>2,$$
which is impossible. Thus $\alpha_{3}=2$. If $\alpha_{4}\geqslant 4$, then
$$2=\frac{\sigma(n)}{n}+\frac{d}{n}>\frac{3^{3}-1}{2\cdot 3^{2}}\cdot \frac{7^{3}-1}{6\cdot 7^{2}}\cdot\frac{11^{3}-1}{10\cdot 11^{2}}\cdot\frac{13^{5}-1}{12\cdot 13^{4}}>2,$$
which is false. Thus $\alpha_{4}=2$. By (\ref{3.1}), we have $n=3^{2}\cdot 7^{2} \cdot 11^{2}\cdot 13^{2}$ and $d=3^{2}\cdot 7\cdot 13$.

{\bf Subcase 7.2 } $p_{4}\in\{17, 19\}$. If $\alpha_{1}\geqslant 4$, then
$$2=\frac{\sigma(n)}{n}+\frac{d}{n}>\frac{3^{5}-1}{2\cdot 3^{4}}\cdot \frac{7^{3}-1}{6\cdot 7^{2}}\cdot\frac{11^{3}-1}{10\cdot 11^{2}}\cdot\frac{p_{4}^{3}-1}{(p_{4}-1)\cdot p_{4}^{2}}>2,$$
which is clearly false. Thus $\alpha_{1}=2$ and
$$2=\frac{\sigma(n)}{n}+\frac{d}{n}<\frac{3^{3}-1}{2\cdot 3^{2}}\cdot\frac{7}{6}\cdot\frac{11}{10}\cdot\frac{17}{16}+\frac{1}{33}<2,$$
which is impossible.

{\bf Subcase 7.3 } $p_{4}=23$. If $\alpha_{1}\geqslant 6$, then
$$2=\frac{\sigma(n)}{n}+\frac{d}{n}>\frac{3^{7}-1}{2\cdot 3^{6}}\cdot \frac{7^{3}-1}{6\cdot 7^{2}}\cdot\frac{11^{3}-1}{10\cdot 11^{2}}\cdot\frac{23^{3}-1}{22\cdot 23^{2}}>2,$$
which is clearly false. If $\alpha_{1}=2$, then
$$2=\frac{\sigma(n)}{n}+\frac{d}{n}<\frac{3^{3}-1}{2\cdot 3^{2}}\cdot\frac{7}{6}\cdot\frac{11}{10}\cdot\frac{23}{22}+\frac{1}{33}<2,$$
which is impossible. Thus $\alpha_{1}=4$. If $\alpha_{2}\geqslant 4$, then
$$2=\frac{\sigma(n)}{n}+\frac{d}{n}>\frac{3^{5}-1}{2\cdot 3^{4}}\cdot\frac{7^{5}-1}{6\cdot 7^{4}}\cdot\frac{11^{3}-1}{10\cdot 11^{2}}\cdot\frac{23^{3}-1}{22\cdot 23^{2}}>2,$$
which is absurd. Thus $\alpha_{2}=2$. If $D\leqslant 303$, then
$$2=\frac{\sigma(n)}{n}+\frac{d}{n}\geqslant \frac{3^{5}-1}{2\cdot 3^{4}}\cdot\frac{7^{3}-1}{6\cdot 7^{2}}\cdot\frac{11^{3}-1}{10\cdot 11^{2}}\cdot\frac{23^{3}-1}{22\cdot 23^{2}}+\frac{1}{303}>2,$$
which is a contradiction. If $D\geqslant 617$, then
$$2=\frac{\sigma(n)}{n}+\frac{d}{n}<\frac{3^{5}-1}{2\cdot 3^{4}}\cdot\frac{7^{3}-1}{6\cdot 7^{2}}\cdot\frac{11}{10}\cdot\frac{23}{22}+\frac{1}{617}<2,$$
which is also a contradiction. Thus $D\in\{363, 441, 483, 529, 539, 567\}$. However, $19\mid (7^{3}-1)$ and $19\nmid (2D-1)d$, which contradicts with (\ref{3.1}).

{\bf Subcase 7.4 } $p_{4}=29$. If $D\geqslant 161$, then
$$2=\frac{\sigma(n)}{n}+\frac{d}{n}<\frac{3}{2}\cdot \frac{7}{6}\cdot\frac{11}{10}\cdot\frac{29}{28}+\frac{1}{161}<2,$$
which is false. Thus $D\in\{33, 49, 63, 77, 81, 87, 99, 121, 147\}$. Since
$\rm{ord}_{29}(3)=\rm{ord}_{29}(11)=28$ and $\rm{ord}_{29}(7)=7$, we have $7\mid (\alpha_{2}+1)$ and $(7^{7}-1)\mid (7^{\alpha_{2}+1}-1)$. However, $4733\mid (7^{7}-1)$, a contradiction.

{\bf Subcase 7.5 } $p_{4}=31$. If $D\geqslant 93$, then
$$2=\frac{\sigma(n)}{n}+\frac{d}{n}<\frac{3}{2}\cdot \frac{7}{6}\cdot\frac{11}{10}\cdot\frac{31}{30}+\frac{1}{93}<2,$$
which is false. Thus $D\in\{33, 49, 63, 77, 81\}$. Since
$\rm{ord}_{31}(3)=\rm{ord}_{31}(11)=30$ and $\rm{ord}_{31}(7)=15$, we have $15\mid (\alpha_{2}+1)$ and $(7^{15}-1)\mid (7^{\alpha_{2}+1}-1)$. However, $2801\mid (7^{15}-1)$, a contradiction.

{\bf Subcase 7.6 } $p_{4}=37$. If $D\geqslant 47$, then
$$2=\frac{\sigma(n)}{n}+\frac{d}{n}<\frac{3}{2}\cdot \frac{7}{6}\cdot\frac{11}{10}\cdot\frac{37}{36}+\frac{1}{47}<2,$$
which is false. Thus $D\in\{33, 37\}$.  Noting that $\rm{ord}_{7}(3)=6, \rm{ord}_{7}(11)=\rm{ord}_{7}(37)=3$, we have $3\mid (\alpha_{3}+1)$ and $(11^{3}-1)\mid (11^{\alpha_{3}+1}-1)$ or $3\mid (\alpha_{4}+1)$ and $(37^{3}-1)\mid (37^{\alpha_{4}+1}-1)$. However, $19\mid (11^{3}-1)$ and $67\mid (37^{3}-1)$, a contradiction.

{\bf Subcase 7.7 } $p_{4}=41$. If $D\geqslant 39$, then
$$2=\frac{\sigma(n)}{n}+\frac{d}{n}<\frac{3}{2}\cdot \frac{7}{6}\cdot\frac{11}{10}\cdot\frac{41}{40}+\frac{1}{39}<2,$$
which is false. Thus $D=33$. Noting that $\rm{ord}_{7}(3)=6, \rm{ord}_{7}(41)=2$ and $\rm{ord}_{7}(11)=3$, we have $3\mid (\alpha_{3}+1)$ and $(11^{3}-1)\mid (11^{\alpha_{3}+1}-1)$. However, $19\mid (11^{3}-1)$, a contradiction.

This completes the proof of Lemma \ref{lem3.1}.
\end{proof}

\begin{lemma}\label{lem3.2}
There is no odd deficient-perfect number of the form $n=3^{\alpha_{1}}7^{\alpha_{2}}13^{\alpha_{3}}p_{4}^{\alpha_{4}}$ with $D\geqslant 7$.
\end{lemma}
\begin{proof}
Assume that $n=3^{\alpha_{1}}7^{\alpha_{2}}13^{\alpha_{3}}p_{4}^{\alpha_{4}}$ is an odd deficient-perfect number with deficient divisor $d=3^{\beta_{1}}7^{\beta_{2}}13^{\beta_{3}}p_{4}^{\beta_{4}}$. By (\ref{Eq1}), we have
\begin{equation}\label{3.2}
\frac{3^{\alpha_{1}+1}-1}{2}\cdot \frac{7^{\alpha_{2}+1}-1}{6}\cdot \frac{13^{\alpha_{3}+1}-1}{12}\cdot \frac{p_{4}^{\alpha_{4}+1}-1}{p_{4}-1}=2\cdot 3^{\alpha_{1}}7^{\alpha_{2}}13^{\alpha_{3}}p_{4}^{\alpha_{4}}-3^{\beta_{1}}7^{\beta_{2}}13^{\beta_{3}}p_{4}^{\beta_{4}}.
\end{equation}
We will divide into the following six cases according to $D$.

{\bf Case 1.} $D=7$. If $\alpha_{1}\geqslant 4$, then
$$2=\frac{\sigma(n)}{n}+\frac{d}{n}>\frac{3^{5}-1}{2\cdot 3^{4}}\cdot\frac{7^{3}-1}{6\cdot 7^{2}}\cdot\frac{13^{3}-1}{12\cdot 13^{2}}\cdot +\frac{1}{7}>2,$$
which is clearly false. Thus $\alpha_{1}=2$. If $p_{4}\geqslant 59$, then
$$2=\frac{\sigma(n)}{n}+\frac{d}{n}<\frac{3^{3}-1}{2\cdot 3^{2}}\cdot\frac{7}{6}\cdot\frac{13}{12}\cdot \frac{59}{58}+\frac{1}{7}<2,$$
which is a contradiction. If $p_{4}\leqslant 47$, then
$$2=\frac{\sigma(n)}{n}+\frac{d}{n}\geqslant \frac{3^{3}-1}{2\cdot 3^{2}}\cdot\frac{7^{3}-1}{6\cdot 7^{2}}\cdot\frac{13^{3}-1}{12\cdot 13^{2}}\cdot \frac{47^{3}-1}{46\cdot 47^{2}}+\frac{1}{7}>2,$$
which is also a contradiction. Thus $p_{4}=53$. Since $\rm{ord}_{7}(13)=2$ and $\rm{ord}_{7}(53)=3$, we have $3\mid (\alpha_{4}+1)$ and $(53^{3}-1)\mid (53^{\alpha_{4}+1}-1)$. However, $409\mid (53^{3}-1)$, a contradiction.

{\bf Case 2.} $D=9$. If $p_{4}\leqslant 23$, then
$$2=\frac{\sigma(n)}{n}+\frac{d}{n}\geqslant \frac{3^{3}-1}{2\cdot 3^{2}}\cdot\frac{7^{3}-1}{6\cdot 7^{2}}\cdot\frac{13^{3}-1}{12\cdot 13^{2}}\cdot \frac{23^{3}-1}{22\cdot 23^{2}}+\frac{1}{9}>2,$$
which is clearly false. Thus $p_{4}\geqslant 29$. If $\alpha_{1}\geqslant 8$, then
$$2=\frac{\sigma(n)}{n}+\frac{d}{n}>\frac{3^{9}-1}{2\cdot 3^{8}}\cdot\frac{7^{3}-1}{6\cdot 7^{2}}\cdot\frac{13^{3}-1}{12\cdot 13^{2}}+\frac{1}{9}>2,$$
which is impossible. Since $\rm{ord}_{17}(7)=16, \rm{ord}_{17}(13)=4$, we have $p_{4}\equiv 1\pmod {17}$. Thus $p_{4}\geqslant 103$ and $\alpha_{1}=2$.
However,
$$2=\frac{\sigma(n)}{n}+\frac{d}{n}<\frac{3^{3}-1}{2\cdot 3^{2}}\cdot\frac{7}{6}\cdot\frac{13}{12}\cdot\frac{103}{102}+\frac{1}{9}<2,$$
a contradiction.

{\bf Case 3.} $D=13$. If $p_{4}\geqslant 71$, then
$$2=\frac{\sigma(n)}{n}+\frac{d}{n}<\frac{3}{2}\cdot\frac{7}{6}\cdot\frac{13}{12}\cdot\frac{71}{70}+\frac{1}{13}<2,$$
which is false. Thus $p_{4}\leqslant 67$. Noting that $\rm{ord}_{5}(3)=\rm{ord}_{5}(7)=\rm{ord}_{5}(13)=4$, we have $p_{4}\in\{31, 41, 61\}, 5\mid(\alpha_{4}+1)$ and $(p_{4}^{5}-1)\mid (p_{4}^{\alpha_{4}+1}-1)$. However, $11\mid (31^{5}-1), 579281\mid (41^{5}-1), 131\mid (61^{5}-1)$, a contradiction.

{\bf Case 4.} $D=17$. Then $p_{4}=17$. Since $\rm{ord}_{17}(3)=\rm{ord}_{17}(7)=16$ and $\rm{ord}_{17}(13)=4$, we deduce that equality (\ref{3.2}) cannot hold.

{\bf Case 5.} $D=19$. Then $p_{4}=19$. Since $\rm{ord}_{37}(3)=18, \rm{ord}_{37}(13)=\rm{ord}_{37}(19)=36$ and $\rm{ord}_{37}(7)=9$, we have $9\mid (\alpha_{2}+1)$ and $(7^{9}-1)\mid (7^{\alpha_{2}+1}-1)$. However, $1063\mid (7^{9}-1)$, a contradiction.

{\bf Case 6.} $D\geqslant 21$. If $p_{4}\geqslant 37$, then
$$2=\frac{\sigma(n)}{n}+\frac{d}{n}< \frac{3}{2}\cdot\frac{7}{6}\cdot\frac{13}{12}\cdot\frac{37}{36}+\frac{1}{21}<2,$$
which is clearly false. Thus $p_{4}\leqslant 31$.

{\bf Subcase 6.1 } $p_{4}=17$. If $\alpha_{1}\geqslant 6$, then
$$2=\frac{\sigma(n)}{n}+\frac{d}{n}\geqslant \frac{3^{7}-1}{2\cdot 3^{6}}\cdot \frac{7^{3}-1}{6\cdot 7^{2}}\cdot\frac{13^{3}-1}{12\cdot 13^{2}}\cdot\frac{17^{3}-1}{16\cdot 17^{2}}>2,$$
which is clearly false. If $\alpha_{1}=2$, then
$$2=\frac{\sigma(n)}{n}+\frac{d}{n}< \frac{3^{3}-1}{2\cdot 3^{2}}\cdot\frac{7}{6}\cdot\frac{13}{12}\cdot\frac{17}{16}+\frac{1}{21}<2,$$
which is impossible. Thus $\alpha_{1}=4$. If $\alpha_{2}\geqslant 4$, then
$$2=\frac{\sigma(n)}{n}+\frac{d}{n}>\frac{3^{5}-1}{2\cdot 3^{4}}\cdot\frac{7^{5}-1}{6\cdot 7^{4}}\cdot\frac{13^{3}-1}{12\cdot 13^{2}}\cdot\frac{17^{3}-1}{16\cdot 17^{2}}>2,$$
which is absurd. Thus $\alpha_{2}=2$. Noting that $\rm{ord}_{7}(13)=2, \rm{ord}_{7}(17)=6, \rm{ord}_{13}(17)=6$ and $\rm{ord}_{17}(13)=16$, we have $\beta_{2}=\beta_{3}=\beta_{4}=0$ and
$$-1=\left(\frac{2}{11}\right)=\left(\frac{2\cdot 3^{4}\cdot 7^{2}\cdot 13^{\alpha_{3}}17^{\alpha_{4}}}{11}\right)=\left(\frac{3^{\beta_{1}}}{11}\right)=1,$$
which is a contradiction.

{\bf Subcase 6.2 } $p_{4}=19$. Noting that $\rm{ord}_{7}(3)=\rm{ord}_{7}(19)=6$ and $\rm{ord}_{7}(13)=2$, we have $\beta_{2}=0$ and $D\geqslant 49$.

If $\alpha_{1}=2$, then
$$2=\frac{\sigma(n)}{n}+\frac{d}{n}<\frac{3^{3}-1}{2\cdot 3^{2}}\cdot\frac{7}{6}\cdot\frac{13}{12}\cdot\frac{19}{18}+\frac{1}{49}<2,$$
which is clearly false. If $\alpha_{1}=4$, then $121\mid (2D-1)$. Thus $D>143$ and
$$2=\frac{\sigma(n)}{n}+\frac{d}{n}<\frac{3^{5}-1}{2\cdot 3^{4}}\cdot \frac{7}{6}\cdot\frac{13}{12}\cdot\frac{19}{18}+\frac{1}{143}<2,$$
which is absurd. If $\alpha_{1}=6$, then
$$-1=\left(\frac{2}{1093}\right)=\left(\frac{2\cdot 3^{6}\cdot 7^{\alpha_{2}}13^{\alpha_{3}}19^{\alpha_{4}}}{1093}\right)
=\left(\frac{3^{\beta_{1}}13^{\beta_{3}}19^{\beta_{4}}}{1093}\right)=1,$$
which is a contradiction. Thus $\alpha_{1}\geqslant 8$.

If $\alpha_{3}=2$, then
$$-1=\left(\frac{2}{61}\right)=\left(\frac{2\cdot 3^{\alpha_{1}}7^{\alpha_{2}}\cdot 13^{2}\cdot 19^{\alpha_{4}}}{61}\right)
=\left(\frac{3^{\beta_{1}}13^{\beta_{3}}19^{\beta_{4}}}{61}\right)=1,$$
which is a contradiction. Thus $\alpha_{3}\geqslant 4$.

If $\alpha_{2}=2$, then $D\leqslant 213$. Otherwise, if $D\geqslant 215$, then
$$2=\frac{\sigma(n)}{n}+\frac{d}{n}<\frac{3}{2}\cdot\frac{7^{3}-1}{6\cdot 7^{2}}\cdot\frac{13}{12}\cdot\frac{19}{18}+\frac{1}{215}<2,$$
which is clearly false. If $D\leqslant 197$, then
$$2=\frac{\sigma(n)}{n}+\frac{d}{n}\geqslant \frac{3^{9}-1}{2\cdot 3^{8}}\cdot\frac{7^{3}-1}{6\cdot 7^{2}}\cdot\frac{13^{5}-1}{12\cdot 13^{4}}\cdot\frac{19^{3}-1}{18\cdot 19^{2}}+\frac{1}{197}>2,$$
which is impossible. Thus $199\leqslant D\leqslant 213$. However, the case cannot hold since $49\mid D$. Thus $\alpha_{2}\geqslant 4$ and
$$2=\frac{\sigma(n)}{n}+\frac{d}{n}>\frac{3^{9}-1}{2\cdot 3^{8}}\cdot \frac{7^{5}-1}{6\cdot 7^{4}}\cdot\frac{13^{3}-1}{12\cdot 13^{2}}\cdot\frac{19^{3}-1}{18\cdot 19^{2}}>2,$$
which is absurd.

{\bf Subcase 6.3 } $p_{4}=23$. If $D\geqslant 57$, then
$$2=\frac{\sigma(n)}{n}+\frac{d}{n}<\frac{3}{2}\cdot\frac{7}{6}\cdot\frac{13}{12}\cdot\frac{23}{22}+\frac{1}{57}<2,$$
which is clearly false. Thus $D\in\{21, 23, 27, 39, 49\}$. Since $\rm{ord}_{7}(3)=6, \rm{ord}_{7}(13)=2$ and $\rm{ord}_{7}(23)=3$, we have if $7\mid (2D-1)d$, then $3\mid (\alpha_{4}+1)$ and $(23^{3}-1)\mid (23^{\alpha_{4}+1}-1)$. However, $79\mid (23^{3}-1)$, a contradiction. Thus $D=49, \alpha_{2}=2$ and
$$2=\frac{\sigma(n)}{n}+\frac{d}{n}<\frac{3}{2}\cdot \frac{7^{3}-1}{6\cdot 7^{2}}\cdot\frac{13}{12}\cdot\frac{23}{22}+\frac{1}{49}<2,$$
which is also a contradiction.

{\bf Subcase 6.4 } $p_{4}=29$. If $D\geqslant 29$, then
$$2=\frac{\sigma(n)}{n}+\frac{d}{n}<\frac{3}{2}\cdot \frac{7}{6}\cdot\frac{13}{12}\cdot\frac{29}{28}+\frac{1}{29}<2,$$
which is clearly false. Thus $D\in\{21, 27\}$. Since $\rm{ord}_{41}(3)=8$ and $\rm{ord}_{41}(7)=\rm{ord}_{41}(13)=\rm{ord}_{41}(29)=40$, we have $D=27$ and $\alpha_{1}\geqslant 4$. Since $\rm{ord}_{3}(29)=2$, we have $3\mid (\alpha_{2}+1)$ and $(7^{3}-1)\mid (7^{\alpha_{2}+1}-1)$ or $3\mid (\alpha_{3}+1)$ and $(13^{3}-1)\mid (13^{\alpha_{3}+1}-1)$. However, $19\mid (7^{3}-1)$ and $61\mid (13^{3}-1)$, a contradiction.

{\bf Subcase 6.5 } $p_{4}=31$. If $D\geqslant 25$, then
$$2=\frac{\sigma(n)}{n}+\frac{d}{n}<\frac{3}{2}\cdot \frac{7}{6}\cdot\frac{13}{12}\cdot\frac{31}{30}+\frac{1}{25}<2,$$
which is clearly false. Thus $D=21$. Noting that $\rm{ord}_{7}(3)=\rm{ord}_{7}(31)=6$ and $\rm{ord}_{7}(13)=2$, we deduce that the equality (\ref{3.2}) cannot hold.

This completes the proof of Lemma \ref{lem3.2}.
\end{proof}

\begin{lemma}\label{lem3.3}
There is no odd deficient-perfect number of the form $n=3^{\alpha_{1}}7^{\alpha_{2}}17^{\alpha_{3}}p_{4}^{\alpha_{4}}$ with $D\geqslant 7$.
\end{lemma}

\begin{proof}
Assume that $n=3^{\alpha_{1}}7^{\alpha_{2}}17^{\alpha_{3}}p_{4}^{\alpha_{4}}$ is an odd deficient-perfect number with deficient divisor $d=3^{\beta_{1}}7^{\beta_{2}}17^{\beta_{3}}p_{4}^{\beta_{4}}$. By (\ref{Eq1}), we have
\begin{equation}\label{3.3}
\frac{3^{\alpha_{1}+1}-1}{2}\cdot \frac{7^{\alpha_{2}+1}-1}{6}\cdot \frac{17^{\alpha_{3}+1}-1}{16}\cdot \frac{p_{4}^{\alpha_{4}+1}-1}{p_{4}-1}=2\cdot 3^{\alpha_{1}}7^{\alpha_{2}}17^{\alpha_{3}}p_{4}^{\alpha_{4}}-3^{\beta_{1}}7^{\beta_{2}}17^{\beta_{3}}p_{4}^{\beta_{4}}.
\end{equation}
If $D\geqslant 27$, then
$$2=\frac{\sigma(n)}{n}+\frac{d}{n}<\frac{3}{2}\cdot\frac{7}{6}\cdot\frac{17}{16}\cdot\frac{19}{18}+\frac{1}{27}<2,$$
which is impossible. Thus $D\in\{7, 9, 17, 19, 21, 23\}$.

{\bf Case 1.} $D=7$. Since $\rm{ord}_{17}(3)=\rm{ord}_{17}(7)=16$, we have $p_{4}\geqslant 103$. If $\alpha_{1}=2$, then
$$2=\frac{\sigma(n)}{n}+\frac{d}{n}<\frac{3^{3}-1}{2\cdot 3^{2}}\cdot\frac{7}{6}\cdot\frac{17}{16}\cdot\frac{103}{102}+\frac{1}{7}<2,$$
which is false. Thus $\alpha_{1}\geqslant 6, \alpha_{2}\geqslant 4$ and
$$2=\frac{\sigma(n)}{n}+\frac{d}{n}>\frac{3^{7}-1}{2\cdot 3^{6}}\cdot\frac{7^{5}-1}{6\cdot 7^{4}}\cdot\frac{17^{3}-1}{16\cdot 17^{2}}+\frac{1}{7}>2,$$
which is a contradiction.

{\bf Case 2.} $D\in\{9, 17, 19, 21, 23\}$. If $p_{4}\geqslant 67$, then
$$2=\frac{\sigma(n)}{n}+\frac{d}{n}<\frac{3}{2}\cdot\frac{7}{6}\cdot\frac{17}{16}\cdot\frac{67}{66}+\frac{1}{9}<2,$$
which is absurd. Thus $p_{4}\leqslant 61$.
Since $\rm{ord}_{17}(3)=\rm{ord}_{17}(7)=16$ and ${\rm{ord}}_{17}(p_{4})$ are all even, we deduce that the equality (\ref{3.3}) cannot hold.

This completes the proof of Lemma \ref{lem3.3}.
\end{proof}

\begin{lemma}\label{lem3.4}
There is no odd deficient-perfect number of the form $n=3^{\alpha_{1}}7^{\alpha_{2}}19^{\alpha_{3}}p_{4}^{\alpha_{4}}$ with $D\geqslant 7$.
\end{lemma}

\begin{proof}
Assume that $n=3^{\alpha_{1}}7^{\alpha_{2}}19^{\alpha_{3}}p_{4}^{\alpha_{4}}$ is an odd deficient-perfect number with deficient divisor $d=3^{\beta_{1}}7^{\beta_{2}}19^{\beta_{3}}p_{4}^{\beta_{4}}$. By (\ref{Eq1}), we have
\begin{equation}\label{3.4}
\frac{3^{\alpha_{1}+1}-1}{2}\cdot \frac{7^{\alpha_{2}+1}-1}{6}\cdot \frac{19^{\alpha_{3}+1}-1}{18}\cdot \frac{p_{4}^{\alpha_{4}+1}-1}{p_{4}-1}=2\cdot 3^{\alpha_{1}}7^{\alpha_{2}}19^{\alpha_{3}}p_{4}^{\alpha_{4}}-3^{\beta_{1}}7^{\beta_{2}}19^{\beta_{3}}p_{4}^{\beta_{4}}.
\end{equation}
If $p_{4}\geqslant 191$, then
$$2=\frac{\sigma(n)}{n}+\frac{d}{n}<\frac{3}{2}\cdot\frac{7}{6}\cdot\frac{19}{18}\cdot\frac{191}{190}+\frac{1}{7}<2,$$
which is clearly false. Thus $p_{4}\leqslant 181$. If $D\geqslant 15$, then $$2=\frac{\sigma(n)}{n}+\frac{d}{n}<\frac{3}{2}\cdot\frac{7}{6}\cdot\frac{19}{18}\cdot\frac{23}{22}+\frac{1}{15}<2,$$
which is impossible. Thus $D\in\{7, 9\}$.

{\bf Case 1.} $D=7$. If $\alpha_{1}=2$, then $p_{4}=23$. Otherwise, if $p_{4}\geqslant 29$, then
$$2=\frac{\sigma(n)}{n}+\frac{d}{n}<\frac{3^{3}-1}{2\cdot 3^{2}}\cdot\frac{7}{6}\cdot\frac{19}{18}\cdot\frac{29}{28}+\frac{1}{7}<2,$$
which is absurd. Since $\rm{ord}_{7}(19)=6$ and $\rm{ord}_{7}(23)=3$, we have $3\mid (\alpha_{4}+1)$ and $(23^{3}-1)\mid (23^{\alpha_{4}+1}-1)$. However, $79\mid (23^{3}-1)$, a contradiction. By (\ref{3.4}), we have $\alpha_{1}\geqslant 8$. If $p_{4}\leqslant 113$, then
$$2=\frac{\sigma(n)}{n}+\frac{d}{n}\geqslant \frac{3^{9}-1}{2\cdot 3^{8}} \cdot\frac{7^{3}-1}{6\cdot 7^{2}} \cdot\frac{19^{3}-1}{18\cdot 19^{2}}\cdot\frac{113^{3}-1}{112\cdot 113^{2}}+\frac{1}{7}>2,$$
which is false. Since $\rm{ord}_{7}(3)=\rm{ord}_{7}(19)=6$, we have $p_{4}\in\{127, 137, 149, 151, 163\}$.
If $p_{4}=127$, then $7\mid (\alpha_{4}+1)$ and $(127^{7}-1)\mid (127^{\alpha_{4}+1}-1)$, However, $43\mid (127^{7}-1)$, a contradiction.
If $p_{4}\in \{137, 149, 151, 163\}$, then $3\mid (\alpha_{4}+1)$ and $(p_{4}^{3}-1)\mid (p_{4}^{\alpha_{4}+1}-1)$. However,
$$37\mid (137^{3}-1), 31\mid (149^{3}-1), 1093\mid (151^{3}-1), 67\mid (163^{3}-1),$$
which is also a contradiction.

{\bf Case 2.} $D=9$. Since $\rm{ord}_{17}(3)=\rm{ord}_{17}(7)=16, \rm{ord}_{17}(19)=8$ and $\rm{ord}_{7}(3)=\rm{ord}_{7}(19)=6$, we have $p_{4}=137$. Thus $17\mid (\alpha_{4}+1)$ and $(137^{17}-1)\mid (137^{\alpha_{4}+1}-1)$. However, $103\mid (137^{17}-1)$, a contradiction.

This completes the proof of Lemma \ref{lem3.4}.
\end{proof}

\begin{lemma}\label{lem3.5}
There is no odd deficient-perfect number of the form $n=3^{\alpha_{1}}7^{\alpha_{2}}23^{\alpha_{3}}p_{4}^{\alpha_{4}}$ with $D\geqslant 7$.
\end{lemma}

\begin{proof}
Assume that $n=3^{\alpha_{1}}7^{\alpha_{2}}23^{\alpha_{3}}p_{4}^{\alpha_{4}}$ is an odd deficient-perfect number with deficient divisor $d=3^{\beta_{1}}7^{\beta_{2}}23^{\beta_{3}}p_{4}^{\beta_{4}}$. By (\ref{Eq1}), we have
\begin{equation}\label{3.5}
\frac{3^{\alpha_{1}+1}-1}{2}\cdot \frac{7^{\alpha_{2}+1}-1}{6}\cdot \frac{23^{\alpha_{3}+1}-1}{22}\cdot \frac{p_{4}^{\alpha_{4}+1}-1}{p_{4}-1}=2\cdot 3^{\alpha_{1}}7^{\alpha_{2}}23^{\alpha_{3}}p_{4}^{\alpha_{4}}-3^{\beta_{1}}7^{\beta_{2}}23^{\beta_{3}}p_{4}^{\beta_{4}}.
\end{equation}
If $p_{4}\geqslant 71$, then
$$2=\frac{\sigma(n)}{n}+\frac{d}{n}<\frac{3}{2}\cdot\frac{7}{6}\cdot\frac{23}{22}\cdot\frac{71}{70}+\frac{1}{7}<2,$$
which is false. Thus $p_{4}\in\{29, 31, 37, 41, 43, 47, 53, 59, 61, 67\}$. If $D\geqslant 11$, then $$2=\frac{\sigma(n)}{n}+\frac{d}{n}<\frac{3}{2}\cdot\frac{7}{6}\cdot\frac{23}{22}\cdot\frac{29}{28}+\frac{1}{11}<2,$$
which is impossible. Thus $D\in\{7, 9\}$.

{\bf Case 1.} $D=7$. If $\alpha_{1}=2$, then
$$2=\frac{\sigma(n)}{n}+\frac{d}{n}<\frac{3^{3}-1}{2\cdot 3^{2}}\cdot\frac{7}{6}\cdot\frac{23}{22}\cdot\frac{29}{28}+\frac{1}{7}<2,$$
which is absurd. By (\ref{3.5}), we have $\alpha_{1}\geqslant 12, \alpha_{2}\geqslant 4, \alpha_{3}\geqslant 4$ and
$$2=\frac{\sigma(n)}{n}+\frac{d}{n}\geqslant \frac{3^{13}-1}{2\cdot 3^{12}}\cdot\frac{7^{5}-1}{6\cdot 7^{4}} \cdot\frac{23^{5}-1}{22\cdot 23^{4}} \cdot\frac{67^{3}-1}{66\cdot 67^{2}}+\frac{1}{7}>2,$$
which is a contradiction.

{\bf Case 2.} $D=9$. Since $\rm{ord}_{17}(3)=\rm{ord}_{17}(7)=\rm{ord}_{17}(23)=16$ and ${\rm{ord}}_{17}(p_{4})$ are all even, we deduce that the equality (\ref{3.5}) cannot hold.

This completes the proof of Lemma \ref{lem3.5}.
\end{proof}

\begin{lemma}\label{lem3.6}
There is no odd deficient-perfect number of the form $n=3^{\alpha_{1}}7^{\alpha_{2}}29^{\alpha_{3}}p_{4}^{\alpha_{4}}$ with $D\geqslant 7$.
\end{lemma}

\begin{proof}
Assume that $n=3^{\alpha_{1}}7^{\alpha_{2}}29^{\alpha_{3}}p_{4}^{\alpha_{4}}$ is an odd deficient-perfect number with deficient divisor $d=3^{\beta_{1}}7^{\beta_{2}}29^{\beta_{3}}p_{4}^{\beta_{4}}$. If $p_{4}\geqslant 43$, then
$$2=\frac{\sigma(n)}{n}+\frac{d}{n}<\frac{3}{2}\cdot\frac{7}{6}\cdot\frac{29}{28}\cdot\frac{43}{42}+\frac{1}{7}<2,$$
which is false. Thus $p_{4}\in\{31, 37, 41\}$. If $D\geqslant 9$, then $$2=\frac{\sigma(n)}{n}+\frac{d}{n}<\frac{3}{2}\cdot\frac{7}{6}\cdot\frac{29}{28}\cdot\frac{31}{30}+\frac{1}{9}<2,$$
which is impossible. Thus $D=7$. By (\ref{Eq1}), we have
\begin{equation}\label{3.6}
\frac{3^{\alpha_{2}+1}-1}{2}\cdot\frac{7^{\alpha_{2}+1}-1}{6}\cdot \frac{29^{\alpha_{3}+1}-1}{28}\cdot \frac{p_{4}^{\alpha_{4}+1}-1}{p_{4}-1}=13\cdot 3^{\alpha_{1}}7^{\alpha_{2}-1}29^{\alpha_{3}}p_{4}^{\alpha_{4}}.
\end{equation}
If $\alpha_{1}=2$, then
$$2=\frac{\sigma(n)}{n}+\frac{d}{n}<\frac{3^{3}-1}{2\cdot 3^{2}}\cdot\frac{7}{6}\cdot\frac{29}{28}\cdot\frac{31}{30}+\frac{1}{7}<2,$$
which is absurd. By (\ref{3.6}), we have $\alpha_{1}\geqslant 12, \alpha_{2}\geqslant 4, \alpha_{3}\geqslant 4$ and
$$2=\frac{\sigma(n)}{n}+\frac{d}{n}\geqslant \frac{3^{13}-1}{2\cdot 3^{12}}\cdot\frac{7^{5}-1}{6\cdot 7^{4}} \cdot\frac{29^{5}-1}{28\cdot 29^{4}} \cdot\frac{41^{3}-1}{40\cdot 41^{2}}+\frac{1}{7}>2,$$
which is a contradiction.

This completes the proof of Lemma \ref{lem3.6}.
\end{proof}

\begin{lemma}\label{lem3.7}
There is no odd deficient-perfect number of the form $n=3^{\alpha_{1}}7^{\alpha_{2}}31^{\alpha_{3}}p_{4}^{\alpha_{4}}$ with $D\geqslant 7$.
\end{lemma}

\begin{proof}
Assume that $n=3^{\alpha_{1}}7^{\alpha_{2}}31^{\alpha_{3}}p_{4}^{\alpha_{4}}$ is an odd deficient-perfect number with deficient divisor $d=3^{\beta_{1}}7^{\beta_{2}}31^{\beta_{3}}p_{4}^{\beta_{4}}$.
If $p_{4}\geqslant 41$, then
$$2=\frac{\sigma(n)}{n}+\frac{d}{n}<\frac{3}{2}\cdot\frac{7}{6}\cdot\frac{31}{30}\cdot\frac{41}{40}+\frac{1}{7}<2,$$
which is false. Thus $p_{4}=37$. If $D\geqslant 9$, then $$2=\frac{\sigma(n)}{n}+\frac{d}{n}<\frac{3}{2}\cdot\frac{7}{6}\cdot\frac{31}{30}\cdot\frac{37}{36}+\frac{1}{9}<2,$$
which is impossible. Thus $D=7$. By (\ref{Eq1}), we have
$$\frac{3^{\alpha_{1}+1}-1}{2}\cdot\frac{7^{\alpha_{2}+1}-1}{6}\cdot \frac{31^{\alpha_{3}+1}-1}{30}\cdot \frac{37^{\alpha_{4}+1}-1}{36}=13\cdot 3^{\alpha_{1}}\cdot 7^{\alpha_{2}-1}\cdot 31^{\alpha_{3}}\cdot37^{\alpha_{4}}.$$
Since $\rm{ord}_{7}(3)=\rm{ord}_{7}(31)=6$ and $\rm{ord}_{7}(37)=3$, we have $3\mid (\alpha_{4}+1)$ and $(37^{3}-1)\mid (37^{\alpha_{4}+1}-1)$. However,
$67\mid (37^{\alpha_{4}+1}-1)$, a contradiction.

This completes the proof of Lemma \ref{lem3.7}.
\end{proof}

\section{The case of $p_{2}\in\{11, 13\}$}

 In this section, we study the case of $n=p_{1}^{\alpha_{1}}p_{2}^{\alpha_{2}}p_{3}^{\alpha_{3}}p_{4}^{\alpha_{4}}$ with $p_{1}=3$ and $p_{2}\in\{11, 13\}$.

\begin{lemma}\label{lem4.1}
There is no odd deficient-perfect number of the form $n=3^{\alpha_{1}}11^{\alpha_{2}}p_{3}^{\alpha_{3}}p_{4}^{\alpha_{4}}$.
\end{lemma}
\begin{proof}
Assume that $n=3^{\alpha_{1}}11^{\alpha_{2}}p_{3}^{\alpha_{3}}p_{4}^{\alpha_{4}}$ is an odd deficient-perfect number with deficient divisor $d=3^{\beta_{1}}11^{\beta_{2}}p_{3}^{\beta_{3}}p_{4}^{\beta_{4}}$. By (\ref{Eq1}), we have
\begin{equation}\label{4.1}
\frac{3^{\alpha_{1}+1}-1}{2}\cdot \frac{11^{\alpha_{2}+1}-1}{10}\cdot \frac{p_{3}^{\alpha_{3}+1}-1}{p_{3}-1}\cdot \frac{p_{4}^{\alpha_{4}+1}-1}{p_{4}-1}=2\cdot 3^{\alpha_{1}}11^{\alpha_{2}}p_{3}^{\alpha_{3}}p_{4}^{\alpha_{4}}-3^{\beta_{1}}11^{\beta_{2}}p_{3}^{\beta_{3}}p_{4}^{\beta_{4}}.
\end{equation}
If $p_{3}\geqslant 199$, then
$$2=\frac{\sigma(n)}{n}+\frac{d}{n}<\frac{3}{2}\cdot \frac{11}{10}\cdot \frac{199}{198}\cdot \frac{211}{210}+\frac{1}{3}<2,$$
which is clearly false. Thus $p_{3}\leqslant 197$.

{\bf Case 1.} $p_{3}=13$. If $D\geqslant 11$, then
$$2=\frac{\sigma(n)}{n}+\frac{d}{n}<\frac{3}{2}\cdot \frac{11}{10}\cdot \frac{13}{12}\cdot \frac{17}{16}+\frac{1}{11}<2,$$
which is false. If $D=3$, then
$$2=\frac{\sigma(n)}{n}+\frac{d}{n}>\frac{3^{3}-1}{2\cdot 3^{2}}\cdot \frac{11^{3}-1}{10\cdot 11^{2}}\cdot \frac{13^{3}-1}{12\cdot 13^{2}}+\frac{1}{3}>2,$$
which is absurd. Thus $D=9$. If $p_{4}\geqslant 19$, then
$$2=\frac{\sigma(n)}{n}+\frac{d}{n}<\frac{3}{2}\cdot \frac{11}{10}\cdot \frac{13}{12}\cdot \frac{19}{18}+\frac{1}{9}<2,$$
which is impossible Thus $p_{4}=17$. Since $\rm{ord}_{17}(3)=\rm{ord}_{17}(11)=16$ and $\rm{ord}_{17}(13)=4$, we know that the equality (\ref{4.1}) cannot hold.

{\bf Case 2.} $p_{3}\geqslant 17$. If $D\geqslant 9$, then
$$2=\frac{\sigma(n)}{n}+\frac{d}{n}<\frac{3}{2}\cdot \frac{11}{10}\cdot \frac{17}{16}\cdot \frac{19}{18}+\frac{1}{9}<2,$$
which is impossible. Thus $D=3$. By (\ref{4.1}), we have $\alpha_{1}\geqslant 4$ and $\alpha_{2}\geqslant 4$. If $p_{3}\leqslant 71$, then
$$2=\frac{\sigma(n)}{n}+\frac{d}{n}>\frac{3^{5}-1}{2\cdot 3^{4}}\cdot \frac{11^{5}-1}{10\cdot 11^{4}}\cdot \frac{71^{3}-1}{70\cdot 71^{2}}+\frac{1}{3}>2,$$
which is clearly false. Thus $p_{3}\geqslant 73$.

If $\alpha_{1}=6$, then $p_{4}=1093$ and $p_{3}\leqslant 103$. Otherwise, if $p_{3}\geqslant 107$, then
$$2=\frac{\sigma(n)}{n}+\frac{d}{n}<\frac{3^{7}-1}{2\cdot 3^{6}}\cdot \frac{11}{10}\cdot \frac{107}{106}\cdot \frac{1093}{1092}+\frac{1}{3}<2,$$
which is a contradiction. However,
$$2=\frac{\sigma(n)}{n}+\frac{d}{n}\geqslant \frac{3^{7}-1}{2\cdot 3^{6}}\cdot \frac{11^{5}-1}{10\cdot 11^{4}}\cdot \frac{103^{3}-1}{102\cdot 103^{2}}\cdot \frac{1093^{3}-1}{1092\cdot 1093^{2}}+\frac{1}{3}>2,$$
which is also a contradiction. Thus $\alpha_{1}\neq 6$. By (\ref{4.1}), we have $\alpha_{1}\neq 8$ and $\alpha_{1}\neq 10$. If $p_{3}\geqslant 139$, then $\alpha_{1}\neq 4$. Otherwise, if $\alpha_{1}=4$, then
$$2=\frac{\sigma(n)}{n}+\frac{d}{n}<\frac{3^{5}-1}{2\cdot 3^{4}}\cdot \frac{11}{10}\cdot \frac{139}{138}\cdot \frac{149}{148}+\frac{1}{3}<2,$$
which is impossible.

{\bf Subcase 2.1 } $p_{3}\in\{73, 79, 83, 89, 97\}$. If $\alpha_{1}\geqslant 12$, then
$$2=\frac{\sigma(n)}{n}+\frac{d}{n}\geqslant \frac{3^{13}-1}{2\cdot 3^{12}}\cdot \frac{11^{5}-1}{10\cdot 11^{4}}\cdot \frac{97^{3}-1}{96\cdot 97^{2}}+\frac{1}{3}>2,$$
which is clearly false. Thus $\alpha_{1}=4$.

If $p_{3}=73$, then $p_{4}\leqslant 2621$. Otherwise, if $p_{4}\geqslant 2633$, then
$$2=\frac{\sigma(n)}{n}+\frac{d}{n}<\frac{3^{5}-1}{2\cdot 3^{4}}\cdot \frac{11}{10}\cdot \frac{73}{72}\cdot \frac{2633}{2632}+\frac{1}{3}<2,$$
which is impossible. By (\ref{4.1}), we have $\alpha_{2}\geqslant 6$. If $\alpha_{3}=2$, then $p_{4}=1801$ and
$$2=\frac{\sigma(n)}{n}+\frac{d}{n}\geqslant \frac{3^{5}-1}{2\cdot 3^{4}}\cdot \frac{11^{7}-1}{10\cdot 11^{6}}\cdot \frac{73^{3}-1}{72\cdot 73^{2}}\cdot \frac{1801^{3}-1}{1800\cdot 1801^{2}}+\frac{1}{3}>2,$$
which is a contradiction. Thus $\alpha_{3}\geqslant 4$ and
$$2=\frac{\sigma(n)}{n}+\frac{d}{n}\geqslant \frac{3^{5}-1}{2\cdot 3^{4}}\cdot \frac{11^{7}-1}{10\cdot 11^{6}}\cdot \frac{73^{5}-1}{72\cdot 73^{4}}\cdot \frac{2621^{3}-1}{2620\cdot 2621^{2}}+\frac{1}{3}>2,$$
which is also a contradiction.

If $p_{3}=79$, then $p_{4}\leqslant 691$. Otherwise, if $p_{4}\geqslant 701$, then
$$2=\frac{\sigma(n)}{n}+\frac{d}{n}<\frac{3^{5}-1}{2\cdot 3^{4}}\cdot \frac{11}{10}\cdot \frac{79}{78}\cdot \frac{701}{700}+\frac{1}{3}<2,$$
which is impossible. By (\ref{4.1}), we have $\alpha_{2}\geqslant 6$ and
$$2=\frac{\sigma(n)}{n}+\frac{d}{n}\geqslant \frac{3^{5}-1}{2\cdot 3^{4}}\cdot \frac{11^{7}-1}{10\cdot 11^{6}}\cdot \frac{79^{3}-1}{78\cdot 79^{2}}\cdot \frac{691^{3}-1}{690\cdot 691^{2}}+\frac{1}{3}>2,$$
which is a contradiction.

If $p_{3}=83$, then $p_{4}\leqslant 487$. Otherwise, if $p_{4}\geqslant 491$, then
$$2=\frac{\sigma(n)}{n}+\frac{d}{n}<\frac{3^{5}-1}{2\cdot 3^{4}}\cdot \frac{11}{10}\cdot \frac{83}{82}\cdot \frac{491}{490}+\frac{1}{3}<2,$$
which is impossible. By (\ref{4.1}), we have $\alpha_{2}\geqslant 6$ and
$$2=\frac{\sigma(n)}{n}+\frac{d}{n}\geqslant \frac{3^{5}-1}{2\cdot 3^{4}}\cdot \frac{11^{7}-1}{10\cdot 11^{6}}\cdot \frac{83^{3}-1}{82\cdot 83^{2}}\cdot \frac{487^{3}-1}{486\cdot 487^{2}}+\frac{1}{3}>2,$$
which is a contradiction.

If $p_{3}=89$, then $p_{4}\leqslant 347$. Otherwise, if $p_{4}\geqslant 349$, then
$$2=\frac{\sigma(n)}{n}+\frac{d}{n}<\frac{3^{5}-1}{2\cdot 3^{4}}\cdot \frac{11}{10}\cdot \frac{89}{88}\cdot \frac{349}{348}+\frac{1}{3}<2,$$
which is impossible. By (\ref{4.1}), we have $\alpha_{2}\geqslant 6$ and
$$2=\frac{\sigma(n)}{n}+\frac{d}{n}\geqslant \frac{3^{5}-1}{2\cdot 3^{4}}\cdot \frac{11^{7}-1}{10\cdot 11^{6}}\cdot \frac{89^{3}-1}{88\cdot 89^{2}}\cdot \frac{347^{3}-1}{346\cdot 347^{2}}+\frac{1}{3}>2,$$
which is a contradiction.

If $p_{3}=97$, then $p_{4}\leqslant 257$. Otherwise, if $p_{4}\geqslant 263$, then
$$2=\frac{\sigma(n)}{n}+\frac{d}{n}<\frac{3^{5}-1}{2\cdot 3^{4}}\cdot \frac{11}{10}\cdot \frac{97}{96}\cdot \frac{263}{262}+\frac{1}{3}<2,$$
which is impossible. It follows that
 $$2=\frac{\sigma(n)}{n}+\frac{d}{n}\geqslant \frac{3^{5}-1}{2\cdot 3^{4}}\cdot \frac{11^{5}-1}{10\cdot 11^{4}}\cdot \frac{97^{3}-1}{96\cdot 97^{2}}\cdot \frac{257^{3}-1}{256\cdot 257^{2}}+\frac{1}{3}>2,$$
which is a contradiction.

{\bf Subcase 2.2 } $p_{3}=101$. If $p_{4}\geqslant 10009$, then
$$2=\frac{\sigma(n)}{n}+\frac{d}{n}<\frac{3}{2}\cdot \frac{11}{10}\cdot \frac{101}{100}\cdot \frac{10009}{10008}+\frac{1}{3}<2,$$
which is clearly false. Thus $p_{4}\leqslant 9929$.

If $\alpha_{1}=4$, then $p_{4}\leqslant 233$. Otherwise, if $p_{4}\geqslant 239$, then
$$2=\frac{\sigma(n)}{n}+\frac{d}{n}<\frac{3^{5}-1}{2\cdot 3^{4}}\cdot \frac{11}{10}\cdot \frac{101}{100}\cdot \frac{239}{238}+\frac{1}{3}<2,$$
which is impossible. It follows that
 $$2=\frac{\sigma(n)}{n}+\frac{d}{n}\geqslant \frac{3^{5}-1}{2\cdot 3^{4}}\cdot \frac{11^{5}-1}{10\cdot 11^{4}}\cdot \frac{101^{3}-1}{100\cdot 101^{2}}\cdot \frac{233^{3}-1}{232\cdot 233^{2}}+\frac{1}{3}>2,$$
which is a contradiction. By (\ref{4.1}), we have $\alpha_{1}\geqslant 16$. If $\alpha_{2}=4$, then $p_{4}=3221$ and
$$2=\frac{\sigma(n)}{n}+\frac{d}{n}\geqslant \frac{3^{17}-1}{2\cdot 3^{16}}\cdot \frac{11^{5}-1}{10\cdot 11^{4}}\cdot \frac{101^{3}-1}{100\cdot 101^{2}}\cdot \frac{3221^{3}-1}{3220\cdot 3221^{2}}+\frac{1}{3}>2,$$
which is absurd. Thus $\alpha_{2}\geqslant 6$. Noting that $10303\mid (101^{3}-1)$, we have $\alpha_{3}\geqslant 4$ and
$$2=\frac{\sigma(n)}{n}+\frac{d}{n}\geqslant \frac{3^{17}-1}{2\cdot 3^{16}}\cdot \frac{11^{7}-1}{10\cdot 11^{6}}\cdot \frac{101^{5}-1}{100\cdot 101^{4}}\cdot \frac{9929^{3}-1}{9928\cdot 9929^{2}}+\frac{1}{3}>2,$$
a contradiction.

{\bf Subcase 2.3 } $p_{3}=103$. If $p_{4}\geqslant 3407$, then
$$2=\frac{\sigma(n)}{n}+\frac{d}{n}<\frac{3}{2}\cdot \frac{11}{10}\cdot \frac{103}{102}\cdot \frac{3407}{3406}+\frac{1}{3}<2,$$
which is clearly false. Thus $p_{4}\leqslant 3391$.

If $\alpha_{1}=4$, then $p_{4}\leqslant 223$. Otherwise, if $p_{4}\geqslant 227$, then
$$2=\frac{\sigma(n)}{n}+\frac{d}{n}<\frac{3^{5}-1}{2\cdot 3^{4}}\cdot \frac{11}{10}\cdot \frac{103}{102}\cdot \frac{227}{226}+\frac{1}{3}<2,$$
which is impossible. Noting that $\rm{ord}_{5}(3)=\rm{ord}_{5}(103)=4$ and $3221\mid (11^{5}-1)$, we have $\alpha_{2}\geqslant 6$ and $p_{4}\equiv 1\pmod 5$.
Thus $p_{4}\leqslant 211$ and
$$2=\frac{\sigma(n)}{n}+\frac{d}{n}\geqslant \frac{3^{5}-1}{2\cdot 3^{4}}\cdot \frac{11^{7}-1}{10\cdot 11^{6}}\cdot \frac{103^{3}-1}{102\cdot 103^{2}}\cdot \frac{211^{3}-1}{210\cdot 211^{2}}+\frac{1}{3}>2,$$
 which is absurd. By (\ref{4.1}), we have $\alpha_{1}\geqslant 16$ and $\alpha_{3}\geqslant 4$. If $\alpha_{2}=4$, then $p_{4}=3221$ and
$$2=\frac{\sigma(n)}{n}+\frac{d}{n}\geqslant \frac{3^{17}-1}{2\cdot 3^{16}}\cdot \frac{11^{5}-1}{10\cdot 11^{4}}\cdot \frac{103^{5}-1}{102\cdot 103^{4}}\cdot \frac{3221^{3}-1}{3220\cdot 3221^{2}}+\frac{1}{3}>2,$$
which is a contradiction. Thus $\alpha_{2}\geqslant 6$ and
$$2=\frac{\sigma(n)}{n}+\frac{d}{n}\geqslant \frac{3^{17}-1}{2\cdot 3^{16}}\cdot \frac{11^{7}-1}{10\cdot 11^{6}}\cdot \frac{103^{5}-1}{102\cdot 103^{4}}\cdot \frac{3391^{3}-1}{3390\cdot 3391^{2}}+\frac{1}{3}>2,$$
which is also a contradiction.

{\bf Subcase 2.4 } $p_{3}=107$. If $p_{4}\geqslant 1523$, then
$$2=\frac{\sigma(n)}{n}+\frac{d}{n}<\frac{3}{2}\cdot \frac{11}{10}\cdot \frac{107}{106}\cdot \frac{1523}{1522}+\frac{1}{3}<2,$$
which is clearly false. Thus $p_{4}\leqslant 1511$.

If $\alpha_{1}=4$, then $p_{4}\leqslant 199$. Otherwise, if $p_{4}\geqslant 211$, then
$$2=\frac{\sigma(n)}{n}+\frac{d}{n}<\frac{3^{5}-1}{2\cdot 3^{4}}\cdot \frac{11}{10}\cdot \frac{107}{106}\cdot \frac{211}{210}+\frac{1}{3}<2,$$
which is impossible. It follows that
 $$2=\frac{\sigma(n)}{n}+\frac{d}{n}\geqslant \frac{3^{5}-1}{2\cdot 3^{4}}\cdot \frac{11^{5}-1}{10\cdot 11^{4}}\cdot \frac{107^{3}-1}{106\cdot 107^{2}}\cdot \frac{199^{3}-1}{198\cdot 199^{2}}+\frac{1}{3}>2,$$
which is absurd. By (\ref{4.1}), we have $\alpha_{1}\geqslant 16, \alpha_{2}\geqslant 6$ and
$$2=\frac{\sigma(n)}{n}+\frac{d}{n}\geqslant \frac{3^{17}-1}{2\cdot 3^{16}}\cdot \frac{11^{7}-1}{10\cdot 11^{6}}\cdot \frac{107^{3}-1}{106\cdot 107^{2}}\cdot \frac{1511^{3}-1}{1510\cdot 1511^{2}}+\frac{1}{3}>2,$$
which is a contradiction.

{\bf Subcase 2.5 } $p_{3}=109$. If $p_{4}\geqslant 1201$, then
$$2=\frac{\sigma(n)}{n}+\frac{d}{n}<\frac{3}{2}\cdot \frac{11}{10}\cdot \frac{109}{108}\cdot \frac{1201}{1200}+\frac{1}{3}<2,$$
which is clearly false. Thus $p_{4}\leqslant 1193$.

If $\alpha_{1}=4$, then $p_{4}\leqslant 199$. Otherwise, if $p_{4}\geqslant 211$, then
$$2=\frac{\sigma(n)}{n}+\frac{d}{n}<\frac{3^{5}-1}{2\cdot 3^{4}}\cdot \frac{11}{10}\cdot \frac{109}{108}\cdot \frac{211}{210}+\frac{1}{3}<2,$$
which is impossible. It follows that
 $$2=\frac{\sigma(n)}{n}+\frac{d}{n}\geqslant \frac{3^{5}-1}{2\cdot 3^{4}}\cdot \frac{11^{5}-1}{10\cdot 11^{4}}\cdot \frac{109^{3}-1}{108\cdot 109^{2}}\cdot \frac{199^{3}-1}{198\cdot 199^{2}}+\frac{1}{3}>2,$$
which is absurd. By (\ref{4.1}), we have $\alpha_{1}\geqslant 16, \alpha_{2}\geqslant 6$ and
$$2=\frac{\sigma(n)}{n}+\frac{d}{n}\geqslant \frac{3^{17}-1}{2\cdot 3^{16}}\cdot \frac{11^{7}-1}{10\cdot 11^{6}}\cdot \frac{109^{3}-1}{108\cdot 109^{2}}\cdot \frac{1193^{3}-1}{1192\cdot 1193^{2}}+\frac{1}{3}>2,$$
which is a contradiction.

{\bf Subcase 2.6 } $p_{3}=113$. If $p_{4}\geqslant 863$, then
$$2=\frac{\sigma(n)}{n}+\frac{d}{n}<\frac{3}{2}\cdot \frac{11}{10}\cdot \frac{113}{112}\cdot \frac{863}{862}+\frac{1}{3}<2,$$
which is clearly false. Thus $p_{4}\leqslant 859$.

If $\alpha_{1}=4$, then $p_{4}\leqslant 181$. Otherwise, if $p_{4}\geqslant 191$, then
$$2=\frac{\sigma(n)}{n}+\frac{d}{n}<\frac{3^{5}-1}{2\cdot 3^{4}}\cdot \frac{11}{10}\cdot \frac{113}{112}\cdot \frac{191}{190}+\frac{1}{3}<2,$$
which is impossible. It follows that
 $$2=\frac{\sigma(n)}{n}+\frac{d}{n}\geqslant \frac{3^{5}-1}{2\cdot 3^{4}}\cdot \frac{11^{5}-1}{10\cdot 11^{4}}\cdot \frac{113^{3}-1}{112\cdot 113^{2}}\cdot \frac{181^{3}-1}{180\cdot 181^{2}}+\frac{1}{3}>2,$$
which is absurd. By (\ref{4.1}), we have $\alpha_{1}\geqslant 16, \alpha_{2}\geqslant 6$ and
$$2=\frac{\sigma(n)}{n}+\frac{d}{n}\geqslant \frac{3^{17}-1}{2\cdot 3^{16}}\cdot \frac{11^{7}-1}{10\cdot 11^{6}}\cdot \frac{113^{3}-1}{112\cdot 113^{2}}\cdot \frac{859^{3}-1}{858\cdot 859^{2}}+\frac{1}{3}>2,$$
which is a contradiction.

{\bf Subcase 2.7 } $p_{3}=127$. If $p_{4}\geqslant 467$, then
$$2=\frac{\sigma(n)}{n}+\frac{d}{n}<\frac{3}{2}\cdot \frac{11}{10}\cdot \frac{127}{126}\cdot \frac{467}{466}+\frac{1}{3}<2,$$
which is clearly false. Thus $p_{4}\leqslant 463$.

If $\alpha_{1}=4$, then $p_{4}\leqslant 157$. Otherwise, if $p_{4}\geqslant 163$, then
$$2=\frac{\sigma(n)}{n}+\frac{d}{n}<\frac{3^{5}-1}{2\cdot 3^{4}}\cdot \frac{11}{10}\cdot \frac{127}{126}\cdot \frac{163}{162}+\frac{1}{3}<2,$$
which is impossible. It follows that
 $$2=\frac{\sigma(n)}{n}+\frac{d}{n}\geqslant \frac{3^{5}-1}{2\cdot 3^{4}}\cdot \frac{11^{5}-1}{10\cdot 11^{4}}\cdot \frac{127^{3}-1}{126\cdot 127^{2}}\cdot \frac{157^{3}-1}{156\cdot 157^{2}}+\frac{1}{3}>2,$$
which is absurd. By (\ref{4.1}), we have $\alpha_{1}\geqslant 16, \alpha_{2}\geqslant 6$ and
$$2=\frac{\sigma(n)}{n}+\frac{d}{n}\geqslant \frac{3^{17}-1}{2\cdot 3^{16}}\cdot \frac{11^{7}-1}{10\cdot 11^{6}}\cdot \frac{127^{3}-1}{126\cdot 127^{2}}\cdot \frac{463^{3}-1}{462\cdot 463^{2}}+\frac{1}{3}>2,$$
which is a contradiction.

{\bf Subcase 2.8 } $p_{3}=131$. If $p_{4}\geqslant 421$, then
$$2=\frac{\sigma(n)}{n}+\frac{d}{n}<\frac{3}{2}\cdot \frac{11}{10}\cdot \frac{131}{130}\cdot \frac{421}{420}+\frac{1}{3}<2,$$
which is clearly false. Thus $p_{4}\leqslant 419$.

If $\alpha_{1}=4$, then $p_{4}\leqslant 151$. Otherwise, if $p_{4}\geqslant 157$, then
$$2=\frac{\sigma(n)}{n}+\frac{d}{n}<\frac{3^{5}-1}{2\cdot 3^{4}}\cdot \frac{11}{10}\cdot \frac{131}{130}\cdot \frac{157}{156}+\frac{1}{3}<2,$$
which is impossible. It follows that
 $$2=\frac{\sigma(n)}{n}+\frac{d}{n}\geqslant \frac{3^{5}-1}{2\cdot 3^{4}}\cdot \frac{11^{5}-1}{10\cdot 11^{4}}\cdot \frac{131^{3}-1}{130\cdot 131^{2}}\cdot \frac{151^{3}-1}{150\cdot 151^{2}}+\frac{1}{3}>2,$$
which is absurd. By (\ref{4.1}), we have $\alpha_{1}\geqslant 16, \alpha_{2}\geqslant 6$ and
$$2=\frac{\sigma(n)}{n}+\frac{d}{n}\geqslant \frac{3^{17}-1}{2\cdot 3^{16}}\cdot \frac{11^{7}-1}{10\cdot 11^{6}}\cdot \frac{131^{3}-1}{130\cdot 131^{2}}\cdot \frac{419^{3}-1}{418\cdot 419^{2}}+\frac{1}{3}>2,$$
which is a contradiction.

{\bf Subcase 2.9 } $p_{3}=137$. If $p_{4}\geqslant 373$, then
$$2=\frac{\sigma(n)}{n}+\frac{d}{n}<\frac{3}{2}\cdot \frac{11}{10}\cdot \frac{137}{136}\cdot \frac{373}{372}+\frac{1}{3}<2,$$
which is clearly false. Thus $p_{4}\leqslant 367$.

If $\alpha_{1}=4$, then $p_{4}=139$. Otherwise, if $p_{4}\geqslant 149$, then
$$2=\frac{\sigma(n)}{n}+\frac{d}{n}<\frac{3^{5}-1}{2\cdot 3^{4}}\cdot \frac{11}{10}\cdot \frac{137}{136}\cdot \frac{149}{148}+\frac{1}{3}<2,$$
which is impossible.It follows that
 $$2=\frac{\sigma(n)}{n}+\frac{d}{n}\geqslant \frac{3^{5}-1}{2\cdot 3^{4}}\cdot \frac{11^{5}-1}{10\cdot 11^{4}}\cdot \frac{137^{3}-1}{136\cdot 137^{2}}\cdot \frac{139^{3}-1}{138\cdot 139^{2}}+\frac{1}{3}>2,$$
which is absurd. By (\ref{4.1}), we have $\alpha_{1}\geqslant 16, \alpha_{2}\geqslant 6$ and
$$2=\frac{\sigma(n)}{n}+\frac{d}{n}\geqslant \frac{3^{17}-1}{2\cdot 3^{16}}\cdot \frac{11^{7}-1}{10\cdot 11^{6}}\cdot \frac{137^{3}-1}{136\cdot 137^{2}}\cdot \frac{367^{3}-1}{366\cdot 367^{2}}+\frac{1}{3}>2,$$
which is a contradiction.

{\bf Subcase 2.10 } $p_{3}=139$. If $p_{4}\geqslant 359$, then
$$2=\frac{\sigma(n)}{n}+\frac{d}{n}<\frac{3}{2}\cdot \frac{11}{10}\cdot \frac{139}{138}\cdot \frac{359}{358}+\frac{1}{3}<2,$$
which is clearly false. Thus $p_{4}\leqslant 353$. By (\ref{4.1}), we have $\alpha_{1}\geqslant 16, \alpha_{2}\geqslant 6$ and
$$2=\frac{\sigma(n)}{n}+\frac{d}{n}\geqslant \frac{3^{17}-1}{2\cdot 3^{16}}\cdot \frac{11^{7}-1}{10\cdot 11^{6}}\cdot \frac{139^{3}-1}{138\cdot 139^{2}}\cdot \frac{353^{3}-1}{352\cdot 353^{2}}+\frac{1}{3}>2,$$
which is a contradiction.

{\bf Subcase 2.11 } $p_{3}\in\{149, 151\}$. If $p_{4}\geqslant 307$, then
$$2=\frac{\sigma(n)}{n}+\frac{d}{n}<\frac{3}{2}\cdot \frac{11}{10}\cdot \frac{149}{148}\cdot \frac{307}{306}+\frac{1}{3}<2,$$
which is clearly false. Thus $p_{4}\leqslant 293$. By (\ref{4.1}), we have $\alpha_{1}\geqslant 16, \alpha_{2}\geqslant 6$ and
$$2=\frac{\sigma(n)}{n}+\frac{d}{n}\geqslant \frac{3^{17}-1}{2\cdot 3^{16}}\cdot \frac{11^{7}-1}{10\cdot 11^{6}}\cdot \frac{151^{3}-1}{150\cdot 151^{2}}\cdot \frac{293^{3}-1}{292\cdot 293^{2}}+\frac{1}{3}>2,$$
which is a contradiction.

{\bf Subcase 2.12 } $p_{3}=157$. If $p_{4}\geqslant 277$, then
$$2=\frac{\sigma(n)}{n}+\frac{d}{n}<\frac{3}{2}\cdot \frac{11}{10}\cdot \frac{157}{156}\cdot \frac{277}{276}+\frac{1}{3}<2,$$
which is clearly false. Thus $p_{4}\leqslant 271$. By (\ref{4.1}), we have $\alpha_{1}\geqslant 16, \alpha_{2}\geqslant 6$ and
$$2=\frac{\sigma(n)}{n}+\frac{d}{n}\geqslant \frac{3^{17}-1}{2\cdot 3^{16}}\cdot \frac{11^{7}-1}{10\cdot 11^{6}}\cdot \frac{157^{3}-1}{156\cdot 157^{2}}\cdot \frac{271^{3}-1}{270\cdot 271^{2}}+\frac{1}{3}>2,$$
which is a contradiction.

{\bf Subcase 2.13 } $p_{3}=163$. If $p_{4}\geqslant 263$, then
$$2=\frac{\sigma(n)}{n}+\frac{d}{n}<\frac{3}{2}\cdot \frac{11}{10}\cdot \frac{163}{162}\cdot \frac{263}{262}+\frac{1}{3}<2,$$
which is clearly false. Thus $p_{4}\leqslant 257$. By (\ref{4.1}), we have $\alpha_{1}\geqslant 16, \alpha_{2}\geqslant 6$ and
$$2=\frac{\sigma(n)}{n}+\frac{d}{n}\geqslant \frac{3^{17}-1}{2\cdot 3^{16}}\cdot \frac{11^{7}-1}{10\cdot 11^{6}}\cdot \frac{163^{3}-1}{162\cdot 163^{2}}\cdot \frac{257^{3}-1}{256\cdot 257^{2}}+\frac{1}{3}>2,$$
which is a contradiction.

{\bf Subcase 2.14 } $p_{3}=167$. If $p_{4}\geqslant 251$, then
$$2=\frac{\sigma(n)}{n}+\frac{d}{n}<\frac{3}{2}\cdot \frac{11}{10}\cdot \frac{167}{166}\cdot \frac{251}{250}+\frac{1}{3}<2,$$
which is clearly false. Thus $p_{4}\leqslant 241$. By (\ref{4.1}), we have $\alpha_{1}\geqslant 16, \alpha_{2}\geqslant 6$ and
$$2=\frac{\sigma(n)}{n}+\frac{d}{n}\geqslant \frac{3^{17}-1}{2\cdot 3^{16}}\cdot \frac{11^{7}-1}{10\cdot 11^{6}}\cdot \frac{167^{3}-1}{166\cdot 167^{2}}\cdot \frac{241^{3}-1}{240\cdot 241^{2}}+\frac{1}{3}>2,$$
which is a contradiction.

{\bf Subcase 2.15 } $p_{3}=173$. If $p_{4}\geqslant 239$, then
$$2=\frac{\sigma(n)}{n}+\frac{d}{n}<\frac{3}{2}\cdot \frac{11}{10}\cdot \frac{173}{172}\cdot \frac{239}{238}+\frac{1}{3}<2,$$
which is clearly false. Thus $p_{4}\leqslant 233$. By (\ref{4.1}), we have $\alpha_{1}\geqslant 16, \alpha_{2}\geqslant 6$ and
$$2=\frac{\sigma(n)}{n}+\frac{d}{n}\geqslant \frac{3^{17}-1}{2\cdot 3^{16}}\cdot \frac{11^{7}-1}{10\cdot 11^{6}}\cdot \frac{173^{3}-1}{172\cdot 173^{2}}\cdot \frac{233^{3}-1}{232\cdot 233^{2}}+\frac{1}{3}>2,$$
which is a contradiction.

{\bf Subcase 2.16 } $p_{3}=179$. If $p_{4}\geqslant 227$, then
$$2=\frac{\sigma(n)}{n}+\frac{d}{n}<\frac{3}{2}\cdot \frac{11}{10}\cdot \frac{179}{178}\cdot \frac{227}{226}+\frac{1}{3}<2,$$
which is clearly false. Thus $p_{4}\leqslant 223$. By (\ref{4.1}), we have $\alpha_{1}\geqslant 16, \alpha_{2}\geqslant 6$ and
$$2=\frac{\sigma(n)}{n}+\frac{d}{n}\geqslant \frac{3^{17}-1}{2\cdot 3^{16}}\cdot \frac{11^{7}-1}{10\cdot 11^{6}}\cdot \frac{179^{3}-1}{178\cdot 179^{2}}\cdot \frac{223^{3}-1}{222\cdot 223^{2}}+\frac{1}{3}>2,$$
which is a contradiction.

{\bf Subcase 2.17 } $p_{3}=181$. If $p_{4}\geqslant 223$, then
$$2=\frac{\sigma(n)}{n}+\frac{d}{n}<\frac{3}{2}\cdot \frac{11}{10}\cdot \frac{181}{180}\cdot \frac{223}{222}+\frac{1}{3}<2,$$
which is clearly false. Thus $p_{4}\leqslant 211$. By (\ref{4.1}), we have $\alpha_{1}\geqslant 16, \alpha_{2}\geqslant 6$ and
$$2=\frac{\sigma(n)}{n}+\frac{d}{n}\geqslant \frac{3^{17}-1}{2\cdot 3^{16}}\cdot \frac{11^{7}-1}{10\cdot 11^{6}}\cdot \frac{181^{3}-1}{180\cdot 181^{2}}\cdot \frac{211^{3}-1}{210\cdot 211^{2}}+\frac{1}{3}>2,$$
which is a contradiction.

{\bf Subcase 2.18 } $p_{3}\in\{191, 193, 197\}$. If $p_{4}\geqslant 211$, then
$$2=\frac{\sigma(n)}{n}+\frac{d}{n}<\frac{3}{2}\cdot \frac{11}{10}\cdot \frac{191}{190}\cdot \frac{211}{210}+\frac{1}{3}<2,$$
which is clearly false. Thus $p_{4}\leqslant 199$. Noting that $\rm{ord}_{5}(3)=4$ and ${\rm{ord}}_{5}(p_{4})$ are all even, we have $5\mid (\alpha_{2}+1)$ and $(11^{5}-1)\mid(11^{\alpha_{2}+1}-1)$ or $5\mid (\alpha_{3}+1)$ and $(191^{5}-1)\mid(191^{\alpha_{3}+1}-1)$. However, $3221\mid (11^{5}-1)$ and $1871\mid (191^{5}-1)$, a contradiction.

This completes the proof of Lemma \ref{lem4.1}.
\end{proof}

\begin{lemma}\label{lem4.2}
There is no odd deficient perfect number of the form $n=3^{\alpha_{1}}13^{\alpha_{2}}p_{3}^{\alpha_{3}}p_{4}^{\alpha_{4}}$.
\end{lemma}
\begin{proof}
Assume that $n=3^{\alpha_{1}}13^{\alpha_{2}}p_{3}^{\alpha_{3}}p_{4}^{\alpha_{4}}$ is an odd deficient-perfect number with deficient divisor $d=3^{\beta_{1}}13^{\beta_{2}}p_{3}^{\beta_{3}}p_{4}^{\beta_{4}}$. If $p_{3}\geqslant 79$, then
$$2=\frac{\sigma(n)}{n}+\frac{d}{n}<\frac{3}{2}\cdot \frac{13}{12}\cdot \frac{79}{78}\cdot \frac{83}{82}+\frac{1}{3}<2,$$
which is false. Thus $p_{3}\leqslant 73$. If $D\geqslant 9$, then
$$2=\frac{\sigma(n)}{n}+\frac{d}{n}<\frac{3}{2}\cdot \frac{13}{12}\cdot \frac{17}{16}\cdot \frac{19}{18}+\frac{1}{9}<2,$$
which is absurd. Thus $D=3$. By (\ref{Eq1}), we have
\begin{equation}\label{4.2}
\frac{3^{\alpha_{1}+1}-1}{2}\cdot \frac{13^{\alpha_{2}+1}-1}{12}\cdot \frac{p_{3}^{\alpha_{3}+1}-1}{p_{3}-1}\cdot \frac{p_{4}^{\alpha_{4}+1}-1}{p_{4}-1}=5\cdot 3^{\alpha_{1}-1}13^{\alpha_{2}}p_{3}^{\alpha_{3}}p_{4}^{\alpha_{4}}.
\end{equation}
Since $11\mid (3^{5}-1)$, we have $\alpha_{1}\neq 4$. If $p_{3}\geqslant 31$, then $\alpha_{1}\geqslant 6$. Otherwise, if $\alpha_{1}=2$, then
$$2=\frac{\sigma(n)}{n}+\frac{d}{n}<\frac{3^{3}-1}{2\cdot 3^{2}}\cdot \frac{13}{12}\cdot \frac{31}{30}\cdot \frac{37}{36}+\frac{1}{3}<2,$$
which is impossible.  If $p_{3}\geqslant 43$, then $p_{4}\leqslant 557$. Otherwise, if $p_{4}\geqslant 563$, then
$$2=\frac{\sigma(n)}{n}+\frac{d}{n}<\frac{3}{2}\cdot \frac{13}{12}\cdot \frac{43}{42}\cdot \frac{563}{562}+\frac{1}{3}<2,$$
which is a contradiction. Now we divide into the following seven cases according to $p_{3}$.

{\bf Case 1.} $p_{3}\in\{17, 19, 23, 29, 31, 37\}$. If $\alpha_{1}\geqslant 6$, then
$$2=\frac{\sigma(n)}{n}+\frac{d}{n}>\frac{3^{7}-1}{2\cdot 3^{6}}\cdot \frac{13^{3}-1}{12\cdot 13^{2}}\cdot \frac{37^{3}-1}{36\cdot 37^{2}}+\frac{1}{3}>2,$$
which is impossible. Thus $\alpha_{1}=2$ and $p_{3}\in\{17, 19, 23, 29\}$.

If $p_{3}=17$, then $p_{4}\leqslant 409$. Otherwise, if $p_{4}\geqslant 419$, then
$$2=\frac{\sigma(n)}{n}+\frac{d}{n}<\frac{3^{3}-1}{2\cdot 3^{2}}\cdot \frac{13}{12}\cdot \frac{17}{16}\cdot \frac{419}{418}+\frac{1}{3}<2,$$
which is impossible. Since $\rm{ord}_{17}(13)=4$ and $\rm{ord}_{5}(13)=\rm{ord}_{5}(17)=4$, we have $p_{4}\equiv 1\pmod{85}$. Thus $p_{4}\geqslant 1021$, a contradiction.

If $p_{3}=19$, then $p_{4}\leqslant 109$. Otherwise, if $p_{4}\geqslant 113$, then
$$2=\frac{\sigma(n)}{n}+\frac{d}{n}<\frac{3^{3}-1}{2\cdot 3^{2}}\cdot \frac{13}{12}\cdot \frac{19}{18}\cdot \frac{113}{112}+\frac{1}{3}<2,$$
which is false. Since $\rm{ord}_{5}(13)=4$ and $\rm{ord}_{5}(19)=2$, we have $p_{4}\leqslant 101$ and
 $$2=\frac{\sigma(n)}{n}+\frac{d}{n}\geqslant \frac{3^{3}-1}{2\cdot 3^{2}}\cdot \frac{13^{3}-1}{12\cdot 13^{2}}\cdot \frac{19^{3}-1}{18\cdot 19^{2}}\cdot \frac{101^{3}-1}{100\cdot 101^{2}}+\frac{1}{3}>2,$$
which is impossible.

If $p_{3}=23$, then $p_{4}\leqslant 53$. Otherwise, if $p_{4}\geqslant 59$, then
$$2=\frac{\sigma(n)}{n}+\frac{d}{n}<\frac{3^{3}-1}{2\cdot 3^{2}}\cdot \frac{13}{12}\cdot \frac{23}{22}\cdot \frac{59}{58}+\frac{1}{3}<2,$$
which is false. Since $\rm{ord}_{5}(13)=\rm{ord}_{5}(23)=4$, we have $p_{4}\leqslant 41$ and
 $$2=\frac{\sigma(n)}{n}+\frac{d}{n}\geqslant \frac{3^{3}-1}{2\cdot 3^{2}}\cdot \frac{13^{3}-1}{12\cdot 13^{2}}\cdot \frac{23^{3}-1}{22\cdot 23^{2}}\cdot \frac{41^{3}-1}{40\cdot 41^{2}}+\frac{1}{3}>2,$$
which is impossible.

If $p_{3}=29$, then $p_{4}=31$. Otherwise, if $p_{4}\geqslant 37$, then
$$2=\frac{\sigma(n)}{n}+\frac{d}{n}<\frac{3^{3}-1}{2\cdot 3^{2}}\cdot \frac{13}{12}\cdot \frac{29}{28}\cdot \frac{37}{36}+\frac{1}{3}<2,$$
which is a contradiction. However,
 $$2=\frac{\sigma(n)}{n}+\frac{d}{n}\geqslant \frac{3^{3}-1}{2\cdot 3^{2}}\cdot \frac{13^{3}-1}{12\cdot 13^{2}}\cdot \frac{29^{3}-1}{28\cdot 29^{2}}\cdot \frac{31^{3}-1}{30\cdot 31^{2}}+\frac{1}{3}>2,$$
which is also a contradiction.

{\bf Case 2.} $p_{3}=41$. If $p_{4}\geqslant 1601$, then
$$2=\frac{\sigma(n)}{n}+\frac{d}{n}<\frac{3}{2}\cdot \frac{13}{12}\cdot \frac{41}{40}\cdot \frac{1601}{1600}+\frac{1}{3}<2,$$
which is impossible.  Since $\rm{ord}_{5}(3)=\rm{ord}_{5}(13)=4$ and $579281\mid (41^{5}-1)$, we have $p_{4}\equiv 1\pmod 5$ and $5\mid (\alpha_{4}+1)$. Thus $p_{4}\leqslant 1571$ and $\alpha_{3}\neq 4$. By (\ref{4.2}), we have $\alpha_{3}\geqslant 6$ and $\alpha_{1}\geqslant 12$. If $\alpha_{2}=2$, then $p_{4}=61$ and $(61^{5}-1)\mid (61^{\alpha_{4}+1}-1)$. However, $131\mid (61^{5}-1)$, a contradiction. Thus $\alpha_{2}\geqslant 4$ and
$$2=\frac{\sigma(n)}{n}+\frac{d}{n}\geqslant \frac{3^{13}-1}{2\cdot 3^{12}}\cdot \frac{13^{5}-1}{12\cdot 13^{4}}\cdot \frac{41^{7}-1}{40\cdot 41^{6}}\cdot\frac{1571^{5}-1}{1570\cdot 1571^{4}}+\frac{1}{3}>2,$$
which is also a contradiction.

{\bf Case 3.} $p_{3}\in \{43, 47, 53, 59\}$. Since $\rm{ord}_{5}(3)=\rm{ord}_{5}(13)$ and ${\rm{ord}}_{5}(p_{3})$ are all even, we have $p_{4}\equiv 1\pmod 5$ and $5\mid (\alpha_{4}+1)$. By (\ref{4.2}), we have $\alpha_{3}\geqslant 4$ and $\alpha_{1}\geqslant 12$. If $\alpha_{2}=2$, then $p_{4}=61$ and $(61^{5}-1)\mid (61^{\alpha_{4}+1}-1)$. However, $131\mid (61^{5}-1)$, a contradiction. Thus $\alpha_{2}\geqslant 4$.

If $p_{3}=43$, then $p_{4}\leqslant 541$. Otherwise, if $p_{4}\geqslant 563$, then
$$2=\frac{\sigma(n)}{n}+\frac{d}{n}<\frac{3}{2}\cdot \frac{13}{12}\cdot \frac{43}{42}\cdot \frac{563}{562}+\frac{1}{3}<2,$$
which is a contradiction. However,
$$2=\frac{\sigma(n)}{n}+\frac{d}{n}\geqslant \frac{3^{13}-1}{2\cdot 3^{12}}\cdot \frac{13^{5}-1}{12\cdot 13^{4}}\cdot \frac{43^{5}-1}{42\cdot 43^{4}}\cdot\frac{541^{5}-1}{540\cdot 541^{4}}+\frac{1}{3}>2,$$
which is also a contradiction.

If $p_{3}=47$, then $p_{4}\leqslant 251$. Otherwise, if $p_{4}\geqslant 263$, then
$$2=\frac{\sigma(n)}{n}+\frac{d}{n}<\frac{3}{2}\cdot \frac{13}{12}\cdot \frac{47}{46}\cdot \frac{263}{262}+\frac{1}{3}<2,$$
which is a contradiction. However,
$$2=\frac{\sigma(n)}{n}+\frac{d}{n}\geqslant \frac{3^{13}-1}{2\cdot 3^{12}}\cdot \frac{13^{5}-1}{12\cdot 13^{4}}\cdot \frac{47^{5}-1}{46\cdot 47^{4}}\cdot\frac{251^{5}-1}{250\cdot 251^{4}}+\frac{1}{3}>2,$$
which is also a contradiction.

If $p_{3}=53$, then $p_{4}\leqslant 151$. Otherwise, if $p_{4}\geqslant 163$, then
$$2=\frac{\sigma(n)}{n}+\frac{d}{n}<\frac{3}{2}\cdot \frac{13}{12}\cdot \frac{53}{52}\cdot \frac{163}{162}+\frac{1}{3}<2,$$
which is a contradiction. However,
$$2=\frac{\sigma(n)}{n}+\frac{d}{n}\geqslant \frac{3^{13}-1}{2\cdot 3^{12}}\cdot \frac{13^{5}-1}{12\cdot 13^{4}}\cdot \frac{53^{5}-1}{52\cdot 53^{4}}\cdot\frac{151^{5}-1}{150\cdot 151^{4}}+\frac{1}{3}>2,$$
which is also a contradiction.

If $p_{3}=59$, then $p_{4}\leqslant 101$. Otherwise, if $p_{4}\geqslant 127$, then
$$2=\frac{\sigma(n)}{n}+\frac{d}{n}<\frac{3}{2}\cdot \frac{13}{12}\cdot \frac{59}{58}\cdot \frac{127}{126}+\frac{1}{3}<2,$$
which is a contradiction. However,
$$2=\frac{\sigma(n)}{n}+\frac{d}{n}\geqslant \frac{3^{13}-1}{2\cdot 3^{12}}\cdot \frac{13^{5}-1}{12\cdot 13^{4}}\cdot \frac{59^{5}-1}{58\cdot 59^{4}}\cdot\frac{101^{5}-1}{100\cdot 101^{4}}+\frac{1}{3}>2,$$
which is also a contradiction.

{\bf Case 4.} $p_{3}=61$. If $p_{4}\geqslant 127$, then
$$2=\frac{\sigma(n)}{n}+\frac{d}{n}<\frac{3}{2}\cdot \frac{13}{12}\cdot \frac{61}{60}\cdot \frac{127}{126}+\frac{1}{3}<2,$$
which is false. Thus $p_{4}\leqslant 113$. Since $\rm{ord}_{5}(3)=\rm{ord}_{5}(13)=4, \rm{ord}_{25}(61)=\rm{ord}_{25}(71)=\rm{ord}_{25}(101)=5,$
we have $(61^{5}-1)\mid (61^{\alpha_{3}+1}-1)$ or $(71^{5}-1)\mid (71^{\alpha_{4}+1}-1)$ or $(101^{5}-1)\mid (101^{\alpha_{4}+1}-1)$. However,
$131\mid (61^{5}-1), 11\mid (71^{5}-1), 31\mid (101^{5}-1)$, a contradiction.

{\bf Case 5.} $p_{3}=67$. If $p_{4}\geqslant 101$, then
$$2=\frac{\sigma(n)}{n}+\frac{d}{n}<\frac{3}{2}\cdot \frac{13}{12}\cdot \frac{67}{66}\cdot \frac{101}{100}+\frac{1}{3}<2,$$
which is false. Thus $p_{4}\leqslant 97$. Observing that $\rm{ord}_{5}(3)=\rm{ord}_{5}(13)=\rm{ord}_{5}(67)=4$, we have $p_{4}\equiv 1\pmod 5$ and $5\mid(\alpha_{4}+1)$. Thus $p_{4}=71$ and $(71^{5}-1)\mid (71^{\alpha_{4}+1}-1)$. However, $11\mid (71^{5}-1)$, a contradiction.

{\bf Case 6.} $p_{3}=71$. If $p_{4}\geqslant 97$, then
$$2=\frac{\sigma(n)}{n}+\frac{d}{n}<\frac{3}{2}\cdot \frac{13}{12}\cdot \frac{71}{70}\cdot \frac{97}{96}+\frac{1}{3}<2,$$
which is false. Thus $p_{4}\leqslant 89$. Since $\rm{ord}_{5}(3)=\rm{ord}_{5}(13)=4$ and ${\rm{ord}}_{5}(p_{4})$ are all even, we have $5\mid(\alpha_{3}+1)$ and $(71^{5}-1)\mid (71^{\alpha_{3}+1}-1)$. However, $11\mid (71^{5}-1)$, a contradiction.

{\bf Case 7.} $p_{3}=73$. If $p_{4}\geqslant 89$, then
$$2=\frac{\sigma(n)}{n}+\frac{d}{n}<\frac{3}{2}\cdot \frac{13}{12}\cdot \frac{73}{72}\cdot \frac{89}{88}+\frac{1}{3}<2,$$
a contradiction. Thus $p_{4}\in\{79, 83\}$. Observing that $\rm{ord}_{5}(3)=\rm{ord}_{5}(13)=\rm{ord}_{5}(73)=4$ and ${\rm{ord}}_{5}(p_{4})$ are all even, we see that the case cannot hold.

This completes the proof of Lemma \ref{lem4.2}.
\end{proof}

\section{Proofs}

\begin{proof}[Proof of Theorem 1] Assume that $n=p_{1}^{\alpha_{1}}p_{2}^{\alpha_{2}}p_{3}^{\alpha_{3}}p_{4}^{\alpha_{4}}$ is an odd deficient-perfect number with deficient divisor $d=p_{1}^{\beta_{1}}p_{2}^{\beta_{2}}p_{3}^{\beta_{3}}p_{4}^{\beta_{4}}$. By \cite{M}, we need to consider $p_{1}=3$ and $p_{2}\in\{5, 7, 11, 13, 17\}$.

{\bf Case 1.} $p_{2}=5$. If $D=3$, then
$$2=\frac{\sigma(n)}{n}+\frac{d}{n}>\frac{3^{3}-1}{2\cdot 3^{2}}\cdot \frac{5^{3}-1}{4\cdot 5^{2}}+\frac{1}{3}>2,$$
which is clearly false. If $D=5$, then by (\ref{Eq1}), we have
$$\frac{3^{\alpha_{1}+1}-1}{2}\cdot\frac{5^{\alpha_{2}+1}-1}{4}\cdot \frac{p_{3}^{\alpha_{3}+1}-1}{p_{3}-1}\cdot \frac{p_{4}^{\alpha_{4}+1}-1}{p_{4}-1}= 3^{\alpha_{1}+2}5^{\alpha_{2}-1}p_{3}^{\alpha_{3}}p_{4}^{\alpha_{4}}.$$
If $\alpha_{1}\geqslant 4$, then
$$2=\frac{\sigma(n)}{n}+\frac{d}{n}>\frac{3^{5}-1}{2\cdot 3^{4}}\cdot \frac{5^{3}-1}{4\cdot 5^{2}}+\frac{1}{5}>2,$$
which is absurd. Thus $\alpha_{1}=2$. If $\alpha_{2}\geqslant 4$, then
$$2=\frac{\sigma(n)}{n}+\frac{d}{n}>\frac{3^{3}-1}{2\cdot 3^{2}}\cdot \frac{5^{5}-1}{4\cdot 5^{4}}+\frac{1}{5}>2,$$
which is impossible. Thus $\alpha_{2}=2, p_{3}=13$ and $p_{4}=31$. However,
$$2=\frac{\sigma(n)}{n}+\frac{d}{n}\geqslant \frac{3^{3}-1}{2\cdot 3^{2}}\cdot \frac{5^{3}-1}{4\cdot 5^{2}}\cdot \frac{13^{3}-1}{12\cdot 13^{2}}\cdot \frac{31^{3}-1}{30\cdot 31^{2}}+\frac{1}{5}>2,$$
a contradiction. If $p_{3}=7$, then
$$2=\frac{\sigma(n)}{n}+\frac{d}{n}>\frac{3^{3}-1}{2\cdot 3^{2}}\cdot \frac{5^{3}-1}{4\cdot 5^{2}}\cdot \frac{7^{3}-1}{6\cdot 7^{2}}>2,$$
which is a contradiction. Thus $D\geqslant 9$. If $p_{3}\geqslant 271$, then
$$2=\frac{\sigma(n)}{n}+\frac{d}{n}<\frac{3}{2}\cdot \frac{5}{4}\cdot \frac{271}{270}\cdot \frac{277}{276}+\frac{1}{9}<2,$$
which is false. Thus $11\leqslant p_{3}\leqslant 269$. If $19\leqslant p_{3}\leqslant 61$, then $D\geqslant 15$. Otherwise, if $D=9$, then $\alpha_{1}\geqslant 6$ and
$$2=\frac{\sigma(n)}{n}+\frac{d}{n}>\frac{3^{7}-1}{2\cdot 3^{6}}\cdot \frac{5^{3}-1}{4\cdot 5^{2}}\cdot \frac{61^{3}-1}{60\cdot 61^{2}}+\frac{1}{9}>2,$$
which is impossible. Now we divide into the following five subcases according to $p_{3}$.

{\bf Subcase 1.1 } $p_{3}\in \{11, 13, 17\}$. By Lemma \ref{lem2.1}-Lemma \ref{lem2.3}, there is no
odd deficient-perfect number of the form $n=3^{\alpha_{1}}5^{\alpha_{2}}p_{3}^{\alpha_{3}}p_{4}^{\alpha_{4}}$.

{\bf Subcase 1.2 } $p_{3}\in \{19, 23, 29\}$. If $D=15$, then $\alpha_{1}\geqslant 6$. If $p_{4}=31$, then
$$2=\frac{\sigma(n)}{n}+\frac{d}{n}>\frac{3^{7}-1}{2\cdot 3^{6}}\cdot \frac{5^{3}-1}{4\cdot 5^{2}}\cdot \frac{29^{3}-1}{28\cdot 29^{2}}\cdot \frac{31^{3}-1}{30\cdot 31^{2}}+\frac{1}{15}>2,$$
which is a contradiction. If $p_{4}>31$, then $\alpha_{2}\geqslant 6$ and
$$2=\frac{\sigma(n)}{n}+\frac{d}{n}>\frac{3^{7}-1}{2\cdot 3^{6}}\cdot \frac{5^{7}-1}{4\cdot 5^{6}}\cdot \frac{29^{3}-1}{28\cdot 29^{2}}+\frac{1}{15}>2,$$
which is also a contradiction. Thus $D>15$. By Lemma \ref{lem2.4}-Lemma \ref{lem2.6}, there is no
odd deficient-perfect number of the form $n=3^{\alpha_{1}}5^{\alpha_{2}}p_{3}^{\alpha_{3}}p_{4}^{\alpha_{4}}$.

{\bf Subcase 1.3 } $p_{3}\in\{31, 37, 41, 43\}$. By Lemma \ref{lem2.7}-Lemma \ref{lem2.10}, there is no
odd deficient-perfect number of the form $n=3^{\alpha_{1}}5^{\alpha_{2}}p_{3}^{\alpha_{3}}p_{4}^{\alpha_{4}}$.

{\bf Subcase 1.4 } $p_{3}\in \{47, 53, 59, 61\}$. If $D\geqslant 22$, then
$$2=\frac{\sigma(n)}{n}+\frac{d}{n}<\frac{3}{2}\cdot \frac{5}{4}\cdot \frac{47}{46}\cdot \frac{53}{52}+\frac{1}{22}<2,$$
which is clearly false. Thus $D=15, \alpha_{1}\geqslant 6$ and $\alpha_{2}\geqslant 6$.

If $p_{3}=47$, then $p_{4}\leqslant 109$. Otherwise, if $p_{4}\geqslant 113$, then
$$2=\frac{\sigma(n)}{n}+\frac{d}{n}<\frac{3}{2}\cdot \frac{5}{4}\cdot \frac{47}{46}\cdot \frac{113}{112}+\frac{1}{15}<2,$$
which is false. Since $\rm{ord}_{5}(3)=\rm{ord}_{5}(47)=4$, we have $p_{4}\leqslant 101$ and
$$2=\frac{\sigma(n)}{n}+\frac{d}{n}\geqslant\frac{3^{7}-1}{2\cdot 3^{6}}\cdot \frac{5^{7}-1}{4\cdot 5^{6}}\cdot \frac{47^{3}-1}{46\cdot 47^{2}}\cdot \frac{101^{3}-1}{100\cdot 101^{2}}+\frac{1}{15}>2,$$
which is impossible.

If $p_{3}=53$, then $p_{4}\leqslant 83$. Otherwise, if $p_{4}\geqslant 89$, then
$$2=\frac{\sigma(n)}{n}+\frac{d}{n}<\frac{3}{2}\cdot \frac{5}{4}\cdot \frac{53}{52}\cdot \frac{89}{88}+\frac{1}{15}<2,$$
which is false. Since $\rm{ord}_{5}(3)=\rm{ord}_{5}(53)=4$, we have $p_{4}\leqslant 71$ and
$$2=\frac{\sigma(n)}{n}+\frac{d}{n}\geqslant\frac{3^{7}-1}{2\cdot 3^{6}}\cdot \frac{5^{7}-1}{4\cdot 5^{6}}\cdot \frac{53^{3}-1}{52\cdot 53^{2}}\cdot \frac{71^{3}-1}{70\cdot 71^{2}}+\frac{1}{15}>2,$$
which is impossible.

If $p_{3}=59$, then $p_{4}\leqslant 73$. Otherwise, if $p_{4}\geqslant 79$, then
$$2=\frac{\sigma(n)}{n}+\frac{d}{n}<\frac{3}{2}\cdot \frac{5}{4}\cdot \frac{59}{58}\cdot \frac{79}{78}+\frac{1}{15}<2,$$
which is false. Since $\rm{ord}_{5}(3)=4, \rm{ord}_{5}(59)=2$, we have $p_{4}\leqslant 71$ and
$$2=\frac{\sigma(n)}{n}+\frac{d}{n}\geqslant\frac{3^{7}-1}{2\cdot 3^{6}}\cdot \frac{5^{7}-1}{4\cdot 5^{6}}\cdot \frac{59^{3}-1}{58\cdot 59^{2}}\cdot \frac{71^{3}-1}{70\cdot 71^{2}}+\frac{1}{15}>2,$$
which is impossible.

If $p_{3}=61$, then $p_{4}\leqslant 71$. Otherwise, if $p_{4}\geqslant 73$, then
$$2=\frac{\sigma(n)}{n}+\frac{d}{n}<\frac{3}{2}\cdot \frac{5}{4}\cdot \frac{61}{60}\cdot \frac{73}{72}+\frac{1}{15}<2,$$
which is clearly false. However,
$$2=\frac{\sigma(n)}{n}+\frac{d}{n}\geqslant\frac{3^{7}-1}{2\cdot 3^{6}}\cdot \frac{5^{7}-1}{4\cdot 5^{6}}\cdot \frac{61^{3}-1}{60\cdot 61^{2}}\cdot \frac{71^{3}-1}{70\cdot 71^{2}}+\frac{1}{15}>2,$$
which is impossible.

{\bf Subcase 1.5 } $67\leqslant p_{3}\leqslant 269$. By Lemma \ref{lem2.11}, there is no
odd deficient-perfect number of the form $n=3^{\alpha_{1}}5^{\alpha_{2}}p_{3}^{\alpha_{3}}p_{4}^{\alpha_{4}}$.

{\bf Case 2.} $p_{2}=7$. If $D=3$, then
$$2=\frac{\sigma(n)}{n}+\frac{d}{n}>\frac{3^{3}-1}{2\cdot 3^{2}}\cdot\frac{7^{3}-1}{6\cdot 7^{2}}+\frac{1}{3}>2,$$
which is clearly false. Thus $D\geqslant 7$. If $p_{3}\geqslant 37$, then
$$2=\frac{\sigma(n)}{n}+\frac{d}{n}<\frac{3}{2}\cdot\frac{7}{6}\cdot\frac{37}{36}\cdot \frac{41}{40}+\frac{1}{7}<2,$$
which is impossible. Thus $p_{3}\leqslant 31$. By Lemma 3.1-Lemma 3.7, we have $n=3^{2}\cdot 7^{2}\cdot 11^{2}\cdot 13^{2}$ with deficient divisor $d=3^{2}\cdot 7\cdot 13$.

{\bf Case 3.} $p_{2}\in\{11, 13\}$. By Lemma \ref{lem4.1}-Lemma \ref{lem4.2}, there is no
odd deficient-perfect number of the form $n=3^{\alpha_{1}}p_{2}^{\alpha_{2}}p_{3}^{\alpha_{3}}p_{4}^{\alpha_{4}}$.

{\bf Case 4.} $p_{2}=17$. If $D\geqslant 9$, then
$$2=\frac{\sigma(n)}{n}+\frac{d}{n}<\frac{3}{2}\cdot \frac{17}{16}\cdot \frac{19}{18}\cdot \frac{23}{22}+\frac{1}{9}<2,$$
which is absurd. Thus $D=3$. If $p_{3}\geqslant 47$, then
$$2=\frac{\sigma(n)}{n}+\frac{d}{n}<\frac{3}{2}\cdot \frac{17}{16}\cdot \frac{47}{46}\cdot \frac{53}{52}+\frac{1}{3}<2,$$
which is false. If $p_{3}=19$, then $\alpha_{1}\geqslant 6$ and
$$2=\frac{\sigma(n)}{n}+\frac{d}{n}\geqslant \frac{3^{7}-1}{2\cdot 3^{6}}\cdot \frac{17^{3}-1}{16\cdot 17^{2}}\cdot \frac{19^{3}-1}{18\cdot 19^{2}}+\frac{1}{3}>2,$$
which is impossible. Thus $p_{3}\in\{23, 29, 31, 37, 41, 43\}$. By (\ref{Eq1}), we have
$$\frac{3^{\alpha_{1}+1}-1}{2}\cdot \frac{17^{\alpha_{2}+1}-1}{16}\cdot \frac{p_{3}^{\alpha_{3}+1}-1}{p_{3}-1}\cdot \frac{p_{4}^{\alpha_{4}+1}-1}{p_{4}-1}=5\cdot 3^{\alpha_{1}-1}17^{\alpha_{2}}p_{3}^{\alpha_{3}}p_{4}^{\alpha_{4}}$$
and $\alpha_{1}\geqslant 6$. Since $\rm{ord}_{17}(3)=16$ and ${\rm{ord}}_{17}(p_{3})$ are all even, we have $p_{4}\equiv 1\pmod{17}$.

If $p_{3}=23$, then $p_{4}\leqslant 3517$. Otherwise, if $p_{4}\geqslant 3527$, then
$$2=\frac{\sigma(n)}{n}+\frac{d}{n}<\frac{3}{2}\cdot \frac{17}{16}\cdot \frac{23}{22}\cdot \frac{3527}{3526}+\frac{1}{3}<2,$$
which is false.  Noting that
$\rm{ord}_{5}(3)=\rm{ord}_{5}(17)=\rm{ord}_{5}(23)=4$, we have $p_{4}\equiv 1\pmod{5}$ and $\alpha_{4}\geqslant 4$. Noting that $\rm{ord}_{3}(17)=\rm{ord}_{3}(23)=2$, we have $p_{4}\equiv 1\pmod{3}$. Thus $p_{4}\in\{1021, 1531, 2551, 3061\}$. It follows that $\alpha_{2}\geqslant 12, \alpha_{2}\geqslant 4$ and $\alpha_{3}\geqslant 4$. However,
$$2=\frac{\sigma(n)}{n}+\frac{d}{n}\geqslant \frac{3^{13}-1}{2\cdot 3^{12}}\cdot \frac{17^{5}-1}{16\cdot 17^{4}}\cdot \frac{23^{5}-1}{22\cdot 23^{4}}\cdot \frac{3061^{5}-1}{3060\cdot 3061^{4}}+\frac{1}{3}>2,$$
which is impossible.

If $p_{3}\in\{29, 31, 37, 41, 43\}$, then $p_{4}\leqslant 103$. Otherwise, if $p_{4}\geqslant 107$, then
$$2=\frac{\sigma(n)}{n}+\frac{d}{n}<\frac{3}{2}\cdot \frac{17}{16}\cdot \frac{29}{28}\cdot \frac{107}{106}+\frac{1}{3}<2,$$
which is false. Thus $p_{4}=103$ and $p_{3}=29$. Since $\rm{ord}_{5}(3)=\rm{ord}_{5}(17)=\rm{ord}_{5}(103)=4$ and $\rm{ord}_{5}(29)=2$, we know that the case can not occur.

This completes the proof of Theorem 1.
\end{proof}

\end{document}